\documentclass{article}
\usepackage{graphicx}
\usepackage{a4wide}
\usepackage[T1]{fontenc}
\usepackage{a4wide}
\usepackage{lipsum}
\usepackage{mathtools}
\usepackage{amsthm}
\usepackage{amsfonts}
\usepackage{yfonts}
\usepackage{amssymb}
\usepackage{xcolor}
\usepackage{mathrsfs}
\usepackage{amsmath}
\usepackage{cite}
\usepackage{epsfig}
\usepackage{stmaryrd}
\usepackage{comment}
\usepackage[utf8]{inputenc} 
\usepackage[T1]{fontenc}

\usepackage[utf8]{inputenc}
\usepackage[utf8]{inputenc}
\usepackage{fullpage}
\usepackage{amsthm}
\usepackage[pdftex,colorlinks,backref=page,citecolor=blue]{hyperref}
\usepackage{hyperref} 
\usepackage{cleveref}
\usepackage{amsmath, amssymb, amsbsy}

\usepackage{makeidx}

\makeindex

\newcommand{\beq}[1]{\begin{equation}\label{#1}}
\newcommand{\enq}[0]{\end{equation}}

\newcommand{\eps}{\epsilon}

\newcommand{\gl}{\lambda}

\newcommand{\nin}[0]{\noindent}

\newcommand{\sub}[0]{\subseteq}

\newcommand{\E}{\mathbb{E}}

\newcommand{\ra}{\rightarrow}

\newcommand{\Bin}{{\rm{Bin}}}

\newcommand{\Int}{{\rm{Int}}}

\newcommand{\gO}{\Omega}

\newcommand{\cA}{\mathcal{A} }
\newcommand{\cB}{\mathcal{B} }
\newcommand{\cC}{\mathcal{C} }
\newcommand{\cD}{\mathcal{D} }
\newcommand{\cE}{\mathcal{E} }
\newcommand{\cF}{\mathcal{F} }
\newcommand{\cG}{\mathcal{G} }
\newcommand{\cH}{\mathcal{H} }
\newcommand{\cI}{\mathcal{I} }

\newcommand{\cL}{\mathcal{L} }

\newcommand{\cP}{\mathcal{P} }

\newcommand{\cU}{\mathcal{U} }
\newcommand{\cV}{\mathcal{V} }

\newcommand{\lam}{\lambda}
\newcommand{\Lam}{\Lambda}

\renewcommand{\P}{\mathbb{P}}

\newtheorem{theorem}{Theorem}[section]

\newtheorem{conjecture}[theorem]{Conjecture}

\newtheorem{lemma}[theorem]{Lemma}
\newtheorem*{claim*}{Claim}
\newtheorem{claim}[theorem]{Claim}
\newtheorem{remark}[theorem]{Remark}
\newtheorem*{remark*}{Remark}
\newtheorem{proposition}[theorem]{Proposition}
\newtheorem{defn}[theorem]{Definition}
\newtheorem{corollary}[theorem]{Corollary}
\newtheorem{fact}[theorem]{Fact}

\numberwithin{equation}{section}

\newenvironment{subproof}[1][\proofname]{
  
  \begin{proof}[#1]
}{
  \end{proof}
}

\newcommand{\Pp}{\mathcal{P}}

\Crefname{lemma}{Lemma}{Lemmas}
\Crefname{equation}{Equation}{Equations}

\title{On Dedekind's problem, a sparse version of Sperner's theorem, and antichains of a given size in the Boolean lattice}
\author{Matthew Jenssen\thanks{Department of Mathematics, King's College London. E-mail: {\tt matthew.jenssen@kcl.ac.uk}} \and Alexandru Malekshahian\thanks{Department of Mathematics, King's College London. E-mail: {\tt alexandru.malekshahian@kcl.ac.uk}} \and Jinyoung Park\thanks{Department of Mathematics, Courant Institute of Mathematical Sciences, NYU. E-mail: {\tt jinyoungpark@nyu.edu}}}

\date{}

\begin{document}

\maketitle
\begin{abstract}
Dedekind's problem, dating back to 1897, asks for the total number $\psi(n)$ of antichains contained in the Boolean lattice $B_n$ on $n$ elements. We study Dedekind's problem using a recently developed method based on the cluster expansion from statistical physics and as a result, obtain several new results on the number and typical structure of antichains in $B_n$. We obtain detailed estimates for both $\psi(n)$ and the number of antichains of size $\beta \binom{n}{\lfloor n/2 \rfloor}$ for any fixed $\beta>0$. We also establish  
a sparse version of Sperner's theorem: we determine the sharp threshold and scaling window for the property that almost every antichain of size $m$ is contained in a middle layer of $B_n$.
\end{abstract}

\section{Introduction}\label{introduction}
An \textit{antichain} of a poset $P=(X, \le)$ is a subset of $X$ in which any two elements are incomparable in the partial order $\le$. Dedekind's problem \cite{dedekind}, dating back to 1897, asks for the total number $\psi(n)$ of antichains contained in the $n$-dimensional Boolean lattice $B_n=(2^{[n]}, \subseteq)$ where $[n]:=\{1,2, \ldots, n\}$. Equivalently, $\psi(n)$ is the number of elements of the free distributive lattice on $n$ generators, or the number of monotone Boolean functions $f:\{0,1\}^n\to \{0,1\}$. The exact value of $\psi(n)$ is known only for $n\leq 9$: the case $n=9$ recently came within reach of modern computing power \cite{1dedekind9,2dedekind9}
 more than 30 years after the determination of $\psi(8)$ \cite{wiedemann1991computation}. 
 
 Over the past several decades, a large amount of research has gone into establishing bounds for $\psi(n)$ for large $n$. Dedekind's problem is closely related to the classical extremal problem of finding the \emph{largest} antichain in $B_n$. Sperner's theorem from 1928~\cite{sperner}, a cornerstone of extremal combinatorics, solves this problem:  the largest antichain in $B_n$ has size $\binom{n}{\lfloor n/2 \rfloor}$. Moreover, the only antichains to achieve this bound are the middle layer(s) of $B_n$: the families consisting of all sets of size $\lfloor n/2 \rfloor$ and of all sets of size $\lceil n/2 \rceil$, respectively. Taking all subsets of such an antichain gives the simple lower bound $\log_2 \psi(n)\geq \binom{n}{\lfloor n/2 \rfloor}$.

Kleitman~\cite{kleitman} -- following bounds of Gilbert~\cite{gilbert1954lattice}, Korobkov~\cite{korobkov1965monotone} and Hansel~\cite{hansel1966nombre} -- was the first to estimate the logarithm of $\psi(n)$ up to a $(1+o(1))$ factor, showing that $\log_2 \psi(n)=\left(1+O\left(\frac{\log n}{\sqrt{n}}\right)\right)\binom{n}{\lfloor n/2 \rfloor}.$ The error term was subsequently improved to $O\left(\frac{\log n}{n}\right)$ by Kleitman and Markowsky \cite{km}. A few years later Korshunov \cite{korshunov} determined $\psi(n)$ itself up to a $(1+o(1))$ factor, thereby `asymptotically solving' Dedekind's problem. A shorter proof was later given by Sapozhenko \cite{sapo}. Henceforth, we let $N=\binom{n}{\lfloor n/2 \rfloor}$.

\begin{theorem}[Korshunov \cite{korshunov}, Sapozhenko \cite{sapo}]
\label{sapo}
    As $n\rightarrow \infty$, for $n$ even we have
    \[ \psi(n)=(1+o(1))2^{N}\exp\left[\binom{n}{n/2+1}\left( 2^{-n/2}+\frac{n^2-2n-16}{32}\cdot2^{-n}\right) \right],  \]
    while for $n$ odd
     \begin{center}
     $\displaystyle\psi(n)=(2+o(1))2^{N}\times $ \\\vspace{0.7 em} $ \exp\left[\displaystyle\binom{n}{(n+1)/2} \left(2^{-(n+1)/2} + \frac{3n^2-19}{32}\cdot 2^{-(n+1)}\right)+\binom{n}{(n+3)/2}\left(\frac{1}{2}\cdot 2^{-(n+1)/2} + \frac{n+5}{8}\cdot 2^{-(n+1)}\right)\right].$
     \end{center}
\end{theorem}

Dedekind's problem remains a testbed for modern tools and techniques in combinatorics. Kahn~\cite{kahn2002entropy} showed that the Kleitman-Markowsky bound \cite{km} can be established using entropy techniques and Balogh, Treglown and Wagner~\cite{balogh2016applications} gave a short proof of the same bound with the container method (see also Pippenger~\cite{pippenger1999entropy} for a short proof of a slightly weaker bound using entropy). Recently, the generalization of Dedekind's problem to the \emph{hypergrid} $[t]^n$ for $t\geq 2$ has also attracted significant interest -- see e.g. \cite{PST23,FRT23}.

In this paper we approach Dedekind's problem using a recently developed method based on tools from statistical physics (see Section~\ref{secRelatedWork}). We obtain several new results on the number and typical structure of antichains in $B_n$, including:
\begin{enumerate}
\item improved asymptotics for Dedekind's problem  (Theorem~\ref{saporefined2});
\item precise asymptotics for the number of antichains of size $\beta N$ for any fixed $\beta>0$ (Theorem~\ref{thm5});
\item a `sparse version' of Sperner's theorem: we determine the sharp threshold and scaling window for the property that a typical antichain of size $m$ is contained in a middle layer of $B_n$ (Theorem~\ref{dualthm1}).  
\end{enumerate}

We now discuss each of these results in more detail, starting with our sparse analogue of Sperner's theorem. 
\begin{theorem}
\label{dualthm1}
  Let $m\in [N]$ such that 
  \[
  \frac{m}{N}=\frac{3}{4}-\frac{\ln(n)}{4n} + \frac{h(n)}{n}\, .
  \]
The proportion of antichains in $B_n$ of size $m$ that are contained in a middle layer tends to
\[
\left\{
    \begin{array}{ll}
    1,  &\text{if $h(n)\to\infty$};\\
     0,  &\text{if $h(n)\to -\infty$};\\
    \exp\left\{\frac{-e^{-2c}}{\sqrt{2\pi}}\left(3\cdot \mathbf 1_{\{\text{$n$ even}\}} + \frac{15}{4}\cdot\mathbf 1_{\{\text{$n$ odd}\}}\right)\right\},    &\text{if $h(n)\to c$, $c\in\mathbb{R}$.}
    \end{array}
    \right.
\]
\end{theorem}

Proving sparse analogues of classical extremal theorems has been a recent and fruitful trend in combinatorics. This line of research has led to several breakthroughs and the development of powerful and general combinatorial techniques such as the `transference principle' of Conlon and Gowers~\cite{ConG16} and Schacht~\cite{Sch16}, and the `hypergraph container method' of Balogh, Morris and Samotij~\cite{balogh2015independent} and Saxton and Thomason~\cite{saxton2015hypergraph}.

Theorem~\ref{dualthm1} is similar in spirit to the sparse analogue of Mantel's Theorem due to Osthus, Pr\"{o}mel, and Taraz~\cite{OstPT03} which determines the sharp threshold in $m$ for the property that almost every triangle-free graph on $n$ vertices and $m$ edges is bipartite. More generally, given an extremal problem that seeks to maximize some parameter $f(G)$ over all $G$ in some family of combinatorial objects $\mathcal{C}$, one can ask: what is the threshold in $m$ for the property that almost all $G\in \mathcal{C}$ with $f(G)=m$ have a similar structure to the solutions of the extremal problem? See e.g.~\cite{jenssen2023evolution,alon2014counting, alon2014refinement,balogh2016typical,luczak2000triangle, engelberg2021typical,morris2024asymmetric, luczak2000triangle} for further examples of the study of such questions.

Another way in which an extremal problem can be `sparsified' is to solve the extremal problem relative to some sparse random set. The literature on such problems is vast and so we give an example most relevant to our setting and refer the reader to the excellent survey~\cite{conlon2014combinatorial} for more. Consider the following sparse analogue of the Boolean lattice $B_n$: for $p\in[0,1]$, let $B_{n,p}$ denote the sublattice of $B_n$ obtained by choosing each element of $B_n$ independently with probability $p$. One can then ask for the typical structure of the largest antichain in $B_{n,p}$. This has been studied by several authors~\cite{BalK22, collares2016maximum, embed, kohayakawa2002width, BalMT14, osthus2000maximum, renyi1961random}. Most recently, Balogh and Krueger \cite{BalK22} showed that there is a sharp threshold at $p=3/4$ for the property that the largest antichain in $B_{n,p}$ is contained in a middle layer. Theorem~\ref{dualthm1} is a natural `dual' version of this result. 

Theorem~\ref{dualthm1} will in fact come as a corollary to a detailed understanding of the number of and structure of antichains of a fixed size in $B_n$. Let $\psi(n,m)$ denote the number of antichains in $B_n$ of size $m$. We note that $\psi(n,m)$ has another natural combinatorial interpretation: it is the number of monotone Boolean functions $f:\{0,1\}^n\to\{0,1\}$ with $m$ conjunctions in its minimal DNF representation (see e.g.~\cite{korshunov2002distribution}). Theorem~\ref{dualthm1} shows that if $m/N>3/4+\eps$ for some $\eps>0$, then $\psi(n,m)=(1+o(1))\binom{N}{m}$, and that this is no longer true if $m/N<3/4-\eps$. Our next theorem provides an asymptotic formula for $\psi(n,m)$ whenever $m/N$ is bounded away from $0$. 

\begin{theorem}
\label{thm5}
    There exist sequences of rational functions $R_j^{0}(n, \beta)$ and $R_j^1(n, \beta)$, $j\in\mathbb{N}$, such that 
    the following holds. If $\beta:=m/N$ is bounded away from $0$ then, taking $t=\lceil \log_{\frac{1}{1-\beta}}(4)\rceil$, we have
    \[
    \psi(n,m)=(1+o(1))\binom{N}{m}\exp \left[ N\sum _{j=1}^{t-1} R_j^{0}\left(n, \beta \right) \cdot \left(1-\beta\right)^{\frac{jn}{2}} \right]
    \]
    as $n\rightarrow \infty$ and $n$ is even, and
    \[
    \psi(n,m)=(2+o(1))\binom{N}{m}\exp \left[ N\sum _{j=1}^{t-1} R_j^{1}\left(n, \beta \right) \cdot \left(1-\beta\right)^{j\frac{n+1}{2}} \right]
    \]
    as $n\rightarrow \infty$ and $n$ is odd.
    Moreover, the coefficients of $R_j^{0}$ and $R_j^1$ can be computed in time $e^{O(j\log j)}.$
 \end{theorem}

\nin If $\beta=1$ then we understand $\log_{\frac{1}{1-\beta}}(4)$ to equal $0$ (\Cref{thm5} is of course trivial in this case) and  when $t\in\{0,1\}$ we understand the sums in \Cref{thm5} to equal zero. We highlight that for fixed $j$, $R_j^0$ and $R_j^1$ can be computed in finite time.
For example, one can readily compute (see Remarks~\ref{rem:R01} and~\ref{rem:R11} below)
\[ R_1^0(n,\beta) = \frac{2\beta n}{n+2}\]
and
\[ R_1^1(n, \beta)= \frac{\beta}{1-\beta}+\frac{n-1}{n+3}\beta\, .\]

Theorem~\ref{thm5} significantly generalises a result of Korshunov and Shmulevich~\cite{korshunov2002distribution} who obtained asymptotics for $\psi(n,m)$ for certain sequences $m=m(n)$ satisfying $m/N\to 1/2$. Asymptotics for $\log\psi(n,m)$ were obtained by Balogh, Mycroft and Treglown~\cite{BalMT14} in certain regimes where $m=o(N)$: they showed that if $t\in \mathbb{N}$ and $m=o(N/n^{t-1})$ and $m=\omega(N/n^t)$ then $\log \psi(n,m)=(1+o(1))\binom{N}{m}$. Understanding the typical structure of antichains of size $m$ in these regimes remains an interesting open problem -- see the discussion in Section~\ref{conclusion}. 

Finally, let us return to Dedekind's problem itself. Our next result determines $\psi(n)$ up to a multiplicative factor of $\left(1+2^{-(t n/2)}\right)$ for any fixed $t\geq 1$. 

\begin{theorem}
\label{saporefined2}
     There exist sequences of polynomials $P_j^0 (n),$ $P_j^1 (n)$ and $P_j^2(n)$, $j\in\mathbb N$, such that for any fixed $t\geq 1,$ as $n\rightarrow \infty$, for $n$ even 
    \[ \psi(n)=\left(1+o\left(2^{-tn/2} n^{O(t)}\right)\right) 2^{N} \exp \left[ \binom{n}{n/2+1} \sum_{j=1}^{t+1} P_j^0(n) \cdot 2^{-jn/2}\right], \]
    while for odd $n$
    \[ \psi(n)=\left(2+o\left(2^{-tn/2} n^{O(t)}\right)\right) 2^{N} \exp \left[ \sum_{j=1}^{t+1}\left(\binom{n}{(n+1)/2}P_j^1(n)+\binom{n}{(n+3)/2} P_j^2(n)\right) \cdot2^{-j(n+1)/2}\right]. \]
    Moreover, the coefficients of $P_j^0, P_j^1$ and $P_j^2$ can be computed in time $e^{O(j \log j)}$, and each $P_j^0, P_j^1, P_j^2$ has degree at most $2j$. 
\end{theorem}

 For $j\in\{1,2\}$, $P_j^0, P_j^1$ and $P_j^2$ can be read directly from Theorem \ref{sapo}. The algorithms we give to compute $P_j^0, P_j^1, P_j^2$ and also $R_j^0, R_j^1$ from~\Cref{thm5} are feasible to implement, and in the Appendix we present a few more values of the $P_j$'s and $R_j$'s obtained using computer code written by Mauricio Collares.

We also make progress towards a `dual version' of \cite[Conjecture 3]{BalK22}. We call $\binom{[n]}{k}$ the $k^{\text{th}}$ $\emph{layer}$ of $B_n$.
We say that $\binom{[n]}{k-1}, \binom{[n]}{k}, \binom{[n]}{k+1}$ are three \emph{central layers} if $k=\lfloor n/2 \rfloor$ or $k= \lceil n/2 \rceil$.

\begin{theorem}
\label{sparsesperner3}
    There exists an absolute constant $C$ such that, for $\frac{C\log^2 n}{\sqrt{n}}<\beta\leq 1,$ almost all antichains of size $\lfloor\beta N\rfloor$ in $B_n$ are contained in three central layers.
\end{theorem}

\nin We conjecture that the range of $\beta$ in the above result can be improved to $\beta=\tilde \Omega(1/n)$ (see~\Cref{conj1} and the surrounding discussion).

\subsection{Methods and related work}\label{secRelatedWork}
Here we give an overview of the proofs of the results from the previous section. The intuition for our approach comes from statistical physics.  We build on tools from~\cite{sapo, jenssen2023homomorphisms,JKP2,hypercube,struct, jenssen2023evolution}, and develop some new tools along the way.

Let $\cA=\cA_n$ denote the set of all antichains in $B_n$. 
Our starting point is to write a partition function for antichains: 
\[
Z(\lam)=\sum_{A\in \cA}\lam^{|A|}\, 
\]
along with the corresponding (Gibbs) measure $\mu_\lam$ on $\cA$ given by
\begin{align*}
\mu_\lam(A)=\frac{\lam^{|A|}}{Z(\lam)}\, .
\end{align*}
We note in particular, that $Z(1)$ is equal to $\psi(n)$, the total number of antichains in $B_n$ and $\mu_1$ is the uniform distribution on antichains. For us, the combinatorial relevance of studying $Z(\lam)$ and $\mu_\lam$ at general $\lam>0$ comes from the following identity: for $m\in\mathbb{N}$,
\begin{align}\label{eqLocalId}
    \mu_\lam(|A|=m)=\frac{\lam^m \psi(n,m)}{Z(\lam)}\, .
\end{align}
If $\lam$ is chosen carefully so that the expected size of an antichain sampled from $\mu_\lam$ is close to $m$, then we can hope to estimate $\mu_\lam(|A|=m)$ via a local limit theorem. By the identity~\eqref{eqLocalId}, estimating $\psi(n,m)$ then reduces to the problem of estimating the partition function $Z(\lam)$.  

 Classical statistical physics is largely concerned with understanding partition functions like $Z(\lam)$ and their associated Gibbs measures. From this perspective, a central motivating question is the following: how does the typical structure of a sample from $\mu_\lam$ change with $\lam$? As $\lam$ increases, the measure $\mu_\lam$ becomes more biased toward larger antichains, and so we might expect  samples to increasingly `correlate' with a maximum antichain (a middle layer). Many powerful methods have been developed to rigorously study such `emergence of structure' or `phase transition' phenomena. The tool most relevant to us is the cluster expansion~\cite{ruelle1963correlation,penrose1963convergence,brydges1984short,kp} (see Section~\ref{secPrelims}). Recently, the cluster expansion and related tools have been applied outside the context of classical statistical physics, to algorithmic and combinatorial problems of enumeration and sampling~\cite{barvinok2015computing,regts2015zero,patel2016deterministic,barvinok2017combinatorics,HelmuthAlgorithmic2,jenssen2023homomorphisms,JKP2,hypercube,struct, jenssen2023evolution, li2023number, balogh2021independent, jenssen2024sampling}. Particularly relevant to our current context are the papers of the first author and Perkins~\cite{hypercube}, and the first author, Perkins and Potukuchi~\cite{struct}: here the cluster expansion is applied to study in detail the number and typical structure of independent sets in the $d$-dimensional hypercube $Q_d$.
 
 In our context, we use the cluster expansion to understand the measure $\mu_\lam$ as a perturbation of the `ground state' measure which selects a middle layer uniformly at random and then takes a $\lam/(1+\lam)$-random subset of that layer. This allows for a very precise understanding of the measure $\mu_\lam$ and its associated partition function $Z(\lam)$. In particular, this heuristic shows that for a given value of $\lam$, the expected size of an antichain chosen from $\mu_\lam$ should be close to $\frac{\lam}{1+\lam}N$. Returning to~\eqref{eqLocalId}, we see that in order to estimate $\psi(n,m)$, where $m=\beta N$,  we should expect the suitable choice of $\lam$ to be close to $\frac{\beta}{1-\beta}$.

It is natural to rephrase the problem of counting antichains in terms of counting independent sets in a graph. Abusing notation slightly, we let $B_n$ denote the graph whose vertex set is $2^{[n]}$, where $u$ is adjacent to $v$ if and only if $u\subset v$ or $v\subset u$. An antichain in $B_n$ is then simply an independent set in this graph. 

The \emph{container method} has proven to be an essential tool in the study of independent sets in graphs and hypergraphs. In the graph setting, its roots can be traced back to work of Kleitman and Winston~\cite{kleitman1982number} and Sapozhenko~\cite{sapozhenko1987number}. These tools were generalized to the context of hypergraphs and developed into a powerful method in the seminal works of Balogh, Morris and Samotij~\cite{balogh2015independent} and Saxton and Thomason~\cite{saxton2015hypergraph}. The container method has also proven to be an essential tool in applications of the cluster expansion to problems in combinatorics~\cite{jenssen2023homomorphisms,JKP2,hypercube,struct, jenssen2023evolution, li2023number, balogh2021independent}. Very broadly speaking, the container method allows one to efficiently bound the probability of seeing large deviations from a ground state. This in turn allows one to prove the convergence of cluster expansions of partition functions which encode these perturbations.

In this paper we build on Sapozhenko's approach to Dedekind's problem~\cite{sapo}, another early application of the (graph) container method. Following a recent development of the authors \cite{lambda}, we prove a new container result (Theorem~\ref{MT}) for antichains in consecutive layers of the Boolean lattice. This container result will be precisely what we need to analyze antichains in detail via cluster expansion.

\subsection{Proof overview and structure of the remainder of the paper}
For $k\in [n]$, let $L_k$ denote the $k^{\text{th}}$ layer $\binom{[n]}{k}$ of $B_n$. For $S\subseteq [n]$ we let $L_S=\bigcup_{s\in S}L_s$. Let $C_n\coloneqq L_{[ \lfloor n/2 \rfloor-1, \lceil n/2 \rceil +1 ]}$, the union of the $3$ or $4$ (depending on the parity of $n$) most central layers of $B_n$. 
Our first step is to show that most of the contribution to $Z(\lam)$ comes from antichains contained in $C_n$.
Given $U\subseteq B_n$, let
\[
Z(U, \lam)= \sum_{A \in \cA, \ A \subseteq U}\lam^{|A|}\, ,
\]
along with the corresponding Gibbs measure
\begin{align}\label{eqLocalId0}
\mu_{U,\lam}(A)=\frac{\lam^{|A|}}{Z(U,\lam)}\, .
\end{align}
In particular, $Z(\lam)= Z(B_n, \lam)$ and $\mu_\lam=\mu_{B_n,\lam}$.

\begin{theorem}
\label{trunc}
There exists $C>0$ so that if $\lambda\geq C \log^2 n/\sqrt{n}$, then 
\[
Z(\lam) = \left(1+o\left(\frac{1}{\sqrt{N}} \right)\right)  Z( C_n, \lam)\, .
\]
\end{theorem}

In the special case where $\lam=1$, Theorem~\ref{trunc} was established by Sapozhenko~\cite{sapo} (with a weaker error term\footnote{A closer look at the proof shows that it gives  $Z(1)=(1+e^{-\Omega(\log ^2 n)})Z( C_n, 1).$}). Sapozhenko's approach was to show that the total number of antichains contained in three consecutive layers $L_{k-2}, L_{k-1}, L_{k}$ ($k\leq\lfloor n/2 \rfloor$) of $B_n$ is asymptotically equal to the number of antichains in the larger of the two layers. A careful iteration of this argument then yields the desired result. We follow the same broad strategy; however, significant complications arise when dealing with the case $\lam<1$. In particular, bounding the total weight of antichains containing an element from the smallest of the three layers becomes considerably more challenging. To deal with this, we prove a new container lemma for antichains contained in three consecutive layers of $B_n$ (Theorem~\ref{MT} below). We believe that Theorem~\ref{trunc} holds all the way down to $\lam=\tilde\Omega(1/n)$, but to prove this one would need a sharper container lemma. 

The explicit error term $1+o(1/\sqrt{N})$ above is necessary when applying our results to study antichains of a given size - see Lemma~\ref{transference} later. We note that a closer analysis of our proof of Theorem \ref{trunc} shows that we may easily replace this error term with the sharper $1+e^{-O(n\log n)}.$ In the case where $\lam=1$ we can prove something sharper still.

\begin{theorem}
\label{saporefined}
     As $n \rightarrow \infty$, we have
    \[ \psi(n)=\left( 1+e^{-\Omega(n^2)}\right) Z(C_n, 1).\]
\end{theorem}

We remark that this estimate is tight in the following sense: by counting antichains that contain a fixed element $v\in L_{\lceil n/2 \rceil +2}$ (so in particular $v\notin C_n$), one can show that \(\psi(n)=\left(1+e^{-O(n^2)}\right) Z(C_n, 1). \)

We establish some preliminaries in \Cref{secPrelims}. We prove Theorems \ref{trunc} and \ref{saporefined} in Section \ref{truncation}. 
With these theorems in hand, our next task is to estimate $Z(C_n, \lam)$ by appealing to the cluster expansion. This is the main goal of \Cref{secthreelayer}.
Following the approach used by the first author, Perkins and Potukuchi \cite{struct} for independent sets in the hypercube, we give asymptotics for the number of antichains of a given size in $C_n$ in \Cref{givensize}, and use Theorem \ref{trunc} to deduce information about antichains of a given size in $B_n$. We finalize the proofs of \Cref{dualthm1}-\ref{sparsesperner3} in \Cref{seceven} (for $n$ even) and \Cref{secodd} (for $n$ odd). In Section \ref{secproofcont} we prove our main container result that underpins our proof of the convergence of the cluster expansion. Finally, we conclude with some remarks and open problems in Section \ref{conclusion}. In the appendix we provide a few more values of the functions $P_j^r$ and $R_j^r$. These values were obtained by computer calculation, using code written by Mauricio Collares.  At the end of the paper, we provide an index for the various polymer models (\textit{cf.} \Cref{sec:polymer}) and measures used in the proofs for ease of reference.

\subsection{A note on asymptotic notation}
All asymptotic notation is to be understood with respect to the limit $n \to \infty$. All implicit constants in the asymptotic notation $O, \Omega$ etc. will be absolute constants unless specified otherwise.  For two functions $f,g:\mathbb{N}\to \mathbb{R}$ we write $f(n)\sim g(n)$ to denote that $\lim_{n \to \infty} f(n)/g(n)=1$, and $f(n)\ll g(n)$ to denote that $\lim_{n\to\infty}f(n)/g(n)=0$. We write $f(n)= \tilde O(g(n))$ if there exists $C>0$ such that $|f(n)|\leq (\log n)^C |g(n)|$ for $n$ sufficiently large. We use $\tilde \Omega$ analogously. 
We write that $f(n)=O(g(n))$ \emph{for $n$ even (odd)} if $f(2k)=O(g(2k))$  ($f(2k+1)=O(g(2k+1))$) as $k\to \infty$ respectively. We extend this convention to other asymptotic notation analogously.  

For two functions $f,g:\mathbb{N}\to \mathbb{R}$ we understand $f(n)\leq g(n)$ to mean that the inequality holds for $n$ sufficiently large.

We say a sequence of events $A_n$ holds `with high probability' (abbreviated `whp') if $\mathbb P(A_n) = 1-o(1)$.

\section{Preliminaries}\label{secPrelims}

We let $\mathbb N:=\{1, 2, \ldots\}$ and $\mathbb N_0:=\mathbb N \cup\{0\}$. Recall that we let $B_n$ denote the Boolean lattice on the ground set $[n]$. Abusing notation slightly, we also let $B_n$ denote the graph with vertex set $2^{[n]}$ where two sets are joined by an edge if and only if one is strictly contained in the other (the comparability graph of the poset $B_n$). Similarly if $U\subseteq B_n$, we identify $U$ with the subgraph of this comparability graph induced by the elements of $U$. We note that an antichain $A\subseteq B_n$ is precisely an \emph{independent set} in the comparability graph. For $U\subseteq B_n$, the measure $\mu_{U,\lam}$ introduced at~\eqref{eqLocalId0} can therefore be viewed as the \emph{hard-core measure at activity $\lam$} on independent sets in the graph $U$.  

Throughout this paper, we will use lowercase letters to denote vertices of $B_n$, uppercase letters to denote sets of vertices, and cursive letters to denote collections of sets. 

Given a graph $G$ and a vertex $v\in G$ we write $d(v)$ for the degree of $v$ and $N(v)$ for the neighbourhood of $v$. Given $U\subseteq V(G)$, we write $G[U]$ for the induced subgraph of $G$ on vertex set $U$ and $N_U(v)$ for $N(v)\cap U$. We define the \emph{square graph} $G^2$ to have the same vertex set as $G$ and setting $u$ and $v$ to be adjacent in $G^2$ if and only if the distance between $u$ and $v$ in $G$ is at most 2. We call $U\subseteq V(G)$ \emph{$2$-linked} if $G^2[U]$ is connected.

The following useful result is standard in the literature - see e.g. \cite[Lemma 2.1]{GalK04}. 
\begin{lemma}
    \label{2linked}
    Given a graph $G$ of maximum degree $\Delta$, the number of $2$-linked subsets of $G$ of size $t$ that contain a given vertex $v$ is at most $(e\Delta^2)^{t-1}$.
\end{lemma}

Throughout the paper, we write $\log$ for $\log_2$ and $\ln$ for $\log_e$.  We employ a common abuse of notation by often omitting floor and ceiling symbols for notational convenience.

\subsection{Polymer models and cluster expansion}\label{sec:polymer}
We will also require the notion of polymer model and cluster expansion from statistical physics which we introduce now.

Let $\mathcal{P}$ be a finite set, called the set of \emph{polymers}. Let $w: \mathcal{P} \rightarrow \mathbb{C}$ a weight function and $\sim$ a symmetric and antireflexive relation on $\mathcal{P}$ which we refer to as the \emph{compatibility} relation. We also say two polymers $\gamma_1$ and $\gamma_2$ are incompatible if $\gamma_1 \nsim \gamma_2$. We refer to the triple $(\mathcal{P}, \sim, w)$ as a \emph{polymer model}.

Let $\Omega=\Omega(\mathcal{P},\sim)$ be the collection of all sets consisting of pairwise compatible polymers, including the empty set of polymers. We define the \emph{partition function} $\Xi=\Xi(\cP, \sim, w)$ as

\begin{equation}\label{defpolypartfun}
    \Xi= \sum\limits_{\Lam \in \Omega} \prod\limits_{\gamma \in \Lam} w(\gamma),
\end{equation}
where by convention, we take the contribution from the empty set to be $1$.

Let $\Gamma=(\gamma_1, \ldots, \gamma_k)$ be an (ordered) tuple of polymers. The \emph{incompatibility graph} of $\Gamma$ is the graph $G_\Gamma$ on vertex set $[k]$ where $i$ is adjacent to $j$ if and only if $\gamma_i\nsim \gamma_j$.
We call $\Gamma$ a \emph{cluster} if $G_\Gamma$ is connected.
We define the \emph{size} of $\Gamma$ to be $\|\Gamma\|\coloneqq \sum_{i=1}^k |\gamma_i|$. We define the weight of a cluster $\Gamma$ to be 
\begin{align}\label{eq:wclusterdef}
    w(\Gamma)=\phi(G_\Gamma) \prod\limits_{{\gamma} \in \Gamma} w(\gamma),
\end{align}
where $\phi$ is the Ursell function of a graph $G=(V,E)$, defined as
\begin{align*}
    \phi(G):=\frac{1}{|V|!} \sum\limits_{\substack{A \subseteq E: \\ (V, A) \text{ connected}}} (-1)^{|A|}.
\end{align*}

Let $\mathcal{C}=\mathcal{C}(\mathcal{P},\sim)$ be the set of all clusters. Then the \emph{cluster expansion} of $\ln \Xi$ is the formal power series
\begin{equation}\label{eqclusterexp}
   \ln \Xi =\sum\limits_{\Gamma \in \mathcal{C}} w(\Gamma).
\end{equation}
This turns out to be the multivariate Taylor series of $\ln \Xi$ - see \cite{dobrushin, ScoS05}. A powerful tool that gives a sufficient condition for the convergence of the cluster expansion, as well as tail bounds, is the following:

\begin{theorem} [Koteck\'y and Preiss \cite{kp}]\label{thm.KP}
Let $f,$ $g: \Pp \rightarrow [0, \infty)$ be two functions. If for all ${\gamma} \in \Pp$ we have

\begin{equation}\label{KPbound1}
    \sum\limits_{{\gamma'} \nsim {\gamma}} |w({\gamma}')|e^{f({\gamma}')+g({\gamma}')} \leq f({\gamma}),
\end{equation}
then the cluster expansion~\eqref{eqclusterexp} converges absolutely.

Moreover, setting $g(\Gamma):= \sum_{{\gamma}\in \Gamma} g({\gamma})$ and writing $\Gamma \nsim {\gamma}$ if there exists some ${\gamma}' \in \Gamma$ such that $\gamma' \nsim {\gamma}$, we then have that, for all polymers ${\gamma}$,
\beq{kpbounds}    \sum\limits_{\substack{\Gamma \in \mathcal{C} \\ \Gamma \nsim {\gamma}}} |w(\Gamma)|e^{g(\Gamma)}\leq f({\gamma}). \enq

\end{theorem}

\subsection{Antichains of a given size}

We will make repeated use of the following general lemma which will be useful for dealing with antichains of a given size.  

\begin{lemma}[\textit{cf.} {\cite[Lemma 15]{struct}}]
    \label{transference}  Let $\lam=\lam(n)>0$.   Let $\cF_1$ and $\cF_2$ be families of independent sets in a graph $G=G_n$ with $\cF_1 \subseteq \cF_2$. For $i \in \{1, 2\}$, let $Z_i=\sum_{A\in \cF_i} \lambda^{|A|}$, and let $\mu_i$ be the probability measure given by
\[\mu_i(A)=\begin{cases} \lambda^{|A|}/Z_i & \text{ if $A \in \cF_i$,} \\ 0 & \text{otherwise.} \end{cases}\]
If $Z_2=(1+f(n))Z_1$ where $f(n)=o(1)$, then
\beq{TVD.bd} \left\| \mu_1-\mu_2 \right\|_{TV}\leq f(n)\, . \enq
    Moreover, with $i_m(\cF_i)$ the number of independent sets of size $m$ contained in $\cF_i$, if $\lambda$ and $m$ are such that
    \beq{lambda.cond} i_m(\cF_1)\lambda^m/Z_1\gg f(n),\enq
    then
    \[ i_m(\cF_2)=(1+o(1))i_m(\cF_1)\, .\]
\end{lemma}

\begin{proof} Recall that 
\[
\left\|\mu_1-\mu_2\right\|_{TV}=\sum_{A \subseteq V(G),\,  \mu_1(A)>\mu_2(A)} \mu_1(A)-\mu_2(A)\, . \]
We may assume that $\cF_1 \subsetneq \cF_2$, in which case $\mu_1(A)>\mu_2(A)$ if and only if $A \in \cF_1$. We therefore have
    \begin{equation*}
       \left\|\mu_1-\mu_2\right\|_{TV}=  \sum\limits_{A\in \cF_1} \frac{\lambda^{|A|}}{Z_1} - \frac{\lambda^{|A|}}{Z_2}=\frac{f(n)}{1+f(n)} \times\sum\limits_{A\in \cF_1} \frac{\lambda^{|\cA|}}{Z_1} \leq f(n)\, .
    \end{equation*}

    For the second part, let $\cI_i$ be the collection of independent sets of size $m$ contained in $\cF_i$. Note that $\mu_1(\cI_2)=\mu_1(\cI_1)$ (since $\mu_1(A)=0$ if $A \notin \cF_1$), therefore,
    \begin{align}\label{eq:mu1mu2f}
    |\mu_1(\cI_1)-\mu_2(\cI_2)|=|\mu_1(\cI_2)-\mu_2(\cI_2)| \le \|\mu_1-\mu_2\|_{TV}\stackrel{\eqref{TVD.bd}}{\leq}f(n)\, .
    \end{align}
On the other hand,
    \[ \left|\mu_1(\cI_1)-\mu_2(\cI_2)\right|=\left| \frac{i_m(\cF_1)\lambda^m}{Z_1} - \frac{i_m(\cF_2)\lambda^m}{Z_2} \right|=\frac{\gl^m}{Z_1}|i_m(\cF_1)-(1+o(1))i_m(\cF_2)|\, .\]
    The result now follows from~\eqref{lambda.cond} and \eqref{eq:mu1mu2f}.
\end{proof}

\subsection{Isoperimetry}
Recall that for $i\in [n]$ we let $L_i$ denote the $i^{\text{th}}$ layer $\binom{[n]}{i}$ of $B_n$. For $i\in \{0,\ldots,n-1\}$ and $v\in L_i$, we write $N^+(v)$ for the set of neighbours of $v$ in $B_n$ that belong to the layer $L_{i+1}$.
 We extend this notation to $N^+(S)$ for a set $S\subseteq B_n$ by setting $N^+(S):=\cup_{v \in S} N^+(v)$. For $i\in \{1, \dots n\}$ and $v\in L_i$ we similarly define $N^-(v)$ to be the set of neighbours of $v$ in $B_n$ that lie in $L_{i-1}$, and we extend this notation to $N^-(S)$ in the obvious way.

We will make frequent use of the following result that is an easy consequence of the classical Kruskal-Katona theorem~\cite{katona, kruskal}.

\begin{proposition}\label{isoper} If $i\leq \lceil n/2 \rceil$ and $S \sub  L_{i-1}$, then:

\begin{enumerate}
    \item if $|S|\leq n/10$, then $|N^+(S)|\geq n|S|/2-|S|^2$;
    \item if $|S|\leq n^4$, then $|N^+(S)|\geq n|S|/50$;
    \item furthermore, if $i\leq\lfloor n/2\rfloor$, then $|N^+(S)|\geq(1+1/n)|S|$; and
    \item if $i=\lceil n/2 \rceil$ and $|S|\leq |L_{\lfloor n/2 \rfloor}|/2,$ then $|N^+(S)|\geq (1+1/n)|S|$.
\end{enumerate}

\end{proposition}

\section{Reducing to central layers}
\label{truncation}
In this section we prove Theorem~\ref{trunc}, showing that the dominant contribution to the partition function $Z(\lam)$ comes from antichains contained in the central layers $C_n\subseteq B_n$.

 We prove Theorem~\ref{trunc} by iterating an argument which says that the weighted number of antichains in three consecutive layers $L_{[k-2,k]}$, $k\leq \lfloor n/2 \rfloor$, of $B_n$ is very close to the weighted number in $L_{[k-1,k]}$, the larger two of the three layers. Given $X, Y\subseteq B_n$, let $Y^X=\{v\in Y  \mid v \not\subset w \ \forall w\in X \}$.

\begin{lemma}
\label{truncstep1}
There exists $C>0$ such that if $\lambda\geq C\log^2 n/\sqrt{n}$ then the following holds.
    For any $2\leq k\leq \lfloor n/2 \rfloor$ and $X\subseteq  L_{[k+1, n]}$, we have 
    \beq{3to2}
    Z(L^{X}_{[k-2,k]}, \gl) = \left(1+o\left(\frac{1}{n\sqrt{N}} \right)\right)  Z(L^{X}_{[k-1,k]}, \gl)\, .
    \enq
\end{lemma}
We need the factor of $n$ in the denominator of the error term since we will iterate this lemma approximately $n$ times. Before turning to the proof of Lemma~\ref{truncstep1}, we show how it implies Theorem \ref{trunc}. This step also illuminates the need for the set $X$ in the statement of Lemma~\ref{truncstep1}. 

\begin{proof}[Proof of Theorem \ref{trunc} assuming Lemma~\ref{truncstep1}]
Let $1\le k \le \lfloor n/2 \rfloor-1$. Partitioning antichains according to their intersection with $L_{[k+2,n]}$ we see that
\begin{equation*}\begin{split}
    Z(L_{[k-1,n]},\lam)
    =\sum_{\substack{X\subseteq L_{[k+2,n]} \\ X\in \cA}} \lambda^{|X|} \cdot  Z(L^{X}_{[k-1,k+1]},\lam) &\stackrel{\eqref{3to2}}{=}\left(1+o\left(\frac{1}{n\sqrt{N}} \right)\right) \sum_{\substack{X\subseteq L_{[k+2,n]} \\ X\in \cA}} \lambda^{|X|} \cdot  Z(L^{X}_{[k,k+1]},\lam)   \\ &\ = \left(1+o\left(\frac{1}{n\sqrt{N}} \right)\right) Z(L_{[k, n]}, \lam)\, .
\end{split} \end{equation*}
Iterating the above we conclude that
\begin{equation*}
    Z(\lam)= \left(1+o\left(\frac{1}{n\sqrt{N}} \right)\right) ^{ \lfloor n/2 \rfloor-1}
   Z(L_{[\lfloor n/2 \rfloor-1, n]}, \lam)\, .
    \end{equation*}
Observe that by symmetry $Z(L_{[\lfloor n/2 \rfloor-1, n]}, \lam)=Z(L_{[0, \lceil n/2 \rceil+1]}, \lam)$, and again for any $1\le k \le \lfloor n/2 \rfloor-1$,
\begin{equation*}\begin{split}
    Z(L_{[k-1,\lceil n/2 \rceil +1]},\lam)
    &\ =\sum_{\substack{X\subseteq L_{[k+2,\lceil n/2 \rceil +1]} \\ X\in \cA}} \lambda^{|X|} \cdot  Z(L^{X}_{[k-1,k+1]},\lam) \\&\stackrel{\eqref{3to2}}{=}\left(1+o\left(\frac{1}{n\sqrt{N}} \right)\right) \sum_{\substack{X\subseteq L_{[k+2,\lceil n/2 \rceil +1]} \\ X\in \cA}} \lambda^{|X|} \cdot  Z(L^{X}_{[k,k+1]},\lam)   \\  &\ = \left(1+o\left(\frac{1}{n\sqrt{N}} \right)\right) Z(L_{[k, \lceil n/2 \rceil +1]}, \lam).
\end{split} \end{equation*}
Therefore,
\begin{align*}
    Z(\lam)
    = \left(1+o\left(\frac{1}{n\sqrt{N}} \right)\right) ^{\lfloor n/2 \rfloor -1}Z(L_{[0,\lceil n/2 \rceil +1]},\gl)
    &\le \left(1+o\left(\frac{1}{n\sqrt{N}} \right)\right)^{n}
Z(L_{[\lfloor n/2 \rfloor-1, \lceil n/2 \rceil +1 ]}, \lam)\\
&=\left(1+o\left(\frac{1}{\sqrt{N}} \right)\right)
Z(C_n, \lam)\, .
    \end{align*}
\end{proof}

In order to prove Lemma~\ref{truncstep1}, we shall make use of an `embedding trick' that appears in the work of Hamm and Kahn~\cite{embed} and later in the work of Balogh and Krueger~\cite{BalK22}. This trick allows us to reduce the proof of Lemma~\ref{truncstep1} to the special case $k= \lfloor n/2 \rfloor$ below:

\begin{lemma}
    \label{trunclemma1}
   There exists $C>0$ such that if $\lambda\geq C\log^2 n/\sqrt{n}$ then the following holds.
    For any $X\subseteq  L_{[\lfloor n/2 \rfloor+1, n]}$, we have 
 \[
    Z(L^{X}_{[\lfloor n/2 \rfloor-2,\lfloor n/2 \rfloor]},\lam) = \left(1+o\left(\frac{1}{n\sqrt{N}} \right)\right)  Z(L^{X}_{[\lfloor n/2 \rfloor-1,\lfloor n/2 \rfloor]},\lam)\, .
    \]

\end{lemma}

\begin{proof}[Proof of Lemma~\ref{truncstep1} assuming Lemma~\ref{trunclemma1}]
    For $k= \lfloor n/2 \rfloor$, this is precisely Lemma~\ref{trunclemma1}. Fix $k < \lfloor n/2 \rfloor $ and $X\subseteq L_{[k+1,n]}$.
    Given $w\in B_n$, let $w^{+}=w\cup[n+1,2n-2k]$ considered as an element of the larger lattice $B_{2n-2k}$.
    Given $S\subseteq B_n$, let $S^{+}=\{w^{+}: w\in S\}\subseteq B_{2n-2k}$. For $i\in [2n-2k]$, write $\overline{L}_i\coloneqq \binom{[2n-2k]}{i}$, the $i^\text{th}$ layer in $B_{2n-2k}$, and for $S\subseteq [2n-2k]$ let $\overline{L}_S=\cup_{i\in S}\overline{L}_i$.  Let $X'=X^{+} \cup \bigcup\limits_{j=n+1}^{2n-2k} ( [2n-2k]\backslash\{j\} )$.
    Observe that if $v\in \overline{L}_{[n-k-1,n-k]}^{X'}$ then $[n+1, 2n-2k]\subseteq v$ and $v=w^+$ for some $w\in  {L}_{[k-1,k]}^{X}$.
    Moreover, $A\subseteq L_{[k-1,k]}^{X}$ satisfies  $|A^+|=|A|$ and $A$ is an antichain if and only if $A^+\subseteq \overline{L}_{[n-k-1,n-k]}^{X'}$ is an antichain. It follows that $Z\left( \overline{L}_{[n-k-1,n-k]}^{X'},\lam\right)=Z\left( L_{[k-1,k]}^{X},\lam\right) $ . Similarly $Z\left( \overline{L}_{[n-k-2,n-k]}^{X'},\lam\right)=Z\left( L_{[k-2,k]}^{X},\lam\right) $. 
    Finally, by Lemma~\ref{trunclemma1},
    $Z\left( \overline{L}_{[n-k-2,n-k]}^{X'},\lam\right)=\left(1+o\left(\frac{1}{n\sqrt{N}} \right)\right)Z\left( \overline{L}_{[n-k-1,n-k]}^{X'},\lam\right)$.
\end{proof}

We will in fact prove a result a little more general than~\Cref{trunclemma1}. We state the result first and then explain the motivation for it. We first need a little notation.  

First let $r\in \{\lfloor n/2 \rfloor, \lceil n/2 \rceil\}$ and let $\cA_r$ denote the set of all antichains $S\subseteq B_n$ such that each $2$-linked component $A$, where the 2-linkedness is defined with respect to the comparability graph of $B_n$ induced on $L_{r-1} \cup L_r$, of $S\cap L_{r-1}$ satisfies $|N^{+}(A)|\geq (1+1/n)|A|$.

For $U\subseteq B_n$, let \index{$Z_r(U, \lam)$}
\[
Z_r(U, \lam)= \sum_{A \in \cA_r, \ A \subseteq U}\lam^{|A|}\, .
\] 

\begin{lemma}
    \label{trunclemma2}
   There exists $C>0$ such that if $\lambda\geq C\log^2 n/\sqrt{n}$ then the following holds. If  $r\in \{\lfloor n/2 \rfloor, \lceil n/2 \rceil\}$
    and $X\subseteq  L_{[r+1, n]}$, we have 
 \[
    Z_r(L^{X}_{[r-2,r]},\lam) = \left(1+o\left(\frac{1}{n\sqrt{N}} \right)\right)  Z_r(L^{X}_{[r-1,r]},\lam)\, .
    \]

\end{lemma}

We note that by Proposition~\ref{isoper}, if $r=\lfloor n/2 \rfloor$ then $|N^{+}(A)|\geq (1+1/n)|A|$ for all subsets $A\subseteq L_{r-1}$ and so $\cA_r=\cA$ and so $Z_r\equiv Z$. Thus, \Cref{trunclemma1} is precisely the $r=\lfloor n/2 \rfloor$ case of \Cref{trunclemma2}. 

We now attempt to motivate the $r=\lceil n/2 \rceil$ case of \Cref{trunclemma2}. When $n$ is odd and $r=\lceil n/2 \rceil$, \Cref{trunc} shows that $Z(\lam)$ is very close to $Z(C_n,\lam)$ where $C_n= L_{[ r-2, r +1 ]}$, the union of four central layers.
Roughly speaking, we will show that the dominant contribution to $Z(C_n,\lam)$ comes from antichains that are  `close' to a subset of \emph{one} of the middle layers $L_r$, $L_{r-1}$. We may therefore separate these contributions and estimate them individually. If $X\subseteq L_{r+1}$, then the expression $\lam^{|X|}Z_r(L^X_{[r-2,r]},\lam)$ captures the contribution from antichains $A\subseteq C_n$ that are close to a subset of $L_{r}$ and which satisfy $A\cap L_{r+1}=X$. As we will see, the expansion property in the definition of $\cA_r$ allows us to approximate   $Z_r(L^X_{[r-2,r]},\lam)$ by $Z_r(L^X_{[r-1,r]},\lam)$ in the spirit of \Cref{trunclemma1}.

\subsection{Overview of the proof of Lemma~\ref{trunclemma2}}\label{secOutlineTrunc}
For the remainder of this section, fix $r\in \{\lfloor n/2 \rfloor, \lceil n/2 \rceil\}$
    and $X\subseteq  L_{[r+1, n]}$. 

 Our strategy will be to interpret the two partition functions in the statement of Lemma~\ref{trunclemma2} as the partition functions of two related polymer models which we introduce now.

 Recall that for $i\in \{0,\ldots,n-1\}$ and $A\subseteq L_i$ we let $N^{+}(A)$ denote the set of neighbours of $A$ in the layer $L_{i+1}$. Let $\cP_X=\cP_{X,r}$ denote the set of all $2$-linked (again, in the graph  $L_{r-1} \cup L_r$)
 subsets $A\subseteq L^X_{r-1}$ such that $|N^{+}(A)|\geq (1+1/n)|A|$. We call $\cP_X$ the set of polymers. Note that if $A\subseteq L^X_{r-1}$, then  $N^+(A)\subseteq L^X_{r}$.
Two polymers $A_1$, $A_2$ are \textit{compatible}, written $A_1\sim A_2$,  if and only if the union $A_1 \cup A_2$ is \emph{not} $2$-linked. For a set $A\sub L^X_{r -1}$, we write $\Int(A)$ for $\{ v\in L^X_{r -2}: N^+(v)\subseteq A\}$. Define the weight of $A$ by 
\begin{align}\label{eqwDef}
w(A)=\lambda^{|A|} (1+\lambda)^{-|N^+(A)|}\sum_{B\subseteq \Int(A)} \lambda^{|B|-|N^+(B)|}\, .
\end{align}
The polymer model $(\cP_X, \sim, w)$\index{$(\cP_{X}, \sim, w)$ or $(\cP_{X, r}, \sim, w)$} comes with its associated partition function  $\Xi= \sum\limits_{\Lam \in \Omega} \prod\limits_{A \in \Lam} w(A)$ where, as before, $\Omega$ denotes the collection of all sets of pairwise compatible polymers. Recall that $\cA_r$ denotes the set of all antichains $S\subseteq B_n$ such that each $2$-linked component $A$ of $S\cap L_{r-1}$ satisfies $|N^{+}(A)|\geq (1+1/n)|A|$. The set $\Omega$ is therefore in one-to-one correspondence with the set $\{S\cap L^X_{r-1}: S\in \cA_r\}$.

The following lemma motivates the study of this polymer model in the context of Lemma~\ref{trunclemma2}.

\begin{lemma}\label{lem:polymer-antichain}
\[
(1+\lam)^{M}\cdot \Xi=  Z_r(L^X_{[r-2,r]},\lam)\, ,
\]
where $M=|L^X_{r}|$\, .
\end{lemma}
\begin{proof}
Given an antichain $S\subseteq L^X_{[r-2,r]}$, let 
\[
\rho(S)= (S\cap L^X_{r-1})  \cup N^+(S\cap L^X_{r-2})\, .
\]
We partition the collection of antichains $S$ in  $L^X_{[r-2,r]}$ according to the value of $\rho(S)$. Note that if $\rho(S)=A$, then $S\cap L^X_{r}$ can be any subset of $L^X_{r}\backslash N^+(A)$. Moreover, $S\cap L^X_{r-2}$ can be any subset of $\Int(A)$ and if $S\cap L^X_{r-2}=B$ then $S\cap L^X_{r-1}= A\backslash N^+(B)$. We conclude that 
\begin{align*}
Z_r(L^X_{[r-2,r]},\lam) 
&=
\sum_{\substack{A\subseteq L^X_{r-1}\\ A\in\cA_r}}\sum_{\substack{S\subseteq L^X_{[r-2,r]}\\S\in \mathcal A_r,\,  \rho(S)=A}} \lam^{|S|}\\
&=\sum_{\substack{A\subseteq L^X_{r-1}\\ A\in\cA_r}} (1+\lam)^{M-|N^{+}(A)|}\sum_{B\subseteq \Int(A)}\lam^{|B|}\cdot \lam^{|A|-|N^{+}(B)|}\\
&=
(1+\lam)^M\sum_{\substack{A\subseteq L^X_{r-1}\\ A\in\cA_r}} w(A)\, .
\end{align*}
The result follows by noting that if $A$ decomposes into the collection of mutually compatible polymers $\{A_1, \ldots, A_\ell\}$, then $w(A)=\prod_{i=1}^\ell w(A_i)$.
\end{proof}

We also consider a polymer model $(\cP_X,\sim,w')$\index{$(\cP_X,\sim,w')$, $\Xi'$} where $w'$ is the simplified weight \[w'(A)=\lambda^{|A|} (1+\lambda)^{-|N^+(A)|}.\] We let $\Xi'= \sum_{\Lam \in \Omega} \prod_{A \in \Lam} w'(A)$ the associated partition function and note that (a simplified version of) the proof of Lemma~\ref{lem:polymer-antichain} shows that
\[(1+\lambda)^M\cdot \Xi'= Z_r(L^{X}_{[r-1,r]},\lam)\, ,
\]
where again $M=|L^X_{r}|$.
Therefore, to prove Lemma~\ref{trunclemma2}, it suffices to show that
\beq{xi.comparison} \Xi/\Xi'=1+o(1/(n\sqrt N))\, .\enq
\begin{remark}
\label{remark}
For all polymers $A\in\cP_X$ we have
   $0<w'(A)\leq w(A)$. Therefore, for any cluster $\Gamma\in \cC(\cP_X,\sim)$ we have that $w(\Gamma)$ and $w'(\Gamma)$ (defined in~\eqref{eq:wclusterdef}) have the same sign and that $|w'(\Gamma)|\leq |w(\Gamma)|.$

\end{remark}

We will use the Koteck\'y-Preiss criterion (Theorem~\ref{thm.KP}) to prove that the cluster expansion of $\ln \Xi$ converges absolutely. This gives automatic convergence for the cluster expansion of $\ln \Xi'$ as well, by the above remark (since the summands in the cluster expansion of $\ln \Xi'$ are termwise dominated by those for $\ln \Xi$). Once we have established cluster expansion convergence, \eqref{xi.comparison} will follow quickly.

To verify the Koteck\'y-Preiss condition, we require a container lemma.

For $i\in [n]$ and a set $A\subseteq L_i$, we denote the \emph{upwards closure} of $A$ by $[A]\coloneqq \{ v\in L_i: N^+(v)\subseteq N^+(A)\}$.
For $2\leq r \leq n$ set
\begin{center}
$\cH_r(a,b,g,h):=$ \\ $\{(A,B)\in 2^{L_{r-1}} \times 2^{L_{r-2}}: A \text{ $2$-linked}, |[A]|=a, |[B]|=b, |N^+(A)|=g, |N^+(B)|=h, N^+(B) \subseteq A\}$
\end{center}
and let $t=g-a$ and $t'=h-b$.

We prove the following result in Section~\ref{secproofcont}.

\begin{theorem}\label{MT} There exist $C, \ C'>0$ such that the following holds: for $r\in \{\lfloor n/2 \rfloor, \lceil n/2 \rceil\}$, any $a\geq n^2,$ $g\geq (1+1/n)a, \ b\geq0, \ h\geq0$ and $\lambda \geq C \log^2 n/\sqrt{n}$, setting $\gamma =C'/\sqrt{n}$ we have
\beq{container.ineq}\sum_{(A,B) \in \cH_r(a,b,g,h)} \lambda^{|A|+|B|-h}\leq(1+\lambda)^{g-\gamma(t+ t')}.\enq
\end{theorem}

\subsection{Proof of Lemma~\ref{trunclemma2} assuming Theorem~\ref{MT}}
\label{KPsection}
Throughout this section we assume $\lam\geq C \log^2 n/\sqrt{n}$ where $C$ is a sufficiently large absolute constant. 
 Recall that we have fixed $r\in \{\lfloor n/2 \rfloor, \lceil n/2 \rceil\}$ and $X\subseteq L_{[r+1, n]}$ and recall the definition of the polymer model $(\cP_X,\sim,w)$ from the previous section. 
 
We use Theorem \ref{MT} to verify the Koteck\'y-Preiss condition \eqref{KPbound1} for this polymer model, that is, we show that there exist $f,$ $g: \cP_X \rightarrow [0, \infty)$ such that for all $A \in \cP_X$ we have

\begin{equation}\label{KPbound.r}
    \sum\limits_{A' \nsim A} w(A')e^{f(A')+g(A')} \leq f(A)\, .
\end{equation}

For $A\subseteq B_n$, set
\begin{equation}\label{eq:fgdef}
f(A)=|A|/n^3 \text{\,  and \, }
g(A)=\left\{
\begin{array}{cc}
     \left(\frac{n-2}{2}|A|-|A|^2\right)\ln(1+\lam) - 100|A|\ln n & |A|\leq n/10\\
     \frac{1}{100}n|A|\ln(1+\lam) & n/10<|A|\leq n^4 \\
     \frac{|A|}{n^2}\ln(1+\lam) & |A|> n^4 . 
\end{array} \right.
\end{equation}

Because of the restriction on $a$ in Theorem~\ref{MT}, we need to treat small polymers separately. We use the fact that small polymers have good expansion properties by Proposition~\ref{isoper}.

Let $v\in L_{r-1}$. We first bound the sum
\begin{align*}
    \sum\limits_{A\in \cP_X: A\ni v} w(A)e^{f(A)+g(A)}\, .
\end{align*}
It will be useful to observe that for $A\in \cP_X$ we have $\gl^{|A|+|B|-|N^+(B)|}\le \max\{1, \gl^{|A|}\}$ for all $B\sub \Int(A)$ and $\gl>0$; indeed, if $\lam\leq 1$ then $\lambda^{|A|+|B|-|N^+(B)|}\leq 1$, while if $\lam> 1$ then $\lambda^{|A|+|B|-|N^+(B)|}\leq \lam^{|A|}$. Therefore,
\beq{small.A.bd}\sum_{B \subseteq  \Int(A)} \lambda^{|A|+|B|-|N^+(B)|} \le \max\{1, \gl^{|A|}\} \cdot 2^{|\Int(A)|} \leq (1+\lambda)^{|A|} 2^{|\Int(A)|}\, .\enq

\nin \textbf{Case 1:} $|[A]|\leq n/10$.

For polymers $A\in\cP_X$ with $|A|\leq n/10$, we have $|N^+(A)|\geq \frac{n}{2}|A|-|A|^2$ and $| \Int(A)|\leq 50|A|/n$ by Proposition~\ref{isoper}. 
By~\eqref{small.A.bd}, we then have
\begin{align}
    \sum_{\substack{A \ni v \\ |[A]|\leq n/10}} w(A)e^{f(A)+g(A)} &\leq \sum\limits_{\substack{A\ni v \\ |A|\leq n/10}} (1+\lam)^{|A|} 2^{50|A|/n}(1+\lambda)^{-\frac{n}{2}|A|+|A|^2} e^{|A|}(1+\lam)^{\frac{n-2}{2}|A|-|A|^2} n^{-100|A|}\nonumber  \\ &\leq \sum\limits_{t=1}^{\infty} (e^2n^2  n^{-100}2^{50/n})^t \leq 1/(6n^5), \label{eq:case1}
\end{align}
where for the second inequality we used Lemma~\ref{2linked} to bound the number of $A$'s that contain $v$.

\nin \textbf{Case 2:} $n/10<|[A]|\leq n^4$. 

For polymers $A\in \cP_X$ with $|A|\leq n^4$, we have $|N^+(A)|\geq n|A|/50$ and $| \Int(A)|\leq 50|A|/n$ by Proposition~\ref{isoper}. 
Using the bound \eqref{small.A.bd}, we have
\begin{align}
    \sum\limits_{\substack{A \ni v \\ \frac{n}{10}<|[A]| \leq n^{4}}} w(A)e^{f(A)+g(A)}\nonumber & \leq \sum\limits_{A\ni v, |A|\leq n^4} (1+\lam)^{|A|} 2^{50|A|/n} (1+\lambda)^{-n|A|/50}(e(1+\lam)^{n/100})^{|A|} \\
    & \leq \sum\limits_{t=1}^{\infty} (e^2n^2(1+\lam)^{1-n/100} 2^{50/n})^t \leq 1/(6n^5), \label{eq:case2}
    \end{align}

where, again, the second inequality uses Lemma~\ref{2linked}, and the last inequality holds by our assumption  $\lam\geq C \log^2 n/\sqrt{n}$ (in fact we only need $\lam\geq C \log n/n$ here).

\nin \textbf{Case 3:} $|[A]|> n^4$. 

In this regime we use Theorem~\ref{MT}. Observe, first, that $\cP_X\subseteq \cP_\emptyset$, so that every such $A$ we consider here does indeed appear as the first entry of an element $(A, B)$ of $\cH_r(a,g, b, h)$ for some $b$ and $h$. With $\gamma=C'/\sqrt{n}$ as in Theorem \ref{MT}, set $x=e^{1/n^3}(1+\lambda)^{-\gamma/2n}.$ Note that $x\leq e^{1/n^3} (1+ \log n/ n)^{-\gamma/2n} \leq e^{1/n^3}e^{- \gamma \log n/ 4n^2}  \leq e^{-1/n^{2.5}} <1$. We have
\[\begin{split}&\sum\limits_{A\ni v, |[A]| \geq n^4} w(A)e^{f(A)+g(A)} \\ &\ \leq \sum\limits_{a\geq n^4} \sum\limits_{g\geq a(1+\frac1n)} \sum\limits_{b} \sum\limits_{h\geq b(1+\frac1n)} \sum\limits_{(A, B)\in \mathcal{H}_r(a, b, g, h)} e^{|A|/n^3} (1+\lam)^{|A|/n^2} \lambda^{|A|+|B|-|N^+(B)|} (1+\lambda)^{-|N^+(A)|} \\ &\stackrel{\eqref{container.ineq}}{\leq} \left( \sum\limits_{a\geq n^4} e^{a/n^3} (1+\lam)^{a/n^2} \times \sum\limits_{t\geq a/n} (1+\lambda)^{-\gamma t} \right) \times \left(\sum\limits_{b} \sum\limits_{t'\geq b/n} (1+\lambda)^{-\gamma t'} \right).\end{split}\]
The rightmost bracket above is at most
\begin{equation*}
    \sum\limits_{b \ge 0} (1+\lambda)^{-\gamma b/n}/ (1-(1+\lambda)^{-\gamma}) \leq (1-(1+\lambda)^{-\gamma /n})^{-2}\leq (1-x)^{-2}.
\end{equation*}
The leftmost bracket is at most
    \[\begin{split} \sum\limits_{a\geq n^4} e^{a/n^3} (1+\lam)^{a/n^2}(1+\lambda)^{-\gamma a/n}(1-x)^{-1} & \leq \sum\limits_{a\geq n^4} e^{a/n^3} (1+\lambda)^{-\gamma a/2n}(1-x)^{-1}\\
    &\le e^{n} (1+\lambda)^{-\gamma n^3/2}(1-x)^{-2}\leq (1+\lambda)^{-\gamma n^3/3}(1-x)^{-2}.\end{split}\]
We conclude that
\begin{equation}\label{eq:case3}
    \sum\limits_{A\ni v, |[A]| \geq n^4} w(A)e^{f(A)+g(A)} \le (1+\lambda)^{-\gamma n^3/3}(1-x)^{-4}.
\end{equation}
Noting that $(1-x)^{-1}\leq (1-e^{-n^{-2.5}})^{-1} \leq 2n^{2.5}$, and recalling that $\lambda \ge C\log^2n/\sqrt n$ and $\gamma=C'/\sqrt{n}$, we obtain that the right-hand side of~\eqref{eq:case3} is at most $1/(6n^5).$ Combining with~\eqref{eq:case1} and~\eqref{eq:case2} we conclude that
\begin{align}\label{eq:case1-3}
     \sum\limits_{A\ni v} w(A)e^{f(A)+g(A)} \leq 1/(2n^5)\, .
\end{align}

Now we verify \eqref{KPbound.r}. For any polymer $A$, set $N^2(A)\coloneqq N(N(A))$ and note that $|N^2(A)|\leq n^2|A|$. Using~\eqref{eq:case1-3}, we have
\begin{equation}
\label{KP}
    \sum\limits_{A'\nsim A} w(A')e^{f(A')+g(A')} \leq \sum\limits_{v\in N^2(A)} \sum\limits_{A' \ni v} w(A')e^{f(A')+g(A')} \leq n^2|A| /(2n^5) = f(A)/2\, .
\end{equation}
Therefore, by Theorem \ref{thm.KP} and Remark \ref{remark}, the cluster expansions of $\ln\Xi$ and $\ln\Xi'$ converge absolutely.

Finally, to show \eqref{xi.comparison},
observe that for a cluster $\Gamma\in \cC=\cC(\cP_X,\sim)$ with $\lVert\Gamma\rVert<n/2$, we have $w'(\Gamma)=w(\Gamma)$ since the interior of any polymer $A\in \Gamma$ is empty. Therefore, recalling that $w'(\Gamma)$ and $w(\Gamma)$ have the same sign and  $|w'(\Gamma)|\leq |w(\Gamma)|$, we have
\begin{equation*}
    \left|\ln \Xi - \ln \Xi'\right| \leq \sum_{\lVert\Gamma\rVert\geq n/2} |w(\Gamma)|.
\end{equation*}
Conclusion~\eqref{kpbounds} of Theorem \ref{thm.KP} tells us that for any $v\in L_{r-1}$,
\beq{kp32bound}
    \sum_{\substack{\Gamma \in \cC \\ \Gamma \nsim \{v\}}} |w(\Gamma)|e^{g(\Gamma)}\leq f(\{v\})=1/n^3.
\enq
Note that if $\lVert\Gamma\rVert \geq n/2$, then $g(\Gamma)\geq n\ln n$. Therefore, summing \eqref{kp32bound} over all $v\in L_{r-1}$ and restricting to summands with $\lVert\Gamma\rVert \geq n/2$,
\beq{kptail}
    \sum_{\lVert\Gamma\rVert \geq n/2} |w(\Gamma)| \leq e^{-n\ln n}\sum_{\lVert\Gamma\rVert \geq n/2} |w(\Gamma)|e^{g(\Gamma)} \stackrel{\eqref{kp32bound}}{\leq} e^{-n\ln n} \binom{n}{ r  -1}n^{-3}=o\left(\frac{1}{n\sqrt{N}}\right).
\enq
This yields \eqref{xi.comparison} and completes the proof of Lemma~\ref{trunclemma2} (assuming Theorem~\ref{MT}).

\subsection{Proof of Theorem \ref{saporefined}}\label{saporefpf}

The proof of Theorem \ref{saporefined} is essentially a refinement of the approximation in Theorem \ref{trunc} for $\lam=1$. As we did in the preceding section, we sum the inequality in \eqref{kp32bound} over all choices of $v$, and restrict to only the terms with $\| \Gamma \| \geq n/2.$ Note that for $\lam=1$, we now have $g(\Gamma)=\Omega(n^2)$ for all $\Gamma$ such that $\lVert\Gamma\rVert \geq n/2$. Therefore, instead of \eqref{kptail} we have

\[ \sum_{\| \Gamma\| \geq n/2} |w(\Gamma)| \leq e^{-\Omega(n^2)} \sum_{\| \Gamma\| \geq n/2} |w(\Gamma)|e^{g(\Gamma)} \stackrel{\eqref{kp32bound}}{\leq} e^{-\Omega(n^2)} 2^n/n^3 = e^{-\Omega(n^2)}, \]
from which we deduce that (\textit{cf. \eqref{xi.comparison}})
\beq{gl.1}
    Z_r(L^{X}_{[r-2,r]}, 1) = \left(1+e^{-\Omega(n^2)}\right)  Z_r(L^{X}_{[r-1,r]}, 1)
\enq
for $r\in \{\lfloor n/2 \rfloor, \lceil n/2 \rceil\}$ and $X\subseteq L_{[r+1, n]}$. Note that this is analogous to Lemma~\ref{trunclemma2}, with $\gl=1$ and the error term $O(1/(n\sqrt N))$ replaced with $e^{-\gO(n^2)}.$ Therefore, by following the proof of Theorem \ref{trunc} (using \eqref{gl.1} instead of Lemma~\ref{trunclemma1}) we obtain the desired estimate \[ \psi(n)=\left(1+e^{-\Omega(n^2)}\right) Z(C_n, 1)\, . \]

\section{Antichains in three central layers}\label{secthreelayer}
With Theorem~\ref{trunc} and Theorem~\ref{saporefined} in hand, we now aim to estimate $Z(C_n,\lam)$, the contribution to the partition function $Z(\lam)$ from the middle layers $C_n= L_{[ \lfloor n/2 \rfloor-1, \lceil n/2 \rceil +1 ]}$. Specialising to the case $\lam=1$ will then prove Theorem~\ref{saporefined2}. We will need to estimate $Z(C_n,\lam)$ for general $\lam$ for the proofs of Theorems~\ref{dualthm1},~\ref{thm5} and \ref{sparsesperner3}.

The set $C_n$ consists of either three or four consecutive layers of $B_n$ depending on whether $n$ is even or odd. In the (simpler) case where $n$ is even, we will see that a typical sample from $\mu_{C_n,\lam}$, the hard-core measure on the graph $C_n$ at activity $\lam$\index{$\mu_{C_n,\lam}$} (defined at~\eqref{eqLocalId0}), mostly intersects the middle layer $L_{n/2}$ and has a small number of `defect' elements in layers $L_{n/2-1}, L_{n/2+1}$. When $n$ is odd, we will see in Section~\ref{secodd} that a typical sample from $\mu_{C_n,\lam}$ either $(i)$ mostly intersects $ L_{\lfloor n/2 \rfloor}$ with a small number of defect elements in one layer above and below or $(ii)$ mostly intersects $ L_{\lceil n/2 \rceil}$ with a small number of defect elements in one layer above and below. 

In either case, a typical sample from $\mu_{C_n,\lam}$ is contained in three central layers and in this section we develop the tools needed to analyse such antichains. Again we use the language of polymer models.

Let $k=\lceil n/2 \rceil$ and let $C_n'=L_{[k-1,k+1]}$. We note that $C_n'=C_n$ when $n$ is even.

We define the graph $G_C$\index{$G_C$} (`C' for `central') as follows. Take $V(G_C)=L_{k-1}\cup L_{k+1}$ and set $\{x, y\}\in E(G_C)$ if and only if $x\subsetneq y$ or $y \subsetneq x$.

For a subset $A\subseteq V(G_C)$, define $\partial(A)\coloneqq\left\{ x\in  L_{k} : \exists v\in A, \ {x\sub v \text{ or } v\sub x} \right\}$, the `two-sided' shadow of the set $A$ onto the middle layer $ L_{k}$.

Motivated by the discussion above, we now define a polymer model, $\cP_C$, whose partition function will capture the contribution to $Z(C_n,\lam)$ from antichains that mostly lie in $L_k$. 

For $i \in \{k-1, k+1\}$, the polymers in $\cP_C$ are the sets $A \sub L_i$ which are i) 2-linked in the graph $L_i \cup L_k$; and ii) $|\partial(A)| \ge (1+1/n)|A|$. \footnote{Note that if $n$ is even or if $n$ is odd and $A\subseteq L_{k+1}$, then the condition on $\partial(A)$ is redundant.} We say that two polymers $A_1, A_2$ are compatible, written $A_1 \sim A_2$, if and only if $\partial(A_1) \cap \partial (A_2)=\emptyset$. For the polymers $A\in \cP_C$ define the weight
 \[w_C(A)=\lambda^{|A|}(1+\lambda)^{-|\partial(A)|}.\] Write $\Omega=\Omega(\cP_C,\sim)$ for the collection of all sets of compatible polymers, and
$\cC=\cC(\cP_C,\sim)$ for the set of all clusters. Let
$\Xi_C=\Xi(\cP_C,\sim,w_C)$\index{$(\cP_C,\sim,w_C)$, $\Xi_C$} be the partition function of this polymer model, that is, 
\[\Xi_C=\sum_{\Lam \in \Omega} \prod_{A \in \Lam}w_C(A)\, .\] 
 Recall that $N=\binom{n}{\lfloor n/2 \rfloor}$.
 In later sections we will show that when $n$ is even $Z(C_n, \lam)=(1+\lambda)^N\Xi_C$ (\Cref{lempolymeq}) and when $n$ is odd $Z(C_n, \lam)=\left(2+o\left(\frac{1}{\sqrt{N}}\right)\right)(1+\lambda)^N\Xi_C$ (\Cref{defapprox}). We dedicate the remainder of this section to the study of the polymer model $(\cP_C,\sim,w_C)$ and its partition function $\Xi_C$.

 We begin by showing that the cluster expansion of $\ln \Xi_C$ converges. It turns out that this is almost automatic from our work in the previous section.

\begin{lemma}\label{central.lem}
The polymer model $(\cP_C,\sim,w_C)$ satisfies the Koteck\'y-Preiss condition \eqref{KPbound1}, with $f$ and $g$ as in \eqref{eq:fgdef}. In particular, we have
\begin{equation}
    \label{tailboundscentral}
    \sum_{\Gamma\in \cC} \left|w_C(\Gamma)\right|e^{g(\Gamma)}\leq 2^n/n^3.
\end{equation}
\end{lemma}
\begin{proof}
  We claim that, for any $A \in \cP_C$,
    \beq{modifiedKP}\sum_{B\nsim A} w_C(B)e^{f(B)+g(B)} \le f(A)/2.\enq
   Recall from Section~\ref{secOutlineTrunc} that for $r\in\{\lfloor n/2 \rfloor, \lceil n/2 \rceil\}$, $\cP_{\emptyset,r}$ denotes the set of all $2$-linked
 subsets $A\subseteq L_{r-1}$ such that $|N^{+}(A)|\geq (1+1/n)|A|$. Recall also the definition of $w$ from~\eqref{eqwDef}. In Section \ref{KPsection} (\textit{cf.} \eqref{eq:case1-3}), we showed that for any $v \in L_{r-1},$ we have 
    \begin{align}\label{eqP0r}
    \sum_{\substack{A \in \cP_{\emptyset,r} \\ A \ni v}} w(A)e^{f(A)+g(A)} \le 1/(2n^5)\, .
    \end{align} 
If $r=\lceil n/2 \rceil$ then $\cP_{\emptyset,r}$ is the set of all polymers  $A\subseteq\cP_C$ such that $A\subseteq L_{k-1}$, and we have $w_C(A)\leq w(A)$ for all $A \in \cP_{\emptyset,r}$. Thus,
  \begin{align}\label{eqP0rtowC}
    \sum_{\substack{A \in \cP_C \\ A \ni v}} w_C(A)e^{f(A)+g(A)} \le 1/(2n^5)\, ,
\end{align}
 for all $v\in L_{k-1}$. Applying~\eqref{eqP0r} again, now with $r=\lfloor n/2 \rfloor$ (and noting the symmetry of the involution $B_n\to B_n$ that maps an element to its complement), it follows that ~\eqref{eqP0rtowC} holds also for $v\in L_{k+1}$.  Therefore for any $A\in\cP_C$,
    \[\sum_{B \nsim A}w_C(B)e^{f(B)+g(B)}\le \sum_{v \in N^2(A)}\sum_{B \ni v} w_C(B)e^{f(B)+g(B)} \le 2n^2|A|/(2n^5)= f(A),\]
    where we used the fact that $|N^2(A)|\leq n^2 |A|.$
    
    By conclusion~\eqref{kpbounds} of Theorem \ref{thm.KP} we have
    \[\sum_{\substack{ \Gamma \in \cC \\ \Gamma \nsim \{ v \}}} \left|w_C(\Gamma)\right|e^{g(\Gamma)} \le 1/n^3
    \]
    for any $v \in L_{k-1} \cup L_{k+1}$, (\textit{cf.} \eqref{kp32bound}), from which \eqref{tailboundscentral} follows.
\end{proof}

With this lemma in hand we prove the following estimate on $\Xi_C$. The proof follows those of {\cite[Theorem~8]{hypercube} and \cite[Lemma 16.1]{jenssen2023homomorphisms}}. Recall that we have set $k=\lceil n/2 \rceil$. For a polynomial $P(x,y)$, we write $\deg_x P$ for the degree of $P$ in $x$ and similarly for $y$.

\begin{theorem}
\label{clusterexpansion}
There exists $C>0$ such that if $\lam>C\log^2 n/\sqrt{n}$ then the following holds.
    There exist sequences of bivariate polynomials
    $S^0_j(n,\lam), S^1_j(n,\lam),S^2_j(n,\lam)$, $j\in \mathbb{N},$ such that for any fixed $t\geq 1$ we have
    \[ \Xi_C= \exp \left[\sum_{j=1}^{t-1}\binom{n}{k-1}S^0_j(n,\lam)(1+\lam)
^{-j(k+1)}+ O\left( \frac{2^n n^{100t}}{(1+\lam)^{\frac{n-2}{2}t-t^2}}\right)\right]\,  \]
if $n$ is even and 
 \[ \Xi_C= \exp \left[\sum_{j=1}^{t-1}\binom{n}{k-1}S^1_j(n,\lam)(1+\lam)
^{-j(k+1)}+  
\binom{n}{k+1}S^2_j(n,\lam)(1+\lam)
^{-j(k+1)} + O\left( \frac{2^n n^{100t}}{(1+\lam)^{\frac{n-2}{2}t-t^2}}\right)\right] \] 
if $n$ is odd.
    Moreover, for $r\in\{0,1,2\}$, $\deg_n S_j^r\leq 2(j-1)$ and $\deg_\lam S_j^r\leq 2j^2$, and the coefficients of $S_j^r$ can be computed in time $e^{O(j \log j)}.$
\end{theorem}
\begin{proof} 
We prove the result in the case of $n$ odd. The case of $n$ even is analogous (and simpler).   

We analyse the cluster expansion $\ln \Xi_C=\sum_{\Gamma\in\cC}w_C(\Gamma)$, which is convergent by~\Cref{central.lem}.

    Note that with the choice of function $g$ as in \eqref{eq:fgdef}, for fixed $t$ we have that $\| \Gamma\| \geq t$ implies $g(\Gamma)\geq \left( \frac{n-2}{2}t - t^2\right)\ln(1+\lambda) - 100t\ln n.$
    Restricting the sum in the LHS of \eqref{tailboundscentral} to clusters of size at least $t$, we then have
    \beq{tailbds} \sum_{{\Gamma\in \mathcal{C}: \|\Gamma\|\geq t}} \left|w_C(\Gamma) \right|  \leq \frac{2^n n^{100t}}{n^3 (1+\lam)^{\frac{n-2}{2}t-t^2}}.\enq

  We are therefore done if we can show that for $j\leq t-1$,
  \begin{align}\label{eqs0s1}
  \sum_{\Gamma\in\cC: \|\Gamma\|=j}w_C(\Gamma)= \binom{n}{k-1}S^1_j(n,\lam)(1+\lam)
^{-j(k+1)}+  
\binom{n}{k+1}S^2_j(n,\lam)(1+\lam)
^{-j(k+1)}
\end{align}
where the $S^1_j, S^2_j$ have the properties claimed in the statement of the theorem. 

 Set $v_1=[k-1] \in L_{k-1}$ and $v_2=[k+1]\in L_{k+1}$ regarded as a vertex of the graph $G_C$. We call $v_1$ or $v_2$ the \emph{root} vertex. For $r\in\{1, 2\}$ we say $i\in [n]$ is a \emph{$v_r$-active coordinate} for some $v\in G_C$ if $i\in v \triangle v_r$ (symmetric difference). 
 For $S\subseteq G_C$ we define the set of active coordinates of $S$ to be the union of the sets of active coordinates of the elements of $S$. For a cluster $\Gamma$, we define its set of active coordinates as the set of active coordinates of $V(\Gamma)\coloneqq\cup_{A\in \Gamma} A$.

For $j, \ell \in\mathbb{N}$ and $a_1, a_2\in \mathbb{N}_0$ we now let $\cG^1_{j, \ell, a_1, a_2}$ denote the set of all clusters $\Gamma$ containing $v_1$, with $\|\Gamma\|=j$, $|V(\Gamma)|=\ell$ and with the set of $v_1$-active coordinates equal to $[a_1] \cup [k, k+a_2-1]$. Define $\cG^2_{j, \ell, a_1, a_2}$ in the analogous way, where we instead require that the clusters $\Gamma$ contain $v_2$ and that $[a_1]\cup[k+2, k+a_2+1]$ is the set of $v_2$-active coordinates. Note that if $\Gamma$ is a cluster with $|V(\Gamma)|=\ell$ containing $v_r$, then $\Gamma$ has at most $2(\ell-1)$ $v_r$-active coordinates. Thus $\mathcal{G}^r_{j, \ell, a_1, a_2}=\emptyset$ unless $a_1+a_2\leq 2(\ell-1)$. Note also that $\|\Gamma\|\geq |V(\Gamma)|$, so  $\mathcal{G}^r_{j, \ell, a_1, a_2}=\emptyset$ unless $\ell\leq j$.

    Now by vertex transitivity within the parts $L_{k-1}, L_{k+1}$ of $G_C$ and symmetry of coordinates, we have
    \[\begin{split}
&\sum_{\substack{\Gamma\in \cC \\ \|\Gamma\|=j}} w_C(\Gamma)=\sum_{ \ell=1}^{j} \frac{1}{\ell} \times \\
& \sum_{\substack{(a_1, a_2)\in\mathbb{N}_0^2 \\ a_1+a_2\leq 2(l-1)}} \left( \binom{n}{k-1}\binom{k-1}{a_1} \binom{n-k+1}{a_2}\sum_{\Gamma\in \cG^1_{j, \ell, a_1, a_2}} \mkern-24mu w_C(\Gamma) + \binom{n}{k+1}\binom{k+1}{a_1}\binom{n-k-1}{a_2}\sum_{\Gamma\in \cG^2_{j, \ell, a_1, a_2}}\mkern-24mu w_C(\Gamma) \right).
    \end{split} \]
    
The factors $\binom{n}{k-1}, \binom{n}{k+1}$ come from choosing a root vertex $v_1\in L_{k-1}$ or root vertex $v_2\in L_{k+1}$. The $1/\ell$ comes from the fact that each cluster $\Gamma$ with $|V(\Gamma)|=\ell$ is counted by precisely $\ell$ choices of root vertex. The factors of $\binom{k-1}{a_1}\binom{n-k+1}{a_2}$ and $\binom{k+1}{a_1}\binom{n-k-1}{a_2}$ come from the choice of active coordinates relative to the root vertex.
 
{Given $\Gamma \in \cG^1_{j, \ell, a_1,a_2} \cup \cG^2_{j, \ell, a_1,a_2}$, we can write \[w_C(\Gamma)=\phi(G_\Gamma)\lam^j (1+\lam)^{-(k+1)j + \alpha(\Gamma)}\] with $\alpha(\Gamma)\coloneqq (k+1)j - \sum_{A\in \Gamma} |\partial(A)|$. Observe that $0\leq \alpha(\Gamma)\leq j(j+1)/2$: indeed, each $A\in \Gamma$ is a subset of either $L_{k-1}$ or $L_{k+1}$. Moreover, any pair $u,v\in L_{k-1}$ has at most one common neighbour in $\cU$ and similarly for pairs $u,v\in L_{k+1}$. It follows that $k|A| - |A|(|A|-1)/2\leq |\partial(A)| \leq (k+1)|A|$.}

 Since $|\cG^r_{j, \ell, a_1,a_2}|$ is independent of $n$ for $r\in\{1,2\}$, we have that
    \[\sum_{\Gamma\in \cG^r_{j, \ell, a_1,a_2}} w_{C}(\Gamma) (1+\lam)^{j(k+1)}=\sum_{\Gamma\in \cG^r_{j, \ell, a_1,a_2}} \phi(G_\Gamma) \lam^j (1+\lam)^{\alpha(\Gamma)}
    \]
    is independent of $n$ and is a polynomial in $\lam$ of degree at most $j+j(j+1)/2\leq 2j^2$. 
    It follows that we may choose $S^1_j, S^2_j$ in 
    \eqref{eqs0s1} such that, $\deg_n S_j^r\leq 2(j-1)$ and $\deg_\lam S_j^r\leq 2j^2$.

    It remains to show that the coefficients of $S_j^r(n,\lam)$ can be computed efficiently. We split this into two steps.
    \begin{claim}
        For $r\in\{1,2\}$ and $j, \ell, a_1,a_2\in \mathbb{N}$, the set $\cG^r_{j, \ell, a_1,a_2}$ has size $e^{O(j\log j)}$ and can be generated in time $e^{O(j\log j)}$. 
    \end{claim}
    \begin{subproof}
     We assume that $r=1$ (the proof for $r=2$ is identical). 
     Recall that we may assume $a_1+a_2\leq 2(\ell-1)$ and $\ell\leq j$ else $\cG^1_{j, \ell, a_1,a_2}=\emptyset$.
       For each $m\leq \ell$, we first generate the set $\mathcal{L}_m$ of all $2$-linked sets (in $G_C$) $A$ with $|A|=m$, $v_1\in A$ and whose set of active coordinates is a subset of $[a_1]\cup[k,k+a_2-1]$. We do this by iterating over $1\leq m\leq \ell$.
        
        For $m=1$ this is trivial. For $m\geq 2$, suppose we have already generated $\mathcal{L}_{m-1}$. For each $A\in \cL_{m-1}$ and $v\in A$, run through the list of all vertices $w$ at distance 2 from $v$ and with set of active coordinates contained in $[a_1]\cup[k,k+a_2-1]$. Then, for each such $w$ not in $A$, add $A\cup \{w\}$ to $\cL_m$. There are $m-1$ choices for $v$, and given a choice of $v$ there are at most $\binom{a_1+a_2}{2}\leq 2\ell^2$ choices of $w$. This shows that $|\cL_m|\leq 2m\ell^2|\cL_{m-1}|$ and that $\cL_m$ can be generated in time $O(m\ell^2|\cL_{m-1}|)$ given $\cL_{m-1}$. In particular, $|\cL_{\ell}|\leq 2^\ell \ell! \ell^{2\ell}=e^{O(j\log j)}$. 

        Now let $\cL_\ell^{a_1,a_2}$ denote the subset of $\cL_\ell$ consisting of all $2$-linked sets $A$ whose set of active coordinates is precisely $[a_1]\cup[k,k+a_2-1]$. We can generate $\cL_\ell^{a_1,a_2}$ in time $e^{O(j\log j)}$ by checking each element of $\cL_\ell$ in turn. 

        Finally, we generate $\cG^1_{j, \ell, a_1,a_2}$. For each $A\in\cL_\ell^{a_1,a_2}$ in turn, we will generate the list of all clusters $\Gamma$ with $\|\Gamma\|=j$ and $V(\Gamma)=A$. We do this in time $e^{O(j\log j)}$. Recall that a cluster $\Gamma$ is an ordered tuple of polymers $(A_1, \dots , A_t)$ with $\sum_{i=1}^t |A_i|=j$. Now, there are precisely $2^j$ ordered (integer) partitions of $j$ of the form $(b_1, \dots , b_t)$ with $\sum_{i=1}^t b_i=j$. For each element of $\binom{A}{b_1}\times \dots \times \binom{A}{b_t}$ (a set of size at most $e^{O(j\log j)}$), we may check if it constitutes a cluster.

        This shows that $|\cG^1_{j, \ell, a_1,a_2}|=e^{O(j\log j)}$ and that we may generate the whole list $\cG^1_{j, \ell, a_1,a_2}$ in time $e^{O(j\log j)}$.
    \end{subproof}
    \begin{claim}
        For $r\in\{1,2\}$ and $\Gamma\in\cG^r_{j, \ell, a_1,a_2}$, we can calculate $w_C(\Gamma)$ in time $e^{O(j)}$.
    \end{claim}
    \begin{subproof}
        Given some $\Gamma\in\cG^r_{j, \ell, a_1,a_2}$ and $A\in\Gamma$, we can calculate $|\partial(A)|$ in time $O(j^2)$.
        Moreover, the Ursell function $\phi(G_{\Gamma})$ is an evaluation of the Tutte polynomial of $G_{\Gamma}$, and thus by a result of Bj\"orklund, Husfeldt, Kaski and Koivisto \cite[Theorem 1]{ursellcomp} it may be computed in time $e^{O(|V(G_\Gamma)|)}=e^{O(j)}$. Thus we can compute $w(\Gamma)$ in time $e^{O(j)}$.
    \end{subproof}
    The above two claims conclude the proof of \Cref{clusterexpansion}.
\end{proof}

\section{Antichains in three central layers of a given size}
\label{givensize}
In the previous section we saw how to obtain detailed asymptotics for the partition function $\Xi_C$. As we will see in later sections (Lemmas~\ref{lempolymeq} and \ref{defapprox}), this will translate to detailed asymptotics for $Z(C_n,\lam)$.

For the proofs of Theorems~\ref{dualthm1}, \ref{thm5}, and \ref{sparsesperner3}, we will use \Cref{transference} to obtain information about antichains in $B_n$ given our knowledge of antichains in $C_n$. Let $i_m(C_n)$ denote the number of antichains of size $m$ in $C_n$. Our next aim is twofold. First, we wish to find a suitable value of $\lam$ such that $i_m(C_n) \lam^m/Z(C_n, \lam)=\Omega(1/\sqrt{N})$ (cf. \eqref{lambda.cond} and \Cref{trunc}). Secondly, we wish to estimate $i_m(C_n)$ itself. It turns out that the first of these objectives will in fact be just one step on the way to establishing the second one, so we will focus on the latter. We achieve this by again studying the polymer model $(\cP_C, \sim, w_C)$. Let $\nu$ \index{$\nu$ on $\Omega(\cP_C,\sim)$} denote the measure on $\Omega=\Omega(\cP_C,\sim)$ given by
\beq{nu} \nu(\Lam)=\frac{\prod_{A\in \Lam}w_C(A)}{\Xi_C}\, . \enq

Recall that we let $k=\lceil n/2 \rceil$. Let $\mu^C_\lam$\index{$\mu^C_\lam$} denote the measure on antichains in $C_n'=L_{[k-1,k+1]}$ defined by the following process.
\begin{enumerate}
\item Select $\Lambda$ according to $\nu$ and let $\bar \Lambda=\bigcup_{A\in \Lambda}A$.\label{nu1}
\item Let $U$ be a $\frac{\lam}{1+\lam}$-random subset of $L_k\backslash \partial(\bar\Lambda)$.
\item Output the antichain $S=\bar\Lambda \cup U$.
\end{enumerate}
We let $i_m^C$ denote the number of antichains $A$ of size $m$ such that $\mu^C_\lam(A)>0$ (note that $i_m^C$ is independent of $\lam$). We will see, in \Cref{lempolymeq}, that when $n$ is even $\mu^C_\lam$ is simply the hard-core model on $C_n$ at activity $\lam$ (i.e.\  $\mu_{C_n,\lam}$ defined at~\eqref{eqLocalId0}) and so $i_m(C_n)=i_m^C$. When $n$ is odd we will show in \Cref{defapprox} that $i_m(C_n)=(2+o(1))i_m^C$. 

In the remainder of this section we develop the tools needed to estimate $i_m^C$. The main aim of this section is to prove the following.

\begin{theorem}
\label{thm5pre}
For $r\in\{0,1\}$,
  there exists a sequence of rational functions $R^r_j(n, \beta)$, $j\in \mathbb{N},$ such that the following holds. Given $\eps>0,\ t\in\mathbb{N}$ and $m\in[N]$ such that
$\beta\coloneqq \frac{m}{N}\in [1-4^{-1/t} +\eps,1]$ we have
    \[
    i_m^C=(1+o(1))\binom{N}{m}\exp \left[ N\sum _{j=1}^{t-1} R^r_j\left(n, \beta \right) \cdot \left(1-\beta\right)^{jk}\,\right]
    \]
    if $n\equiv r \pmod 2$.
    Moreover, the coefficients of $R^r_j$ can be computed in time $e^{O(j\log j)}.$
 \end{theorem}

At a high level, we follow the approach taken in \cite{struct} to enumerate independent sets of a fixed size in the hypercube. 

Throughout this section we assume that $\lam\geq C\log^2 n/ \sqrt{n}$ for $C$ sufficiently large. We also assume that $\lam=O(1)$; for concreteness we will take $\lam\leq 10$. We will later treat the easy case where $\lam$ is unbounded separately (in Sections \ref{closeto1} and \ref{closeto1'}), using a more direct argument.

 Let $\mathbf{A}$ denote an antichain sampled according to $\mu^C_\lam$. Note that by the definition of $\mu^C_\lam$ we have 
\begin{align}\label{eqPAm}
\mathbb{P}(|\mathbf{A}|=m)=\frac{ i_m^C \lam^m}{(1+\lam)^N \Xi_C}\, .
\end{align}
 Given our detailed knowledge of $\Xi_C$ obtained from the previous section, it therefore suffices to estimate $\mathbb{P}(|\mathbf{A}|=m)$ in order to estimate $i_m^C$. We are at liberty to choose $\lam$ for this task (provided $\lam\geq C\log^2 n/ \sqrt{n}$ and $\lam=O(1)$) and so we choose $\lam$ so that the expectation of $\mathbf{A}$ is close to $m$. Our strategy for estimating $\mathbb{P}(|\mathbf{A}|=m)$ is then as follows. Let $\mathbf\Lam$ be sampled according to $\nu$ (as in Step~\ref{nu1} in the definition of $\mu^C_\lam$). First we show that  $\mathbf{\bar\Lam}$ is, with very high probability, `small' in some appropriate sense (see Lemma~\ref{largedeviation}). We then show that $|\mathbf{\bar\Lam}|$ and $|\partial(\mathbf{\bar\Lam})|$ are concentrated around their (small) means (Lemma~\ref{clt}). We may therefore think of $|\mathbf{A}|$ as a random variable that is very close to $\text{Bin}\left( N, \frac{\lam}{1+\lam}\right),$ and by the Local Central Limit Theorem we may (asymptotically) write down the probability that $|\mathbf{A}|=m.$

We will encounter several technicalities along the way. For instance, in order to prove Theorem \ref{sparsesperner3} we may (and will) choose $\lam$ so that $\mathbb{E}[|\mathbf{A}|]=m,$ but this value of $\lam$ is not explicit (its existence simply follows by continuity of the expectation). For our proof of Theorem \ref{thm5}, we instead wish to find a value of $\lam$ that can be computed efficiently. We therefore have to relax our condition to just $\left| \mathbb{E}[|\mathbf{A}|] -m \right| = o(\sqrt{N})$; a suitable $\lam$ for which this holds is given by Lemma~\ref{lem17} below.

It will be useful to consider the modified polymer model $(\cP_C,\sim,\tilde w_C)$\index{$(\cP_C,\sim,\tilde w_C)$} where $\tilde w_C(A)\coloneqq w_C(A)e^{|A|/n^3}$. Let $\tilde f(A)\coloneqq |A|/(2n^3)$ and let $g$ be as in~\eqref{eq:fgdef}. We observe that the Koteck\'y-Preiss condition \eqref{KPbound1} for the polymer model $(\cP_C,\sim,\tilde w_C)$ with the functions $\tilde f,$ $g: \cP_C \rightarrow [0, \infty)$  becomes precisely the inequality \eqref{modifiedKP} that we have already established. Recall that we let $\cC=\cC(\cP_C,\sim)$ denote the set of clusters for this polymer model, and we write $g(\Gamma)= \sum_{{\gamma}\in \Gamma} g({\gamma})$.

\begin{lemma}\label{lem:tildetail}
Let $\alpha\in[0,1]$, $M>0$ and let $\cC'\subseteq \cC$. If $g(\Gamma)\geq M$ for all $\Gamma \in \cC'$ then  \[
\sum_{\Gamma\in \cC'}|\tilde w_C(\Gamma)|e^{\alpha g(\Gamma)} \leq 2^n e^{-(1-\alpha)M}\, .
\] 
\end{lemma}
\begin{proof}
By Theorem~\ref{thm.KP}
\[
e^{(1-\alpha)M}\sum_{\substack{\Gamma \in \cC' \\ \Gamma \nsim \{v\}}} |\tilde w_C(\Gamma)|e^{\alpha g(\Gamma)}\le \sum_{\substack{\Gamma \in \cC \\ \Gamma \nsim \{v\}}}| \tilde w_C(\Gamma)|e^{g(\Gamma)} \stackrel{\eqref{kpbounds}}{\le} \tilde f(\{v\})\leq 1\, .
\]
Summing this inequality over all $v\in L_{k-1}\cup L_{k+1}\subseteq B_n$ completes the proof.
\end{proof}

\begin{lemma}
    \label{largedeviation}
    If $ \mathbf{\Lam}$ is sampled according to $\nu$, then
 \[ | \mathbf{\bar\Lam} | \leq N/n^2,\]
    with probability at least $1-\exp\{\Omega(N/n^5)\}$.
   
\end{lemma}
\begin{proof}
    By Markov's inequality, we have $ \mathbb{P}\left( | \mathbf{\bar\Lam} | >N/n^2\right) \leq e^{-N/n^5} \mathbb{E}\left[ e^{|\mathbf{\bar\Lam}|/n^3} \right]$.

    Let $\tilde \Xi_C= \Xi(\cP_C,\sim,\tilde w_C)$ be the polymer model partition function, and note that \[ \mathbb{E}\left[e^{|\mathbf{\Lam}| /n^3}\right]= \tilde\Xi_C/ \Xi_C \leq  \tilde\Xi_C\, .  \]

Observing that $g(\Gamma)\geq \frac{n}{3}\ln(1+\lam)$ for all $\Gamma\in\cC$, {we have}, by \Cref{lem:tildetail} (with $\alpha=0$),
    \[ \ln \tilde\Xi_C = \sum_{\Gamma\in\cC} \tilde w_C(\Gamma) \leq \frac{2^n}{ (1+\lam)^{n/3}}=o\left( \frac{N}{n^5}\right),  \]
    which yields
    \[\mathbb{P}\left( | \mathbf{\bar\Lam} | >N/n^2\right) \leq e^{-N/n^5} \tilde \Xi_C=e^{-\Omega(N/n^5)}.\]
\end{proof}

For a cluster $\Gamma$, recall we write $\|\Gamma\|$ for $\sum_{\gamma\in\Gamma} |\gamma|$. Let us also write $\|\partial(\Gamma)\|$ for $\sum_{\gamma\in\Gamma}|\partial(\gamma)|$, so that we have $w_C(\Gamma)=\phi(G_\Gamma)\lam^{\|\Gamma\|} (1+\lam)^{-\|\partial(\Gamma)\|}.$

\begin{lemma}
    \label{clbounds}
    For any $t, \ell\geq 1$ fixed, we have
    \beq{cmbound}\sum_{\substack{\Gamma\in \cC \\ \|\Gamma\|\geq t}} |w_C(\Gamma)|\|\Gamma\|^{\ell} =O\left( \frac{2^n n^{100t+ \ell}}{(1+\lam)^{\frac{n}{2}t}}\right).\enq
\end{lemma}
\begin{proof}
First, for $t\leq \| \Gamma \| \leq n/10,$ note that $g(\Gamma)\geq \ln(1+\lam) \left(\frac{n-2}{2}t-t^2\right) -100t\ln n$. By Lemma~\ref{lem:tildetail} (with $\alpha=0$), noting that $|w_C(\Gamma)|\leq|\tilde w_C(\Gamma)|$ for all $\Gamma\in\cC$, we have  
    \[ 
    \sum_{\substack{\Gamma\in\cC: \\ t\leq \| \Gamma \| \leq n/10}} |w_C(\Gamma)| \|\Gamma\|^{\ell} \leq \sum_{\substack{\Gamma\in\cC: \\ t\leq \| \Gamma \| \leq n/10}} |w_C(\Gamma)| n^{\ell} \leq \frac{2^n n^{100t+\ell}}{(1+\lam)^{\frac{n-2}{2}t-t^2}}=O\left( \frac{2^n n^{100t+ \ell}}{(1+\lam)^{\frac{n}{2}t}}\right)
    \] 
    (recall that $\lam\leq 10$).
    For $n/10<\| \Gamma \| \leq n^4,$ we have $g(\Gamma)\geq \frac{n\|\Gamma\|}{100} \ln(1+\lam)\geq \frac{n^2}{1000}\ln(1+\lam),$ and so $\|\Gamma\|^{\ell}\leq n^{4\ell}\leq e^{g(\Gamma)/2}$. Thus by Lemma~\ref{lem:tildetail} (with $\alpha=1/2$),
    \[ 
    \sum_{\substack{\Gamma\in\cC \\ \frac{n}{10}< \| \Gamma\| \leq n^4}} |w_C(\Gamma)| \|\Gamma \|^{\ell} \le \sum_{\substack{\Gamma\in\cC \\ \frac{n}{10}< \| \Gamma\| \leq n^4}} |w_C(\Gamma)| e^{g(\Gamma)/2} \leq \frac{2^n}{(1+\lam)^{\frac{n^2}{2000}}}.
    \]  
    
    Finally, for $\| \Gamma \| >n^4,$ observe that $g(\Gamma)\geq \frac{\|\Gamma\|}{n^2} \ln(1+\lam)\geq n^2\ln(1+\lam)$, and so $\|\Gamma\|^{\ell}\leq e^{g(\Gamma)/2}$. Therefore, as above, 
    \[\sum_{\substack{\Gamma\in\cC \\ \| \Gamma\| >n^4}} |w_C(\Gamma)| \|\Gamma \|^{\ell} \leq \frac{2^n}{(1+\lam)^{\frac{n^2}{2}}}.\] Summing the above three bounds we get the result.
\end{proof}

We now introduce the \textit{cumulants} of a random variable $X$. First, the \textit{cumulant generating function} of $X$ is $h_t(X)\coloneqq \ln \mathbb{E} [e^{tX}].$ The $\ell^{\text{th}}$ cumulant of $X$ is then 
\[\kappa_{\ell}(X)\coloneqq \left. \frac{\partial^{\ell} h_t(X)}{\partial t^{\ell}}\right|_{t=0}. \]

The first two cumulants are $\kappa_1(X)=\mathbb{E}[X]$ and $\kappa_2(X)=\text{var}(X)$. Note that for constants $a, b>0$ we have
\beq{cumlin}\kappa_\ell ((X-a)/b)=\kappa_\ell(X)/b^\ell. \enq

As before, let $ \mathbf{\Lam}$ be sampled according to $\nu$.
The cumulants of $|N(\mathbf\Lam)|$ can be expressed in terms of the cluster expansion as follows. 
For $t\geq 0$, consider the polymer model $(\cP_C,\sim,w_t)$\index{$(\cP_C,\sim,w_t)$, $\Xi_t$} with modified weights $w_t(A)=w(A)e^{t|\partial(A)|}$ and let $\Xi_t=\Xi(\cP_C,\sim,w_t)$ denote the corresponding partition function. We then have 
\begin{align*}
h_t(|\partial(\mathbf{\bar\Lam})|)= \ln \Xi_t - \ln \Xi_0\, .
\end{align*}
Applying the cluster expansion to  $\ln \Xi_t$ (for $t$ sufficiently small), taking derivatives, and evaluating at $t=0$ shows that
\begin{align}\label{cumeq}
\kappa_\ell(|\partial(\mathbf{\bar\Lam})|)= \sum_{\Gamma\in\cC}w_C(\Gamma)\|\partial( \Gamma)\|^{\ell}\, .
\end{align}
Similarly
\beq{cumeq2} \kappa_{\ell}\left( | \mathbf{\bar\Lam}| \right) = \sum_{\Gamma\in\cC}w_C(\Gamma)\|\Gamma\|^{\ell}. \enq

We use the following standard fact (see e.g.~\cite{janson1988normal}).

\begin{fact}\label{cumulantfact}
    Let $(X_n)_{n\geq 1}$ be a sequence of real-valued random variables such that $\kappa_1(X_n)\rightarrow 0,$ $\kappa_2(X_n)\rightarrow 1$ and $\kappa_\ell(X_n)\rightarrow 0$ for all $\ell\geq 3.$ Then the sequence $(X_n)_{n\geq 1}$ converges in distribution to a standard Gaussian.
\end{fact}

\begin{lemma}
\label{clt}
If $ \mathbf{\Lam}$ is sampled according to $\nu$, then
    \[ \mathbb{P}\left(|\partial(\mathbf{\bar\Lam})|=\mathbb{E}(|\partial(\mathbf{\bar\Lam})|) + o(\sqrt{\lam N})\right)=1-o(1)\]
    and
    \[ \mathbb{P}\left(|\mathbf{\bar\Lam}|=\mathbb{E}(|\mathbf{\bar\Lam}|) + o(\sqrt{\lam N})\right)=1-o(1).\]
\end{lemma}
\begin{proof}
    We present only the proof for $|\partial(\mathbf{\bar\Lam})|$; the concentration bound for $| \mathbf{\bar\Lam} |$ follows similarly. Note that
    \[ \mathbb{E}[|\partial(\mathbf{\bar\Lam})|]=\kappa_1(|\partial(\mathbf{\bar\Lam})|) \stackrel{\eqref{cumeq}}{=}\sum_{\Gamma\in\cC} w_C(\Gamma)\|\partial(\Gamma)\| \leq n\sum_{\Gamma\in\cC} |w_C(\Gamma)|\| \Gamma\| \stackrel{\eqref{cmbound}}{=} O\left( \frac{2^n n^{102}}{(1+\lam)^{\frac{n}{2}}} \right). \]

    We split our proof into two cases, depending on how large $\lam$ is. 
    
    More precisely, let $\eps>0$ be a sufficiently small constant.
    If $\lam\geq 3-\eps$, then $(1+\lam)^{\frac{n}{2}}\geq 2^{n} (1-\eps/4)^{\frac{n}{2}} = 2^{n} e^{-\eps' n},$ where $\eps'>0$ can be made arbitrarily small by choosing $\eps$ suitably small. Hence the right-hand side of the above bound is $o(\sqrt{\lam N}).$ The conclusion then follows from Markov's inequality.

Now assume that $\lam\leq 3-\eps$. In this case, we will use \Cref{cumulantfact}. For all $\ell\geq 2$, we have
    \beq{varbound} \kappa_{\ell}(|\partial(\mathbf{\bar\Lam})|)\stackrel{\eqref{cumeq}}{=}\sum_{\Gamma \in \cC}w_C(\Gamma)\|\partial(\Gamma)\|^\ell \stackrel{\eqref{cmbound}}{=}  O\left( \frac{2^n n^{100+2\ell}}{(1+\lam)^{\frac{n}{2}}} \right) \enq
    and
 \begin{align*} \text{var}(|\partial(\mathbf{\bar\Lam})|)=\kappa_2(|\partial(\mathbf{\bar\Lam})|)& \stackrel{\eqref{cumeq}}{\geq} \sum_{\substack{\Gamma\in\cC \\ \|\Gamma\|=1}} w_C(\Gamma)\|\partial(\Gamma)\|^2-
  n^2\sum_{\substack{\Gamma\in\cC \\ \|\Gamma\|\geq2}} |w_C(\Gamma)|\|\Gamma\|^2\\
& \stackrel{\eqref{cmbound}}{=}  2\binom{n}{n/2-1} \frac{\lam(n/2+1)^2}{(1+\lam)^{\frac{n}{2}+1}}- O\left( \frac{2^n n^{204}}{(1+\lam)^{n}}\right)\\
&\ =
\Omega\left(\frac{2^n}{(1+\lam)^{\frac{n}{2}}}\right).
\end{align*}
   
Setting $Z=(|\partial(\mathbf{\bar\Lam})|-\mathbb{E}|\partial(\mathbf{\bar\Lam})|)/ \sqrt{\text{var}(|\partial(\mathbf{\bar\Lam})|)},$ we have, for fixed $\ell\geq 3$,
    \[ \kappa_{\ell}(Z)\stackrel{\eqref{cumlin}}{=} \text{var}(|\partial(\mathbf{\bar\Lam})|)^{-\ell/2}\kappa_{\ell}(|\partial(\mathbf{\bar\Lam})|)= O\left( \frac{n^{100+2\ell} ((1+\lam)^{\frac{n}{2}})^{\frac{\ell}{2}-1}}{(2^{n})^{\frac{\ell}{2}-1}}  \right)=o(1)\, , \]
where for the final equality we used the assumption $\lam\leq 3-\eps$.
Hence {by Fact \ref{cumulantfact},} $Z$ converges in distribution to a unit Gaussian. Finally, note that setting $\ell=2$ in \eqref{varbound} we have $\text{var}(|\partial(\mathbf{\bar\Lambda})|)=o(\lam N).$ The conclusion follows.
\end{proof}

{Recall that $\mathbf{A}$ is an antichain sampled from the distribution $\mu^C_\lam$, and $|\mathbf{A}|=|\mathbf{\bar\Lam}|+\text{Bin}\left(N-|\partial(\mathbf{\bar\Lam})|, \frac{\lam}{1+\lam}\right).$} Taking expectations and using \eqref{cumeq} and \eqref{cumeq2}, we obtain the following lemma.

\begin{lemma}

    \label{expformula}
    \[ \mathbb{E}|\mathbf{A}|=\mathbb{E}| \mathbf{\bar\Lam}| +\frac{\lam}{1+\lam}\left(N-\mathbb{E}|\partial(\mathbf{\bar\Lam})|\right) =\frac{\lam}{1+\lam}N+\sum_{\Gamma\in \cC}  w_C(\Gamma)\left( \|\Gamma\| - \frac{\lam}{1+\lam}\|\partial(\Gamma)\|\right).\]
\end{lemma}

In particular, this gives the following estimate for $\mathbb{E}|\mathbf{A}|$. Recall that $k=\lceil n/2 \rceil$.

\begin{corollary}
\label{cor14}
For $r\in\{0,1\}$,
there exists a sequence of polynomials $F^r_j(n ,\lam)$, $j\in \mathbb{N},$ such that for any fixed $t\geq 0$ we have
    \beq{eqEAlong}
    \mathbb{E} |\mathbf{A}| 
    = \frac{\lam}{1+\lam}N\left[1+   \frac{1}{(n-k+1)(k+1)} \sum_{j=1}^{t} F^r_j(n,\lam)(1+\lam)^{-j(k+1)}\right]+ O\left(\frac{2^n n^{100(t+1)+2}}{(1+\lam)^{(t+1)k}}\right) \enq
    if $n\equiv r \pmod 2$.
     Moreover $\deg_n F^r_j\leq 2j+1$ and $\deg_\lam F^r_j\leq 2j^2$, and the coefficients of $F^r_j$ can be computed in time $e^{O(j \log j)}.$
     In particular,
      \beq{eqEAshort}
 \mathbb{E} |\mathbf{A}| = (1+o(1))\frac{\lam}{1+\lam}N \, .
      \enq
\end{corollary}
\begin{proof}
We prove this only in the case of $n$ odd; the case of $n$ even is similar but simpler. By Lemmas~\ref{clbounds} and \ref{expformula}, noting that $\|\partial(\Gamma)\| \leq n\|\Gamma\|$, we have
    \[ \mathbb{E}|\mathbf{A}|=\frac{\lam}{1+\lam}N+\sum_{\substack{\Gamma\in \cC \\ \|\Gamma\|\leq t}}  w_C(\Gamma)\left( \|\Gamma\| - \frac{\lam}{1+\lam}\|\partial(\Gamma)\|\right) + O\left(\frac{2^n n^{100(t+1)+2}}{(1+\lam)^{(t+1)k}} \right).\]
   
     Recalling that $w_C(\Gamma)=\phi(G_\Gamma)\lam^{\|\Gamma\|} (1+\lam)^{-\|\partial(\Gamma)\|}$ and taking the partial derivative with respect to $\lam$ on both sides of~\eqref{eqs0s1} we obtain, for fixed $j\geq 1$, 
\begin{align*}
& \binom{n}{k-1} 
 \left[\frac{\frac{\partial}{\partial \lam}S_j^1}{(1+\lam)^{j(k+1)}}
 - 
 \frac{j(k+1) S_j^1}{(1+\lam)^{j(k+1)+1}}
 \right] 
+
 \binom{n}{k+1} 
 \left[\frac{\frac{\partial}{\partial \lam}S_j^2}{(1+\lam)^{j(k+1)}}
 - 
 \frac{j(k+1) S_j^2}{(1+\lam)^{j(k+1)+1}}
 \right] 
 \\
 &=
 \frac{1}{\lam}\sum_{\substack{\Gamma\in \mathcal{C} \\ \|\Gamma\|=j}} w_C(\Gamma)\left(\|\Gamma\|-\frac{\lam}{1+\lam} \|\partial \Gamma\|\right).
\end{align*}
Noting that $N=\binom{n}{k}$, $\binom{n}{k-1}/N=k/(n-k+1)$,    $\binom{n}{k+1}/N=(n-k)/(k+1)$ and setting
\[
F_j^1 = k(k+1)\left[(1+\lam)\frac{\partial}{\partial \lam}S_j^1-j(k+1) S_j^1\right] + (n-k)(n-k+1)\left[(1+\lam)\frac{\partial}{\partial \lam}S_j^2-j(k+1) S_j^2\right]
\]
we obtain~\eqref{eqEAlong}. By~\Cref{clusterexpansion}, 
$\deg_n F^1_j\leq 2(j-1)+3=2j+1$ and $\deg_\lam F^1_j\leq 2j^2$. 
 Moreover the coefficients of $F^1_j(n,\lam)$ can be computed in time $e^{O(j \log j)}$ since the same is true for $S_j^1(n,\lam), S_j^2(n,\lam)$.

Finally note that we obtain~\eqref{eqEAshort} by setting $t=0$ in~\eqref{eqEAlong}.
     \end{proof}

\begin{lemma}
\label{lem18}
    For $p=p(n)$ bounded away from 1 with $pn\rightarrow\infty$, let $X=X(n)\sim \text{Bin}(n, p).$ Then, for any $m=np+o(\sqrt{np}),$ we have as $n\ra \infty$
    \[ \mathbb{P}(X=m)= \frac{1+o(1)}{\sqrt{2\pi n p (1-p)}}. \]
\end{lemma}
\begin{proof}
    $\mathbb{P}(X=m)=\binom{n}{m}p^m(1-p)^{n-m}$, and the result follows from Stirling's formula.
\end{proof}

\begin{lemma}
    \label{lem16}
    Suppose $m\in[N]$ is such that
    \beq{exp} | \mathbb{E} |\mathbf{A}| - m|=o(\sqrt{{\lambda} N}).\enq
    Then
    \beq{am(Cn)} i_m^C=(1+o(1))\frac{(1+\lambda)^{N+1} \Xi_C}{\lambda^m \sqrt{2\pi N\lambda}}.\enq
    
\end{lemma}
\begin{proof}
    We have $i_m^C=\mathbb{P}(|\mathbf{A}|=m) \lam^{-m}(1+\lam)^N\Xi_C.$ Hence it suffices to show that
    \[ \mathbb{P}(|\mathbf{A}|=m)=(1+o(1)) \frac{1+\lam}{\sqrt{2\pi N\lam}}. \]
As before we let $\mathbf \Lam$ denote a sample from $\nu$ so that
    \beq{bineq} \mathbb{P}(|\mathbf{A}|=m)=\sum_{\Lam\in \Omega} \mathbb{P}(\mathbf{\Lam}=\Lam)\cdot \mathbb{P}\left( \text{Bin}\left(N-|\partial(\bar\Lam)|, \frac{\lam}{1+\lam}\right)=m-| \bar\Lam | \right).\enq

    By Lemma~\ref{largedeviation}, we may condition on the event $|\partial(\bar\Lam)|\leq N/n$ while only incurring an additive error of $o(1/\sqrt{\lam N})$. The binomial probabilities in $\eqref{bineq}$ are then bounded uniformly by $O(1/\sqrt{\lam N})$, so it suffices to show that they are in fact, with high probability, equal to $(1+o(1))\frac{1+\lam}{\sqrt{2\pi N\lam}}$. By Lemma~\ref{lem18}, it suffices to show that with high probability in the choice of $\mathbf{\Lam}$, we have
    \[ \frac{\lam}{1+\lam}\left(N-|\partial(\bar\Lam)|\right)=m-|\bar\Lam|+o(\sqrt{\lam N}).\]
    But this follows immediately by recalling our assumption that $m=\mathbb{E} |\mathbf{A}|+o(\sqrt{\lam N})$ together with the identity given by Lemma~\ref{expformula} and the concentration bounds of Lemma~\ref{clt}.
\end{proof}

\begin{lemma}
    \label{lem17}
    For $r\in\{0, 1\}$, there exists a sequence of rational functions $B_j^r(n,\beta), \ j\in \mathbb{N},$ such that the following holds. Given $\eps>0$ and $t \in \mathbb N$, let $q=\lceil t/2 \rceil-1.$ If $n\equiv r \pmod 2$ and $m\in [N]$ is such that $\beta\coloneqq \frac{m}{N}\in [1-4^{-1/t} +\eps,1-\eps]$, then taking
    \[ \lam \coloneqq \frac{\beta}{1-\beta} + \sum_{j=1}^q B_j^r(n, \beta)(1-\beta)^{j(k+1)}\]
    we have
    \[ \left| \mathbb{E} |\mathbf{A}| -m \right| = o(\sqrt{\lam N}). \]
    Moreover, the coefficients of $B_j^r$ can be computed in time $e^{O(j \log j)}$.
\end{lemma}
\begin{proof}
	With $q=\lceil t/2 \rceil -1$, set
\[ \lam= \frac{\beta}{1-\beta} + \sum_{j=1}^q B^r_j(n, \beta)(1-\beta)^{j(k+1)} \]
where the functions $B^r_j$ are rational functions in $n, \beta$ of constant degree (independent of $n$) to be determined later. Note that by the assumption on $\beta$ we have $\lam=\Theta(1)$.

By \Cref{cor14} we have
 \beq{poly1}  \mathbb{E} |\mathbf{A}| 
    = \frac{\lam}{1+\lam}N\left[1+   \frac{1}{(n-k+1)(k+1)} \sum_{j=1}^{q} F_j^r(n,\lam)(1+\lam)^{-j(k+1)}\right]+ O\left(\frac{2^n n^{100(q+1)+2}}{(1+\lam)^{(q+1)k}}\right)\, , \enq
where $F_j^r (n, \lam)$ is a polynomial with $\deg_n F_j^r\leq 2j+1$ and $\deg_\lam F_j^r \leq 2j^2$. We now show that our choice of $q=\lceil t/2 \rceil-1$ makes the error in the above expression equal to $o(N^{1/2})$.

Let $X\coloneqq\lam(1-\beta)-\beta= \sum_{j=1}^q B^r_j (1-\beta)^{j(k+1)+1}$ and note that $X=o(1)$. Since $(1+\lam)^{-1}=(1-\beta)/(1+X)=(1+o(1)) (1-\beta)$ and $\beta \ge 1-4^{-1/t}+\eps$, it follows that
\begin{align}\label{eq:expwin}
    (1+\lam)^{-(q+1)k}\leq e^{o(n)} (1-\beta)^{(q+1)k}
=O\left( e^{-\Omega(n)} \cdot 4^{-\frac{(q+1)k}{t}}\right) = O\left( e^{-\Omega(n)} N^{-1/2}\right)\, . 
\end{align}
Returning to~\eqref{poly1}, we conclude that 
\begin{align}\label{poly12}   \mathbb{E} |\mathbf{A}| 
    = \frac{\lam}{1+\lam}N\left[1+   \frac{1}{(n-k+1)(k+1)} \sum_{j=1}^{q} F_j^r(n,\lam)(1+\lam)^{-j(k+1)}\right] + o(N^{1/2})\, .
\end{align}
 Our goal is to find a suitable choice of $B^r_1, \dots B^r_q$ that makes \eqref{poly12} equal to $m+o(N^{1/2})$. Our first step is to expand the RHS of~\eqref{poly12} as a function of $n,\beta, B^r_1, \dots B^r_q$. 

Note that $X=O(n^{O(1)}(1-\beta)^{n/2})$. Since $\beta\ge 1-4^{-1/t}+\eps$, it follows that $X^{q+1}=O(e^{-\Omega(n)}N^{-1/2})$.

First we expand $(1+\lam)^{-\ell}$ as function of $\beta$ and $X$: for $\ell=O(n)$,
\begin{align}
(1+\lam)^{-\ell}=\frac{(1-\beta)^\ell }{(1+X)^ \ell}
&= (1-\beta)^\ell\left[ \sum_{i=0}^q \binom{-\ell}{i} X^i + O\left((\ell X)^{q+1}\right)\right]\nonumber\\
&=
(1-\beta)^\ell \sum_{i=0}^q \binom{-\ell}{i} X^i + O\left( e^{-\Omega(n)}N^{-1/2} \right)\, .\label{taylor}
\end{align}

Next we rewrite $F_j^r(n,\lam)$ as a function of $n,\beta$ and $X$. Since $\deg_\lam F_j^r \leq 2j^2$ and $\lam=\frac{\beta+X}{1-\beta}$, we may write $F_j^r (n, \lam)=(1-\beta)^{-c_j} G_j^r  (n, \beta, X)$
for some integer $0\leq c_j\leq 2j^2$ and $G_j^r $ a polynomial in $n, \beta, X$ with $\deg_n G_j^r\leq 2j+1$. Plugging this into~\eqref{poly12} along with~\eqref{taylor} we have
\begin{gather}
    \notag \frac{1}{N}\mathbb{E} |\mathbf{A}|= \\ \label{poly2}(\beta+X)\sum_{i=0}^q(-X)^i\left[1 +  \sum_{j=1}^q \frac{G_j^r (n, \beta, X)}{(n-k+1)(k+1)} (1-\beta)^{j(k+1)-c_j}\sum_{i=0}^q \binom{-j(k+1)}{i}X^i\right] +o(N^{-1/2}).
\end{gather}

Recalling that $X=\sum_{j=1}^q B^r_j (1-\beta)^{j(k+1)+1}$, we rewrite \eqref{poly2} as a polynomial in $(1-\beta)^{k+1}$ and note that by~\eqref{eq:expwin}, $n^{O(1)}(1-\beta)^{s(k+1)}=o(N^{-1/2})$ for $s\geq q+1$. This yields
\beq{poly3} 
\frac{1}{N}\mathbb{E}_{C_n, \lam} |\mathbf{A}|= \beta + \sum_{j=1}^q \frac{H^r_j (n, \beta, B^r_1 , \dots B^r_j)}{(n-k+1)(k+1)(1-\beta)^{c_j}}\cdot (1-\beta)^{j(k+1)}+ o(N^{-1/2}) \enq
where $H^r_j $ is a polynomial in $n, \beta, B^r_1 , \dots B^r_j$. Moreover, note that each $H^r_j $ is linear in $B^r_j$, with nonzero coefficient. By induction, we conclude that there is a choice of $B^r_1 , \dots B^r_j$ such that $H^r_1 =\dots=H^r_q =0$, with each $B^r_j$ a rational function of $n,\beta$ of constant degree (depending on $j$ but not on $n$). Plugging this choice of $B^r_j$ into \eqref{poly3} and recalling that $\lam=\Theta(1)$ we conclude that
\[\mathbb{E}|\mathbf{A}|=\beta N+o(\sqrt{\lam N})\, . \]

Finally, note that the above proof also gives an algorithm for computing the functions $B^r_j$. Indeed, the coefficients of $F^r_j $ can be computed in time $e^{O(j\log j)}$ (by \Cref{cor14}), and therefore so can the coefficients of  $G^r_j $ and $H^r_j $. Having computed $B^r_1, \ldots, B^r_{j-1}$, we can compute $B^r_j$ by solving a linear equation.
\end{proof}

We are now in a position to prove \Cref{thm5pre}. The final ingredient that we require is the following simple reformulation of \Cref{lem18}: for $m=m(n)\in [N]$ such that $\beta:=m/N$ is bounded away from  $0$ and $1$, we have, with $\gl_0:=\frac{\beta}{1-\beta}$,
\beq{binombd} \binom{N}{m}=(1+o(1))\frac{(1+\lam_0)^{N+1}}{\lam_0^{m} \sqrt{ 2\pi m \lam_0}}\,.  \enq

 \begin{proof}[Proof of \Cref{thm5pre}]
Once again, we present only the proof for $n$ odd. The case of even $n$ is analogous. Take $\lam$ as in \Cref{lem17} and let $\lam_0=\beta/(1-\beta)$. Note that $\lam=\Theta(1)$ and so we are free to apply the results of \Cref{givensize}. Note also that $\lam=\lam_0+o(1)$.

By \Cref{lem16} 
and \eqref{binombd} we have 
\begin{align}\label{eq:psit1}
i_m^C
= &(1+o(1)) \binom{N}{m} \left(\frac{1+\lam}{1+\lam_0}\right)^N \left( \frac{\lam_0}{\lam}\right)^{m}\Xi_C\, ,
\end{align}
and by~\Cref{clusterexpansion}

\begin{align}\label{eq:psit2}
\Xi_C=(1+o(1))\exp \left[\sum_{j=1}^{t-1}\binom{n}{k-1}S^1_j(n,\lam)(1+\lam)
^{-j(k+1)}+  
\binom{n}{k+1}S^2_j(n,\lam)(1+\lam)
^{-j(k+1)}\right]\, .
\end{align}
Taking logarithms and dividing by $N=\binom{n}{k}$, it remains to prove that we can find $R_j^1=R_j^1(n, \beta)$ for all $j\geq 1$ such that

\begin{align}\label{eqRjdef}
&\ln\left(\frac{1+\lam}{1+\lam_0}\right) -\beta\ln\left(\frac{\lam}{\lam_0}\right) + \sum_{j=1}^{t-1} \frac{k}{n-k+1}S^1_j(n,\lam)(1+\lam)
^{-j(k+1)}+  
\frac{n-k}{k+1}S^2_j(n,\lam)(1+\lam)
^{-j(k+1)}\nonumber\\
&=\sum_{j=1}^{t-1}R_j^1(n, \beta)(1-\beta)^{jk} + o(N^{-1}).
\end{align}

We now expand each term on the LHS of the above as a power series in $(1-\beta)^{k}$. As in the proof of \Cref{lem17}, we take $X\coloneqq \sum_{j=1}^q B^r_j(n, \beta) (1-\beta)^{j(k+1)+1}$, where $q=\lceil t/2\rceil -1$. Note that $X^t=e^{-\Omega(n)}N^{-1}$ since $\beta\geq 1-4^{-1/t}+\eps$. By Taylor expanding, we have
\beq{taylorseries}\ln\left(\frac{1+\lam}{1+\lam_0}\right)-\beta\ln\left(\frac{\lam}{\lam_0}\right)=\ln(1+X)-\beta\ln(1+X/\beta)=\sum_{i=1}^{t-1}\frac{(-1)^{i+1}}{i}(1-\beta^{1-i})X^i+o(N^{-1}).\enq

We also note that if $\ell=O(n)$, then
\begin{align}\label{tayloragain}
(1+\lam)^{-\ell}=
(1-\beta)^\ell \sum_{i=0}^{t-1} \binom{-\ell}{i} X^i + o\left(N^{-1} \right)\, .
\end{align}

Now for $r\in\{1,2\}$, each $S_j^r$ is a polynomial in $n$ and $\lam$ with $\deg_n(S_j^r)\leq 2j$ and $\deg_{\lam}(S_j^r)\leq 2j^2$. Since $\lam=(\beta+X)/(1-\beta)$, we may therefore write
\[S_j^r(n, \lam)=(1-\beta)^{-c_j}T_j^r(n, \beta, X)\]
for some non-negative integer $c_j=c_j^r\leq 2j^2$ and $T_j^r$ a polynomial in $n, \beta, X$ such that $\deg_n(T_j^r)\leq 2j$. Therefore by~\eqref{tayloragain},
\beq{polyfinal}\sum_{j=1}^{t-1}S_j^r(n, \lam)(1+\lam)^{-j(k+1)}=\sum_{j=1}^{t-1}T_j^r(n, \beta, X)(1-\beta)^{j(k+1)-c_j^{(r)}}\sum_{i=0}^{t-1}\binom{-j(k+1)}{i}X^i+o(N^{-1}).  \enq

To compute the coefficients of $R_j^1$, we substitute \eqref{taylorseries} and \eqref{polyfinal} into the LHS of~\eqref{eqRjdef}, expand the powers of $X=\sum_{j=1}^q B_j (1-\beta)^{j(k+1)+1}$ and collect the coefficients of $(1-\beta)^{k}, \dots, (1-\beta)^{(t-1)k}$. Finally, observe that $R_j^1$ can be computed in time $e^{O(j\log j)}$, since so can $S_j^r, T_j^r$ ($r\in\{1,2\}$) and $B^1_j$.
 \end{proof}

\section{The proofs of Theorems~\ref{dualthm1}-\ref{sparsesperner3}: the case of $n$ even}\label{seceven}
In this section we prove each of our main theorems in the case where $n$ is even.

\subsection{Theorem~\ref{saporefined2}}
We begin with Theorem~\ref{saporefined2}, which follows easily from the following lemma. Recall that when $n$ is even, $C_n=L_{[k-1,k+1]}$ and $k=n/2$.

 \begin{lemma}\label{lempolymeq} 
\[Z(C_n, \lam)\equiv(1+\lambda)^N\Xi_C\, .\]
In particular,
\[
i_m(C_n)=i_m^C\,
\]
for all $m\in[N]$.
 \end{lemma}
 \begin{proof}
Given an antichain $S\subseteq C_n$, let 
\[
\rho(S)= S\cap(L_{k-1} \cup L_{k+1})\, .
\]
We partition the collection of antichains $S\subseteq C_n$ according to the value of $\rho(S)$. Note that if $\rho(S)=A$, then $S\cap L_k$ can be any subset of $L_k\backslash \partial(A)$. We conclude that 
\begin{align*}
Z(C_n,\lam) 
=
\sum_{A\subseteq L_{k-1}\cup L_{k+1}}\sum_{\substack{S\subseteq C_n\\S\in \mathcal A,\,  \rho(S)=A}} \lam^{|S|}
=\sum_{A\subseteq L_{k-1}\cup L_{k+1}} \lam^{|A|}(1+\lam)^{N-|\partial(A)|}
=
(1+\lam)^N\sum_{A\subseteq L_{k-1}\cup L_{k+1}}  w_C(A)\, .
\end{align*}
The first claim follows by noting that if $A$ decomposes into the collection of mutually compatible polymers $\{A_1, \ldots, A_\ell\}$, then $w_C(A)=\prod_{i=1}^\ell w_C(A_i)$.
The second claim follows by noting that $Z(C_n,\lam)=\sum_{m=0}^Ni_m(C_n)\lam^m$ and (e.g. by \eqref{eqPAm}) $(1+\lambda)^N\Xi_C=\sum_{m=0}^Ni_m^C\lam^m$.
\end{proof}

In the case of $n$ even, Theorem \ref{saporefined2} now follows by combining Theorem \ref{saporefined} and Theorem \ref{clusterexpansion}, taking $\lam=1$ and setting $P_j^0(n)\coloneqq 2^{-j}S^0_j(n, 1)$. We remark that in order to obtain the desired error term $O\left(2^{-tn/2}n^{O(t)} \right)$, we need to sum the first $t+1$ terms in the sum in \Cref{clusterexpansion}.

Next we turn to Theorems~\ref{thm5} and \ref{sparsesperner3}. We will apply our results from \Cref{givensize}. Recall that in \Cref{givensize} we assumed that $\lam$ is bounded, which will translate to the case where $\beta$ in Theorems~\ref{thm5} and \ref{sparsesperner3} is bounded away from $1$. We treat the (simpler) case of $\beta$ close to $1$ separately in  \Cref{closeto1}.

\subsection{Theorem \ref{sparsesperner3} in the case of $\beta$ bounded away from $1$}
\label{sparsespernersection}
For concreteness, let us assume that $\beta \leq 9/10$.
Let $2C\log^2 n/\sqrt{n}\leq\beta\leq 9/10$ where $C$ is chosen sufficiently large so that Theorem~\ref{trunc} and Lemma~\ref{lem16} hold. Set $m=\beta N$. Our task is to show that $i_m(B_n)=(1+o(1))i_m(C_n)$. For this we will appeal to Lemma~\ref{transference}. 

As in \Cref{givensize}, let $\mathbf A$ denote a sample from $\mu_\lam^C$. Recall from \Cref{cor14} that 
\[\E |\mathbf A|=(1+o(1))\lam N/(1+\lam).\]
Since $\E |\mathbf A|$ is continuous in $\lam$, it follows that for $\lam\in (C\log^2 n/ \sqrt{n}, 10)$ the quantity $\mathbb{E}|\mathbf{A}|$ takes all values in the interval $\big[{2C\log^2 n} N/\sqrt{n}, 9N/10\big].$ There therefore exists a value of $\lam\in (C\log^2 n/ \sqrt{n}, 10)$ such that $\mathbb{E} | \mathbf{A}|=m.$

By Theorem~\ref{trunc}, for this choice of $\lam$ we have $Z(\lam)=Z(B_n,\lam)=\left(1+o(1/\sqrt{N})\right)Z(C_n,\lam)$. By Lemma~\ref{lem16}, we also have $i_m(C_n)\lam^m/Z(C_n, \lam)=\Omega\left(\frac{1}{\sqrt{\lam N}}\right),$ and so condition \eqref{lambda.cond} of Lemma~\ref{transference} holds. By Lemma~\ref{transference}, we obtain $i_m(B_n)=(1+o(1))i_m(C_n)$ as desired.

\subsection{Theorem \ref{thm5} in the case of $\beta$ bounded away from $1$}\label{sparsespernersection2}

As in the previous section, we assume that $\beta\leq 9/10$. 

Fix $t\geq 1$, $\eps>0$ and suppose $m\in[N]$ is such that $\beta=m/N\geq 1-4^{-1/t}+\eps$. By \Cref{sparsesperner3} and \Cref{lempolymeq}, $\psi(n,m)=i_m(B_n)=(1+o(1))i_m(C_n)=(1+o(1))i_m^C$. Theorem \ref{thm5} (for $\beta\leq 9/10$) now follows from \Cref{thm5pre}.

\subsection{Theorems~\ref{thm5} and \ref{sparsesperner3} in the case of $\beta$ close to $1$}\label{closeto1}
Suppose that $\lam\geq 9$. In this case, we may further truncate the partition function $Z(C_n, \lam),$ reducing it to the trivial partition function of the middle layer $Z(L_{n/2}, \lam)=(1+\lam)^N$. Indeed, Theorem \ref{clusterexpansion} (applied with $t=1$) and \Cref{lempolymeq} gives $Z(C_n, \lam)=(1+o(1/\sqrt{N}))(1+\lam)^N$ for $\lam\geq 8$. Therefore by \Cref{trunc},
\beq{1layer} Z(B_n, \lam)=\left(1+o(1/\sqrt{N})\right)Z(L_k, \lam)\, . \enq
To obtain the result of Theorem \ref{thm5}, we wish to apply Lemma~\ref{transference} with $\mathcal F_1$ and $\mathcal F_2$ the set of antichains in $L_k$ and $B_n$, respectively. Set $\beta=m/N$ and suppose that $\beta\geq 9/10$ and set $\gl:=\frac{\beta}{1-\beta}$ (so in particular $\lam\geq 9$). We claim that $\lam$  satisfies \eqref{lambda.cond} with $f(n)=o(1/\sqrt{N})$ as in~\eqref{1layer}, for which it suffices to show that $\binom{N}{m} \lam^m(1+\lam)^{-N} =\Omega( 1/\sqrt{N}).$ 

Recall the standard estimate $\binom{N}{m}\geq \frac{1}{\sqrt{2N}} 2^{H(\beta)N},$ where $H$ is the binary entropy function. For the above $\gl$  we obtain
\[ \frac{\binom{N}{m} \lam^m}{(1+\lam)^N} \geq \frac{1}{\sqrt{2N}} \left(1-\beta\right)^{-(1-\beta)N} \beta^{-\beta N} \left(\frac{\beta}{1-\beta}\right)^{\beta N} \left( 1-\beta\right)^N=\frac{1}{\sqrt{2N}}.\]

{Combining this with \eqref{1layer} and using Lemma~\ref{transference}, we see that $\psi(n,m)=i_m(B_n)=(1+o(1))i_m(L_k)=(1+o(1))\binom{N}{m}$ for all $m\geq  9N/10$. This is precisely the statement of Theorem \ref{thm5}, and Theorem \ref{sparsesperner3} in this regime follows immediately.}

\subsection{Theorem \ref{dualthm1}}
\label{phasetransitionanalysis}

\subsubsection{1-statement and behaviour inside the scaling window}
\label{1statement}
Let $\beta=m/N=\frac{3}{4}-\frac{\ln(n)}{4n} + \frac{h(n)}{n}$. Theorem \ref{thm5} shows that if $h(n)=\Omega(n)>0$ then $\psi(n,m)=i_m(B_n)=(1+o(1))\binom{N}{m}$. For the 1-statement and behaviour inside the scaling window we may therefore assume that $h(n)=o(n)$. Let $\lam$ be as in Lemma~\ref{lem17} applied with $t=2$, i.e. $\lam=\beta/(1-\beta)$. 
Combining~\eqref{eqs0s1},~\eqref{eq:psit1},~\eqref{eq:psit2} (again with $t=2$), and \Cref{lempolymeq} gives
\begin{align}\label{eq:psinm2}
\psi(n, m)= (1+o(1)) \binom{N}{m} \exp\left( \sum_{\Gamma\in \mathcal{C}, \|\Gamma\|=1} w_C(\Gamma) \right)\, .
\end{align}
We now turn our attention to computing the sum on the RHS.

Let $\Gamma\in \cC$ be a cluster of size 1. Then $\Gamma$ consists of precisely one polymer $A$ which has size 1, i.e. it is an element of $V(G_C)=L_{n/2-1}\cup L_{n/2+1}$. In particular, the incompatibility graph $G_\Gamma$ consists of a single vertex and so $\phi(G_\Gamma)=1$ and  $w_C(\Gamma)=w_C(A)=\lambda (1+\lambda)^{-\frac{n+2}{2}}$. There are $2\binom{n}{n/2+1}$ choices for $\Gamma$, and hence,
\begin{align}\label{1clusters}
    \sum\limits_{\|\Gamma \|=1} w_C(\Gamma) 
    &= 2\binom{n}{n/2+1} \lambda (1+\lambda)^{-\frac{n+2}{2}}\nonumber\\
    &=N  \frac{2\beta n}{(n+2)} {(1-\beta)^{\frac{n}{2}}}\\
     &=(1+o(1))\frac{2^n}{\sqrt{\pi n/2}}\cdot2\cdot \frac{3}{4}\left(\frac{1+\frac{\ln n}{n}-\frac{4h(n)}{n}}{4}\right)^{n/2}\nonumber\\
 &= (1+o(1))\frac{3}{\sqrt{2\pi n}} 
 \exp\left\{\frac{\ln n}{2}-2h(n)+O\left(\frac{(\ln n)^2+(h(n))^2}{n}\right)\right\}\nonumber\\
 &=(1+o(1))\frac{3}{\sqrt{2\pi }} 
 \exp\left\{-2h(n)(1+o(1))\right\}\, ,\nonumber
\end{align}
where for last two equalities we used the assumption that $h(n)=o(n)$.
Considering the cases $h(n)\to\infty$ and $h(n)\to c\in\mathbb R$ recovers the first two statements of Theorem \ref{dualthm1} in the case of $n$ even.

\begin{remark}\label{rem:R01}
Theorem \ref{thm5} shows that for $\beta=m/N\geq 1/2+\eps$,
\[
    \psi(n,m)=(1+o(1))\binom{N}{m}\exp \left( N R_1^{0}\left(n, \beta \right) \cdot \left(1-\beta\right)^{\frac{n}{2}} \right)\, .
    \]
    Comparing this to~\eqref{eq:psinm2} and~\eqref{1clusters} reveals that $R_1^{0}\left(n, \beta \right)=\frac{2\beta n}{n+2}$.
\end{remark}

\subsubsection{$0$-statement}
\label{0statement}
We now consider the case $h(n)\to-\infty$. First note that if $\beta=\Omega(1)$ then Theorem \ref{thm5} shows that
\begin{align}\label{eq:psinm3}
    \psi(n,m)=(1+o(1))\binom{N}{m}\exp \left((1+o(1)) N R_1^{\cE}\left(n, \beta \right) \cdot \left(1-\beta\right)^{\frac{n}{2}} \right)\, .
    \end{align}
    If $h(n)\to-\infty$ then $N R_1^{\cE}\left(n, \beta \right) \cdot \left(1-\beta\right)^{\frac{n}{2}}\to \infty$ and the result follows. 
    
    We may therefore assume that $\beta=m/N=o(1)$. In this case we proceed with a more direct argument. As before, we let $k=n/2$. 
    
    We consider first the case where $m=\omega(1)$. In this case it suffices to count the number of antichains of size $m$ that have one element in $L_{k-1}$ and the other $m-1$ in $L_k$. The number of such antichains is $\binom{n}{k-1} \binom{N-k-1}{m-1},$ and we wish to prove that this is much bigger than $\binom{N}{m}.$ We have:

\beq{binoms} \frac{\binom{n}{k-1}\binom{N-k-1}{m-1}}{\binom{N}{m}}=(1+o(1))\frac{N\binom{N-k-1}{m-1}}{\binom{N}{m}}=(1+o(1))m\prod_{i=1}^{m-1}\frac{N-k-i}{N-i}. \enq

Now, when $m\leq k$ this can be lower bounded by $(1+o(1))m \left(\frac{N-k}{N}\right)^m\geq (1+o(1))me^{-\frac{mk}{N}}=(1+o(1))m=\omega(1)$. For $m> k$, instead note that the product on the right hand side of \eqref{binoms} telescopes and thus the above expression equals
\[ (1+o(1))m\prod_{i=1}^{k} \frac{N-m+1-i}{N-i} \geq (1+o(1))m \left( \frac{N-m}{N}\right)^k\geq (1+o(1))me^{-\frac{mk}{N}}=\omega(1), \]
where the last equality holds since $m=o(N)$ and so in particular $m/\ln m=o(N/\ln N)=o(N/k)$.

Finally, it remains to deal with the case of $m=O(1)$. Here, our previous approach doesn't work. Indeed, a constant proportion of the antichains of size $m$ in $L_{k-1}\cup L_k$ are contained entirely in $L_k$. We will instead count \textit{all} antichains of size $m$ in $B_n$.

Let $u$ and $v$ be two elements of $B_n$ chosen independently and uniformly at random. {Recall that we say $u$ and $v$ are adjacent in $B_n$ if and only if $u\subsetneq v$ or $v\subsetneq u$, and we write $u\sim v$.} Conditioning on $|u|=\ell,$ the probability that $u$ and $v$ are {either adjacent in $B_n$ or equal} is precisely $\frac{2^\ell + 2^{n-\ell}{-1}}{2^n}.$ Dropping the conditioning, we have
\[ \mathbb{P}(u\sim v {\text{ or } u=v})= \sum_{\ell=0}^{n} \frac{\binom{n}{\ell}}{2^n}\times\frac{2^\ell+ 2^{n-\ell}{-1}}{2^n}=2\cdot\left(\frac{3}{4}\right)^n{-\left(\frac{1}{2}\right)^n=o(1)}. \]

Now sample $m=O(1)$ distinct elements $u_1, \dots, u_m$ independently and uniformly at random from $B_n.$ {For $1\leq i<j\leq m$, let $E_{i, j}$ denote the event `$u_i\nsim u_j$ and $u_i \neq u_j$'. Observe that the event $\bigcap_{i,j} E_{i, j}$ is precisely the event that $\{ u_1, \dots u_m\}$ is an antichain of size $m$. But by the above, we have $\mathbb{P}(E_{i, j})=1-o(1)$ for all $i, \ j$. Since there are only $\binom{m}{2}=O(1)$ events $E_{i, j}$, we have} 

\beq{bddsize} \mathbb{P}\left(\{ u_1, \dots, u_m\} \text{ is an antichain}\right)=1-o(1). \enq

Let $U$ be the random variable with distribution equal to that of $\{u_1, \dots, u_m\}$ conditioned on the event that $\{u_1, \dots, u_m\}$ is an antichain {of size $m$}. We clearly have that $U$ is uniformly distributed over the set of all antichains of size $m$. On the other hand, $\mathbb{P}(\{ u_1, \dots , u_m\} \subseteq L_k)=o(1),$ and therefore by \eqref{bddsize}, $\mathbb{P}(U\subseteq L_k)=o(1)$ too. This concludes the proof. 

\begin{remark*}
    The computations in \eqref{bddsize} hold for $m=o((4/3)^{n/2})$, and hence we may conclude that in this regime $\psi(n,m)=(1+o(1))\displaystyle\binom{2^n}{m}.$
\end{remark*}

\section{The proofs of Theorems~\ref{dualthm1}-\ref{sparsesperner3}: the case of $n$ odd}\label{secodd}
In this section we prove Theorems~\ref{dualthm1}-\ref{sparsesperner3} in the case where $n$ is odd. Throughout this section we set $k=\lceil n/2 \rceil=(n+1)/2$ and $\lam>C \log^2 n/\sqrt{n}$ where $C$ is a sufficiently large constant. 

Recall that $C_n$ is the union of the four consecutive layers $L_{k-2}, L_{k-1}, L_{k}, L_{k+1}$ and by Theorem~\ref{trunc} the partition function $Z(\lam)$ is closely approximated by $Z(C_n,\lam)$. 

Our first aim is to establish an analog of \Cref{lempolymeq} (\Cref{defapprox} below), relating $Z(C_n,\lam)$ to a suitable polymer model partition function.

The main complication arising from the $n$ odd case is that there are now \emph{two} maximum antichains $L_k$ and $L_{k-1}$, or in physics terminology two \emph{ground states}. It will turn out that almost all of the contribution to $Z(C_n,\lam)$ comes from antichains that are `close' to a subset of $L_{k}$ or $L_{k-1}$. Contrast this to the case where $n$ is even, where almost all of the contribution to $Z(C_n,\lam)$ comes from antichains that are close to a subset of the (unique) middle layer. 

To deal with this, we show that the measure $\mu_{C_n,\lam}$ (the hard-core measure on $C_n$ at activity $\lam$ defined at~\eqref{eqLocalId0}) is well-approximated by a mixture of two measures: one on antichains in $L_{[k-1,k+1]}$ and one on antichains in $L_{[k-2,k]}$. The measure on antichains in $L_{[k-1,k+1]}$ will concentrate on antichains that are close to a subset of $L_k$ and the measure on antichains in $L_{[k-2,k]}$ will concentrate on antichains that are close to a subset of $L_{k-1}$. The measure on antichains in $L_{[k-1,k+1]}$ will be precisely the measure $\mu^C_{\lam}$ introduced in \Cref{givensize}. The measure on antichains $L_{[k-2,k]}$ will equivalent under the symmetry $B_n\to B_n$, $A \mapsto A^c$.

We note that a similar strategy was carried out in~\cite{hypercube} to analyse independent sets in the hypercube: there it is shown that the hard-core measure on the discrete hypercube $Q_n$ is well-approximated by a mixture of measures that favour odd/even vertices of $Q_n$. This idea is also present in the earlier work of Korshunov and Sapozhenko~\cite{korshunov1983number} and Galvin~\cite{galvin} on independent sets in $Q_n$.

\subsection{Approximating $Z(C_n,\lam)$ by a polymer model partition function}\label{odd}

Write $\cL= L_{k-1}$, $\cU= L_{k}$ (for `lower' and `upper' ground states respectively).

In this section we introduce two polymer models, $(\cP_\cL,\sim,w_\cL)$ and $(\cP_\cU,\sim,w_\cU)$, which are designed to capture antichains in $C_n$ which are `close' to a subset of $\cL$ or $\cU$. In the next section, we perform one more `truncation step' which shows that the main contribution to $Z(C_n,\lam)$ comes from antichains that are a subset of either  $L_{[k-1,k+1]}$ or $L_{[k-2,k]}$.

Given $A\subseteq \cL$, let $\Int(A)=\{ v\in L_{k-2}: N^+(v)\subseteq A\}$. For $i \in \{k-1, k+1\}$, $\cP_\cL$ consists of the sets $A \sub L_i$ which are i) 2-linked in the graph $L_i \cup L_k$; and ii) $|\partial(A)| \ge (1+1/n)|A|$, recalling that $\partial(A)=\{x \in \cU:\exists v \in A, x \subseteq v \text{ or } v \subseteq x\}$.
We say $A_1, A_2\in \cP_\cL$ are compatible, written $A_1\sim A_2$, if and only if $\partial(A_1) \cap \partial(A_2)=\emptyset$.
For $A\in \cP_\cL$, set
\[
w_\cL(A)=
\begin{cases}
\lambda^{|A|}(1+\lambda)^{-|N^-(A)|}
& \text{ if $A \subseteq L_{k+1}$,} \\
\lambda^{|A|}(1+\lambda)^{-|N^+(A)|}\sum\limits_{B\subseteq \Int(A)} \lambda^{|B|-|N^+(B)|} & \text{if $A \subseteq \cL$ .} 
\end{cases}
\]

As a look ahead, it will be useful to compare this polymer model to the polymer model $(\cP_C, \sim, w_C)$ defined in~\Cref{secthreelayer}. Indeed we have $\cP_C=\cP_\cL$ and $w_C(A)\leq w_\cL(A)$ for all polymers $A$. In the next section we will `reduce' the study of $(\cP_\cL,\sim,w_\cL)$\index{$(\cP_\cL,\sim,w_\cL)$} to the study of the (simpler) polymer model $(\cP_C, \sim,w_C)$.

We define a polymer model $(\cP_\cU, \sim, w_\cU)$\index{$(\cP_\cU, \sim, w_\cU)$} analogously: given $A\subseteq \cU$, let $\Int(A)=\{ v\in L_{k+1}: N^-(v)\subseteq A\}$. For $i \in \{k-2, k\}$, $\cP_\cU$ consists of the sets $A \sub L_i$ which are i) 2-linked in the graph $L_i \cup L_k$; and ii) $|\partial(A)| \ge (1+1/n)|A|$, where, this time, we use $\partial(A):=\{x \in \cL: \exists v \in A, x \subseteq v \text{ or } v \subseteq x\}$.
We say $A_1, A_2\in \cP_\cU$ are compatible, written $A_1\sim A_2$, if and only if $\partial(A_1) \cap \partial(A_2)=\emptyset$. 
For $A\in \cP_\cU$, set
\[
w_\cU(A)=
\begin{cases}
\lambda^{|A|}(1+\lambda)^{-|N^+(A)|}
& \text{ if $A \subseteq L_{k-2}$,} \\
\lambda^{|A|}(1+\lambda)^{-|N^-(A)|}\sum\limits_{B\subseteq \Int(A)} \lambda^{|B|-|N^-(B)|} & \text{if $A \subseteq \cU$.} 
\end{cases}
\]

Recall that for $\cD\in\{\cL,\cU\}$, $\Xi(\cP_\cD, \sim, w_\cD)$\index{$(\cP_\cD, \sim, w_\cD)$} denotes the polymer model partition function (defined at~\eqref{defpolypartfun}).
Note that by symmetry $\Xi(\cP_\cL, \sim, w_\cL)=\Xi(\cP_\cU, \sim, w_\cU)$ and so we denote the partition function by $\Xi$. Let $\Omega_\cD=\Omega(\cP_\cD,\sim)$\index{$\Omega_\cD$} denote the set of 
collections of compatible polymers in $\cP_\cD$. Let $\nu_\cD$ denote the measure on $\Omega_\cD$ given by
\begin{equation*}
    \nu_\cD(\Lam)=\frac{\prod_{A\in \Lam}w_\cD(A)}{\Xi} \text{\,  for\,  } \Lam\in \Omega_\cD.
\end{equation*}

It will be convenient to introduce the following measure on antichains in $C_n$ which will closely approximate the hard-core model on $C_n$ at activity $\lam$ (see \Cref{defapprox} below) but is more amenable to analysis. Let $\hat \mu_\lam$\index{$\hat \mu_\lam$} denote the distribution of the random antichain produced by the following process:

\begin{enumerate}
\item Select $\cD\in\{\cL,\cU\}$ uniformly at random. \label{hat1}
\item Select $\Lambda$ according to $\nu_\cD$ and let $\bar\Lam=\bigcup_{A\in \Lambda}A$.\label{hat2}
\item If $\cD=\cL$, then select $W\subseteq \Int(\bar\Lam \cap \cL)$ with probability proportional to $\lam^{|W|-|N^{+}(W)|}$. 
If $\cD=\cU$, then select $W\subseteq \Int(\bar\Lam\cap \cU)$ with probability proportional to $\lam^{|W|-|N^{-}(W)|}$.\label{hat3}
\item If $\cD=\cL$, then let $U$ be a $\frac{\lam}{1+\lam}$-random subset of $\cU\backslash N(\bar\Lam)$. If $\cD=\cU$, then let $U$ be a $\frac{\lam}{1+\lam}$-random subset of $\cL\backslash N(\bar\Lam)$.\label{hat4}
\item Output the antichain $S=(\bar\Lam\backslash N(W))\cup W \cup U$.
\end{enumerate}

We say that an antichain $S$ is \emph{captured by $\hat \mu_\lam$ under $\cL$} (resp. $\cU$) if there is a positive probability that $S$ is selected by the above process given that $\cL$ (resp. $\cU$) is selected in Step~\ref{hat1}. Moreover we let $c(S)$ ($\in\{0, 1, 2\}$) denote the number of $\cD\in\{\cL,\ \cU\}$ such that $S$ is captured  by $\hat \mu_\lam$ under $\cD$.

We record the following observation about the measure $\hat \mu_\lam$.

\begin{lemma}\label{lem:cap}
Let $\cD$ denote the state chosen at Step~\ref{hat1} in the definition of $\hat \mu_\lam$. We have
\begin{align}\label{eqcapture0}
    \hat \mu_\lam(S \mid \cD=\cL)= \frac{\lam^{|S|}}{ (1+\lam)^{N} \Xi}\, \mathbf{1}_{\{S \text{  captured under $\cL$}\}} \, \, 
\end{align}
and similarly with $\cL$ replaced by $\cU$. In particular,
\begin{align}\label{eqcapture}
    \hat \mu_\lam(S)= c(S) \frac{\lam^{|S|}}{2 (1+\lam)^{N} \Xi}\, .
\end{align}
Moreover, $c(S)\geq 1$ for all antichains $S\subseteq C_n$.
\end{lemma}
\begin{proof}
Let $\cD$ denote the state chosen at Step~\ref{hat1} in the definition of $\hat \mu_\lam$. 
Suppose that $S$ is captured by $\hat \mu_\lam$ under $\cL$ and for $i\in [k-2,k+1]$ let $S_i=S\cap L_i$. Let $S'=N^{+}(S_{k-2})\cup S_{k-1}$ and let $\Lambda$ denote the collection of $2$-linked components of $S'$ (in the graph $L_{k-1} \cup L_k$) and 2-linked components of $S_{k+1}$ (in the graph $L_k \cup L_{k+1}$).
By the definition of $\hat \mu_\lam$, we have that
\begin{align*}
\hat\mu_\lam\left(S\mid \cD=\cL \right)
&= 
\frac{\prod_{A\in \Lambda} w_{\cL}(A)}{\Xi}
\cdot \frac{\lam^{|S_{k-2}|-|N^{+}(S_{k-2})|}}{\sum\limits_{B \subseteq \Int(S')} \lambda^{|B|-|N^+(B)|}}
\cdot
\frac{\lam^{|S_{k}|}}{(1+\lam)^{N-|N^{+}(S')|-|N^-(S_{k+1})|}}\, .
\end{align*}
Also,
\[\begin{split}
\prod_{A\in \Lambda} w_{\cL}(A)&=\prod_{A \in \Lambda, A \sub L_{k+1}} \lambda^{|A|}(1+\gl)^{-|N^-(A)|} \prod_{A \in \Lambda, A \sub \cL} \left(\gl^{|A|}(1+\gl)^{-|N^+(A)|}\sum_{B \sub \Int(A)}\gl^{|B|-|N^+(B)|}\right)\\
&= \lam^{|S'|+|S_{k+1}|} (1+\lam)^{-|N^+(S')|-|N^-(S_{k+1})|}\sum\limits_{B \subseteq \Int(S')} \lambda^{|B|-|N^+(B)|}
\end{split}\]
(note that, in the last equality, we are using the fact that if $B \sub {\rm Int}(S')$, then $B$ can be uniquely partitioned to $B_1, \ldots, B_\ell$ so that $B_i \sub {\rm Int}(A_i)$ for some $A_i \in \Lambda$ where, for all $i=1,\ldots, \ell$, $A_i$ are all distinct and $A_i \sub \cL$). Combining the above two equations, noting that $|S'|+|S_{k+1}|+|S_{k-2}|-|N^{+}(S_{k-2})|=|S|$, we have
\[
\hat \mu_\lam\left(S\mid \cD=\cL \right)= \frac{\lam^{|S|}}{(1+\lam)^N \Xi}\, ,
\]
as desired.
Similarly if $S$ is captured by $\hat \mu_\lam$ under $\cU$ then $\hat \mu_\lam\left(S\mid \cD=\cU \right)= \lam^{|S|}(1+\lam)^{-N}\Xi^{-1}$. Statement~\eqref{eqcapture} follows by noting that $\hat \mu_\lam(S)=\tfrac{1}{2}\hat \mu_\lam\left(S\mid \cD=\cL \right)+ \tfrac{1}{2}\hat \mu_\lam\left(S\mid \cD=\cU \right)$.

For the final claim, suppose now that $S\subseteq C_n$ is an antichain. Since the graph between layers $\cL=L_{k-1}$ and $\cU=L_{k}$ contains a perfect matching and $S$ is an antichain, we must have that either $|N^{+}(S_{k-2})\cup S_{k-1}|\leq N/2$ or $|N^{-}(S_{k+1})\cup S_{k}|\leq N/2$. By \Cref{isoper}, it follows that either $|N^{+}(A)|\geq (1+1/n)|A|$ for all $2$-linked components $A$ of $N^{+}(S_{k-2})\cup S_{k-1}$ or $|N^{-}(A)|\geq (1+1/n)|A|$ for all $2$-linked components of $N^{-}(S_{k+1})\cup S_{k}$ and so $S$ is captured by $\hat \mu_\lam$ under $\cL$ or $\cU$.
\end{proof}

We now aim to establish the following analogue of Lemma~\ref{lempolymeq}.  Recall that we let $N=\binom{n}{\lceil n/2 \rceil}=\binom{n}{k}$ and we let $\mu_{C_n,\lam}$ denote the hard-core measure on $C_n$ at activity $\lam$.

\begin{lemma}
    \label{defapprox}
    \[ \left| \ln  Z(C_n,\lam) - \ln \left[ 2(1+\lambda)^{N} \Xi \right] \right| = e^{-\Omega(N/n^5)}\, , \]
    and
    \[
    \|\hat \mu_\lam-\mu_{C_n,\lam}\|_{TV}= e^{-\Omega(N/n^5)}\, .
    \]
\end{lemma}

We begin by showing that the cluster expansion of $\ln \Xi$ is convergent (\textit{cf}.~\Cref{central.lem}).

\begin{lemma}\label{central.lemodd}
For $\cD\in\{\cL,\cU\}$, 
the polymer model $(\cP_\cD,\sim,w_\cD)$ satisfies the Koteck\'y-Preiss condition \eqref{KPbound1}, with $f$ and $g$ as in \eqref{eq:fgdef}. In particular, with $\cC=\cC(\cP_\cD,\sim)$, we have
\begin{equation}
    \label{tailboundscentralodd}
    \sum_{\Gamma\in \cC} \left|w_\cD(\Gamma)\right|e^{g(\Gamma)}\leq 2^n/n^3.
\end{equation}
\end{lemma}
\begin{proof}
By symmetry we may assume that $\cD=\cL$.
The proof is now identical to~\Cref{central.lem} where we replace each instance of $w_C$ with $w_\cL$.
\end{proof}

We record the following corollary of Lemma~\ref{central.lemodd} which is an analog of \Cref{largedeviation}. We omit the proof since it is identical to that of \Cref{largedeviation}.

\begin{lemma}
    \label{smalldev}
    Let $\cD\in\{\cL,\cU\}$.
       If $ \mathbf{\Lam}$ is sampled according to $\nu_\cD$, then
 \[ | \mathbf{\bar\Lam} | \leq N/n^2,\]
    with probability at least $1-\exp\{\Omega(N/n^5)\}$.
\end{lemma}

We record one final lemma before turning to the proof of Lemma~\ref{defapprox}.

\begin{lemma}\label{lem:cap2} We have
\[
\hat \mu_\lam\left(\{ S: c(S)=2\}\right)=e^{-\Omega(N/n^5)}\, .
\]
\end{lemma}
\begin{proof}
Let $\cE$ denote the event that $c(S)=2$ \emph{and} $|S\cap \cL|\geq |S\cap \cU|$. By symmetry it is enough to show that $\hat \mu_\lam(\cE)= e^{-\Omega(N/n^5)}$. Let $\cD$ denote the state chosen at Step~\ref{hat1} in the definition of $\hat \mu_\lam$. Note that, by \eqref{eqcapture0}, we have $\hat \mu_\lam(\cE \mid \cD=\cL)=\hat \mu_\lam(\cE \mid \cD=\cU)$ and thus $\hat \mu_\lam(\cE)=\hat \mu_\lam(\cE \mid \cD=\cL)$. Let $\Lambda$ denote the configuration chosen at Step~\ref{hat2}. By Lemma~\ref{smalldev},
\[
\hat \mu_\lam(\cE \mid \cD=\cL) = \hat \mu_\lam(\cE \mid \cD=\cL, |\bar\Lambda|\leq N/n^2)+e^{-\Omega(N/n^5)}\, .
\]
Next note that conditioned on the event $\cD=\cL$ and on $\Lambda$, $S\cap \cU$ is a random subset of $\cU\backslash N(\Lambda)$ where each element is selected independently with probability $\lam/(1+\lam)$ (Step~\ref{hat4} in the definition of $\hat \mu_\lam$). Moreover if $|\bar\Lambda|\leq N/n^2$ then $|N(\bar\Lambda)|\leq N/n$. Letting $X\sim \Bin((1-1/n)N,\lam/(1+\lam))$, it follows from Chernoff's bound (recall that $\gl=\tilde \gO(1/\sqrt n))$ that
\[
\hat \mu_\lam\left( |S\cap \cL|\geq |S\cap \cU| \mid \cD=\cL, |\bar\Lambda|\leq N/n^2 \right) \leq \P(X\leq N/n^2)=e^{-\Omega(N\lam)}.
\] The result now follows.
\end{proof}

We finally turn to the proof of Lemma~\ref{defapprox}.

\begin{proof}[Proof of Lemma~\ref{defapprox}.]
By \Cref{lem:cap},
\begin{align}\label{eq:ZleqXi}
Z(C_n,\lam)=\sum_{S\subseteq C_n, S\in \cA}\lam^{|S|}\leq \sum_{S\subseteq C_n, S\in \cA}c(S)\lam^{|S|}=2(1+\lam)^N\Xi\, .
\end{align}
On the other hand, by \Cref{lem:cap} and then \Cref{lem:cap2},
\[
Z(C_n,\lam)\geq 2(1+\lam)^N\Xi - \sum_{S:c(S)=2}c(S)\lam^{|S|}=(1-e^{-\Omega(N/n^5)}) 2(1+\lam)^N\Xi\, .
\]
The first claim of Lemma~\ref{defapprox} follows. For the second, note that if $\hat \mu_\lam(S)>\mu_{C_n,\lam}(S)$ then by \Cref{lem:cap} and~\eqref{eq:ZleqXi}, we must have that $c(S)=2$. It follows that
\begin{align*}
    \|\hat \mu_\lam-\mu_{C_n,\lam}\|_{TV}
    = \sum_{S: \hat \mu_\lam(S)>\mu_{C_n,\lam}(S)}\hat \mu_\lam(S)-\mu_{C_n,\lam}(S)\leq
    \sum_{S:c(S)=2}\hat \mu_\lam(S)\, .
\end{align*}
The result now follows from~\Cref{lem:cap2}.
\end{proof}

\subsection{A final truncation step}\label{secfintrunc}

In the previous section we approximated $Z(C_n,\lam)$ by (a scaling of) the partition function $\Xi$. In this section we show that $\Xi$ is well-approximated by a simplified polymer model partition function which encodes antichains contained in three consecutive layers (rather than four).  

 Consider the polymer model $(\cP_\cL,\sim,w_\cL')$\index{$(\cP_\cL,\sim,w_\cL')$} where $w_\cL'$ is the simplified weight 
\[
w'_\cL(A)=
\begin{cases}
\lambda^{|A|}(1+\lambda)^{-|N^-(A)|}
& \text{ if $A \subseteq L_{k+1}$,} \\
\lambda^{|A|}(1+\lambda)^{-|N^+(A)|} & \text{if $A \subseteq \cL$ .} 
\end{cases}
\]
Observe that $(\cP_\cL,\sim,w_\cL')$ is precisely the polymer model $(\cP_C,\sim,w_C)$ from \Cref{secthreelayer}. It will be useful however to keep the notation $(\cP_\cL,\sim,w_\cL')$, as we will define $(\cP_\cU,\sim,w_\cU')$\index{$(\cP_\cU,\sim,w_\cU')$} analogously. We note that $\Xi(\cP_\cU,\sim,w_\cU')=\Xi(\cP_\cL,\sim,w_\cL')$ and we denote this partition function by $\Xi_C$ as in \Cref{secthreelayer}. For $\cD\in\{\cL,\cU\}$, let $\nu'_\cD$\index{$\nu'_\cD$ on $\Omega_\cD$} denote the measure on $\Omega_\cD$ given by
\begin{equation*}
    \nu'_\cD(\Lam)=\frac{\prod_{A\in \Lam}w'_\cD(A)}{\Xi_C} \text{\,  for\,  } \Lam\in \Omega_\cD.
\end{equation*}

\begin{lemma} \label{lem:XiXiprime}
\begin{align}\label{eq:XiXiprime}
\left|\ln \Xi - \ln \Xi_C \right|=o\left(\frac{1}{\sqrt{N}}\right)
\end{align}
and for $\cD\in\{\cL,\cU\}$,
\begin{align}\label{eq:XiXiprime2}
\|\nu_\cD-\nu'_\cD\|_{TV}=o\left(\frac{1}{\sqrt{N}}\right)\, .
\end{align}
 Moreover, when $\lam=1$ we can tighten both $o( 1/\sqrt{N} )$ error terms above to $e^{-\Omega(n^2)}.$
\end{lemma}
\begin{proof}
Recall from Section~\ref{truncation} that for $X, Y\subseteq B_n$, we let $Y^X=\{v\in Y  \mid v \not\subset w \ \forall w\in X \}$. 
Recall also that $\cA_{k}$ denotes the set of all antichains $S\subseteq B_n$ such that each $2$-linked component $A$ of $S\cap L_{k-1}$ satisfies $|N^{+}(A)|\geq (1+1/n)|A|$, and for $U\subseteq B_n$, we let 
\[
Z_{k}(U, \lam)= \sum_{A \in \cA_{k}, \ A \subseteq U}\lam^{|A|}\, .
\]

\begin{claim}\label{claim:polypart}
\[
(1+\lam)^{N}\Xi= \sum_{X\subseteq L_{k+1}} \lam^{|X|}Z_{k}(L^X_{[k-2, k]}, \lam)\, .
\]
\end{claim}
\begin{subproof}
Given $X\subseteq L_{k+1}$, let $(\cP_{X,{k}},\sim,w)$ be the polymer model defined in Section~\ref{secOutlineTrunc}. In particular $\cP_{X,{k}}$ is  the set of all $2$-linked
 subsets $A\subseteq L^X_{k-1}$ such that $N^{+}(A)\geq (1+1/n)|A|$ and $w(A)=w_{\cL}(A)$ for all $A\in \cP_{X,{k}}$. Letting $\Xi_X=\Xi(\cP_{X,{k}},\sim,w)$, Lemma~\ref{lem:polymer-antichain} shows that 
\beq{eq:polypartclaim}
(1+\lam)^{M}\cdot \Xi_X=  Z_{k}(L^X_{[k-2,k]},\lam)\, ,
\enq
where $M=|L^X_{k}|=N-|N^{-}(X)|$\, . 

Recall that $\Omega_\cL$ denotes the set of all collections of compatible polymers in $\cP_\cL$ and let $\Omega'_\cL\subseteq \Omega_\cL$\index{$\Omega'_\cL$} denote the set of $\Lambda\in \Omega_\cL$ where each polymer $A\in \Lambda$ is contained in $L_{k+1}$. Given $\Lambda\in \Omega'_\cL$, let $\Omega_\Lambda\subseteq \Omega_\cL$ denote the set of $\Lambda'\in\Omega_\cL$ where each polymer $A\in \Lambda'$ is contained in $\cL$ and $\Lambda\cup\Lambda'\in \Omega_\cL$. Crucially, if $\Lambda\in\Omega'_\cL$ where $X=\bigcup_{A\in \Lambda}A$, then
\[\text{$\Omega_\Lambda$ is in one-to-one correspondence with collections of compatible polymers in $\cP_{X,k}$.}\]
We then have 
\begin{align*}
\Xi= \sum_{\Lambda\in \Omega_\cL}\prod_{A\in\Lambda}w_\cL(A)
&=
\sum_{\Lambda\in \Omega'_\cL}\prod_{A\in\Lambda'}w_\cL(A)\sum_{\Lambda'\in \Omega_{\Lambda}}\prod_{A\in\Lambda'}w_\cL(A)\\
&=
\sum_{\Lambda\in \Omega'_\cL}\prod_{A\in\Lambda'}\lam^{|A|}(1+\lam)^{-|N^{-}(A)|}\sum_{\Lambda'\in \Omega_{\Lambda}}\prod_{A\in\Lambda'}w_\cL(A)\\
&=
\sum_{X\subseteq L_{k+2}}\lam^{|X|}(1+\lam)^{-|N^{-}(X)|} \Xi_X\, ,
\end{align*}
where for the final equality we used the one-to-one correspondence between $\Omega'_\cL$ and subsets $X\subseteq L_{k+1}$. The claim follows by combining the identity above with~\eqref{eq:polypartclaim}.
\end{subproof}
By an argument identical to the proof of Claim~\ref{claim:polypart} (replacing the weight $w$ with the simplified weight $w'(A)=\lam^{|A|}(1+\lam)^{-|N^{+}(A)|}$) we also have
\[
(1+\lam)^{N}\Xi_C= \sum_{X\subseteq L_{k+1}} \lam^{|X|}Z_{k}(L^X_{[k-1, k]}, \lam)\, .
\]
\Cref{eq:XiXiprime} now follows from Lemma~\ref{trunclemma2}.

For \cref{eq:XiXiprime2} note that  $w_\cL(A)\geq w'_\cL(A)$ for all $A \in \cP_{\cL}$ and so $\Xi\geq \Xi_C$. Moreover, $w_\cL(A)> w'_\cL(A)$ if and only if $A\subseteq \cL$ and $A$ has non-empty interior. It follows that
\[
\Omega_1:=\{\Lambda: \nu_\cL(\Lambda)> \nu'_\cL(\Lambda) \} \subseteq \{\Lambda: A\subseteq \cL, \Int(A)\neq \emptyset \text{ for some }A\in\Lambda \}=:\Omega_2\, .
\]
Moreover, \[
\Xi-\Xi_C= \sum_{\Lambda\in\Omega_2}
\left(\prod_{A\in \Lam}w_\cL(A)-\prod_{A\in \Lam} w'_\cL(A)\right)\, .
\]
 We then have
\begin{align*}
\|\nu_\cL-\nu'_\cL\|_{TV}& =
\sum_{\Lambda\in\Omega_1}\nu_\cL(\Lambda)-\nu'_\cL(\Lambda)\\
&\leq \sum_{\Lambda\in\Omega_1}
\left(\frac{\prod_{A\in \Lam}w_\cL(A)}{\Xi}-\frac{\prod_{A\in \Lam}w'_\cL(A)}{\Xi} \right)
\\
&\leq 
\frac{1}{\Xi}  \sum_{\Lambda\in\Omega_2}
\left(\prod_{A\in \Lam}w_\cL(A)-\prod_{A\in \Lam} w'_\cL(A) \right)\\
&=1-\Xi_C/\Xi\, .
\end{align*}
Statement~\eqref{eq:XiXiprime2} now follows from~\eqref{eq:XiXiprime}.
\end{proof}

Finally, it will be useful to consider the following (simpler) analogue of the measure $\hat \mu_\lam$: let $\mu'_\lam$\index{$\mu'_\lam$} denote the distribution of the random antichain produced by the following process:

\begin{enumerate}
\item Select $\cD\in\{\cL,\ \cU\}$ uniformly at random. \label{hatp1}
\item Select $\Lambda$ according to $\nu'_\cD$ and let $\bar\Lam=\bigcup_{A\in \Lambda}A$.\label{hatp2}
\item If $\cD=\cL$ let $U$ be a $\frac{\lam}{1+\lam}$-random subset of $\cU\backslash N(\bar\Lam)$. If $\cD=\cU$ let $U$ be a $\frac{\lam}{1+\lam}$-random subset of $\cL\backslash N(\bar\Lam)$.\label{hatp4}
\item Output the antichain $S=\bar\Lam\cup U$.
\end{enumerate}

 Note that $\mu_\lam'$ conditioned on the event that $\cL$ is chosen at Step~\ref{hatp1} is precisely the measure $\mu_\lam^C$ introduced in \Cref{givensize}.

 We say that an antichain $S$ is \emph{captured} by $\mu'_\lam$ under $\cL$ (resp. $\cU$) if there is a positive probability that $S$ is selected by the above process given that $\cL$ (resp. $\cU$) is selected in Step~\ref{hat1}. Moreover we let $c'(S)$ denote the number of $\cD\in\{\cL,\ \cU\}$ (i.e. $0,1$ or $2$) such that $S$ is captured by $\mu'_\lam$ under $\cD$. We record the following analogues of Lemma~\ref{lem:cap} and~\ref{lem:cap2}. The proofs are analogous (and simpler) and so we omit them. 

 \begin{lemma}\label{lem:capp}
If $S\subseteq C_n$ is an antichain then
\begin{align}
     \mu'_\lam(S)= c'(S) \frac{\lam^{|S|}}{2 (1+\lam)^{N} \Xi_C}\, .
\end{align}
\end{lemma}

\begin{lemma}\label{lem:capp2} We have
\[
 \mu'_\lam\left(\{ S: c'(S)=2\}\right)=e^{-\Omega(N/n^5)}\, .
\]
\end{lemma}

We record the following corollary to \Cref{defapprox,lem:XiXiprime,lem:capp,lem:capp2}. Recall that we let $Z(\lam)=Z(B_n,\lam)$ and let $\mu_\lam=\mu_{B_n,\lam}$ denote the hardcore measure on antichains in $B_n$. 

\begin{corollary}
 \label{defapprox2}
    \[ \left| \ln  Z(\lam) - \ln \left[ 2(1+\lambda)^{N} \Xi_C \right] \right| = o\left(\frac{1}{\sqrt{N}}\right)\, , \]
    and 
    \[
    \|\mu'_\lam-\mu_{\lam}\|_{TV}=  o\left(\frac{1}{\sqrt{N}}\right)\, .
    \]
     Moreover, when $\lam=1$ we can tighten both $o( 1/\sqrt{N} )$ error terms above to $e^{-\Omega(n^2)}.$
\end{corollary}
\begin{proof}
The first claim follows by combining \Cref{trunc} and Lemmas~\ref{defapprox} and \ref{lem:XiXiprime} via the triangle inequality. 

Next note that if $S$ is an antichain such that $\mu'_\lam(S)>\hat \mu_\lam(S)$ then by Lemmas~\ref{lem:cap} and~\ref{lem:capp} and the fact that $\Xi=(1+o(1))\Xi_C$ (Lemma~\ref{lem:XiXiprime}) it must be the case that either $c(S)=c'(S)=1$ or $c'(S)=2$. Thus,
\begin{align*}
   \|\hat\mu_\lam-\mu'_\lam\|_{TV}&=\sum_{S: \mu'_\lam(S)>\hat \mu_\lam(S)} \mu'_\lam(S)-\hat \mu_\lam(S)\\
   &\leq \sum_{S: c(S)=c'(S)=1}\left| \frac{\lam^{|S|}}{2 (1+\lam)^{N} \Xi_C}- \frac{\lam^{|S|}}{2 (1+\lam)^{N} \Xi} \right|
   + \mu'_\lam\left(\{ S: c'(S)=2\}\right)\\
   &=o\left(\frac{1}{\sqrt{N}}\right)\, ,
\end{align*}
where for the first inequlaity we used Lemmas~\ref{lem:cap} and~\ref{lem:capp} and for the final equality we used Lemma~\ref{lem:XiXiprime} and~\ref{lem:capp2}. 

Moreover $\|\hat\mu_\lam-\mu_{C_n,\lam}\|_{TV}=o(1/\sqrt{N})$ by \Cref{defapprox} and $\|\mu_\lam-\mu_{C_n,\lam}\|_{TV}=o(1/\sqrt{N})$ by \Cref{trunc}. The second claim now follows by the triangle inequality. 

 Finally we note that when $\lam=1$, the $o(1/\sqrt{N})$ errors can be improved to $e^{-\Omega(n^2)}$ (when $\lam=1$ we apply \Cref{saporefined} in place of \Cref{trunc}).
\end{proof}

We are now in a position to prove our main results in the case of $n$ odd.

\subsection{Theorem~\ref{saporefined2}}
Theorem~\ref{saporefined2} follows by combining \Cref{clusterexpansion}, \Cref{defapprox2} (both with $\lam=1$) and setting 
\[P_j^1(n)=2^{-j}S_j^1(n,1)\]
and
\[P_j^2(n)=2^{-j}S_j^2(n,1).\]

\subsection{Theorem \ref{sparsesperner3} in the case of $\beta$ bounded away from $1$.}

As in Sections~\ref{sparsespernersection} we first assume that $\beta\leq 9/10$. 
Throughout this section we assume that $\lam\leq 10$ and $\lam\geq C\log^2 n/ \sqrt{n}$ for $C$ sufficiently large.

We follow a similar strategy to the case of $n$ even. In the case where $n$ is even we approximated $i_m(B_n)$ by $i_m(C_n)$ and obtained asymptotics for the latter by studying a polymer model representation of the hard-core model $\mu_{C_n,\lam}$. The main difference in the case where $n$ is odd is that we now approximate $i_m(B_n)$ by the number of antichains of size $m$ in the support of $\mu'_\lam$ (which is a close approximation to $\mu_{C_n,\lam}$). We let $\cI^\cL_m$ denote the set of antichains $S$ of size $m$ that are captured by $\mu'_\lam$ under $\cL$. We define $\cI^\cU_m$ analogously. We note that $\cI^\cL_m$ is precisely the set of antichains of size $m$ in the support of the measure $\mu_\lam^C$ defined in \Cref{givensize} and so $|\cI^\cL_m|=i_m^C$. Note also that $|\cI^\cU_m|=i_m^C$ by symmetry. 

We now proceed by proving the following result which will immediately implies \Cref{sparsesperner3} (for $\beta\leq 9/10$).

\begin{lemma} 
\label{directtransf}
    For any $m\in [N]$ with $2C\log^2 n/\sqrt{n}<\frac{m}{N}\leq 9/10$, we have
	\[ i_m( B_n)=(2+o(1))i_m^C\]
    as $n\rightarrow \infty$ and $n$ is odd. Moreover, almost all antichains in $B_n$ of size $m$ are contained in $L_{[k-2,k]}$ or $L_{[k-1,k+1]}$.
\end{lemma}
\begin{proof}
Let $\mathbf A$ denote a sample from the measure $\mu^C_\lam$. As in \Cref{sparsespernersection} there exists $\lam{\in \left(\frac{C \log^2 n}{ \sqrt{n}}, 10\right)}$ such that $\mathbb{E}| \mathbf{A}|=m$. By~\Cref{lem16}, $i_m^C=\Theta(1/\sqrt{\lambda N})\frac{(1+\gl)^N\Xi_C}{\gl^m}.$

Let $\cI_m$ denote the set of all antichains of size $m$ in $B_n$. Note that by~\Cref{lem:capp},
    \[ |\mu_\lam(\cI_m) - \mu'_\lam(\cI_m)| = \left| \frac{i_m(B_n)\lam^m}{Z(\lam)}-\frac{(|\cI_m^\cL|+|\cI_m^\cU|)\lam^m}{2(1+\lam)^{N}\Xi_C} \right|=\left| \frac{i_m(B_n)\lam^m}{Z(\lam)}-\frac{i_m^C\lam^m}{(1+\lam)^{N}\Xi_C} \right|.\]
On the other hand, the above quantity is  $o(1/\sqrt N)$ by \Cref{defapprox2}.
Again, by \Cref{defapprox2}, it follows that
    \[ |i_m(B_n)-(2+o(1))i_m^C|=o(1/\sqrt{N})\times \frac{(1+\lam)^{N}\Xi_C}{\lam^m}=o(i_m^C). \]
    For the final claim, we note that if $S\in \cI_m^\cL \cap \cI_m^\cU$ then $c'(S)=2$ where $c'(S)$ is as in \Cref{lem:capp}. By \Cref{lem:capp2} and the fact that $\|\hat\mu_\lam-\mu'_\lam\|_{TV}=o(1/\sqrt{N})$ we have
    \[
    \mu'_\lam(\cI_m^\cL \cap \cI_m^\cU)=\frac{|\cI_m^\cL \cap \cI_m^\cU|\lam^m}{(1+\lam)^{N}\Xi_C} =e^{-\Omega(N/n^5)}\, .
    \]
    We conclude that $|\cI_m^\cL \cap \cI_m^\cU|=o(i_m^C)$ and so 
    \[
    |\cI_m^\cL \cup \cI_m^\cU|=(1+o(1))(|\cI_m^\cL|+|\cI_m^\cU|)=(2+o(1))i_m^C=(1+o(1))i_m(B_n).
    \]
    The result follows by noting that $\cI_m^\cL\subseteq L_{[k,k+2]}$ and $\cI_m^\cU\subseteq L_{[k-1,k+1]}$.
\end{proof}

\subsection{Theorem \ref{thm5} in the case of $\beta$ bounded away from $1$.}
As in the previous section, we assume that $\beta\leq 9/10$. 

Fix $t\geq 1$, $\eps>0$ and suppose $m\in[N]$ is such that $\beta=m/N\geq 1-4^{-1/t}+\eps$. By \Cref{directtransf},  $\psi(n,m)=i_m(B_n)=(2+o(1))i_m^C$. Theorem \ref{thm5} (for $\beta\leq 9/10$) now follows from \Cref{thm5pre}.

\subsection{Theorems~\ref{thm5} and \ref{sparsesperner3} in the case of $\beta$ close to $1$}\label{closeto1'}
As in the case of $n$ even, we give a direct proof when $\beta$ is close to $1$. The argument is a minor modification of the one given in~\Cref{closeto1}.

Let $\cA'\subseteq\cA$ denote the set of antichains that are a subset of either $\cL$ or $\cU$. 
\Cref{clusterexpansion} (applied with $t=1$) and \Cref{defapprox2}  show that for $\lam\geq 9$,
\begin{align}\label{eqBnLU}
Z(B_n, \lam)=(1+o(1/\sqrt{N}))2(1+\lam)^N =(1+o(1/\sqrt{N}))\sum_{A\in \cA'}\lam^{|A|}\, .
\end{align}
Set $\beta=m/N$ and suppose that $\beta\geq 9/10 $ and set $\gl:=\frac{\beta}{1-\beta}$ (so that $\lam\geq 9$).
Let $f_m$ denote the number of antichains in $\cA\backslash \cA'$ of size $m$. Then by~\eqref{eqBnLU} (recalling that $Z(B_n,\lam)=\sum_{A\in\cA}\lam^{|A|}$)
\[
f_m \lam^m \leq \sum_{A\in \cA\backslash\cA'}\lam^{|A|}=o(1/\sqrt{N})(1+\lam)^N.
\]
As we showed in \Cref{closeto1}, we have $\binom{N}{m} \lam^m(1+\lam)^{-N} =\Omega( 1/\sqrt{N})$ and so $f_m=o(1)\binom{N}{m}$. Since the number of antichains in $\cA'$ of size $m$ is equal to $2\binom{N}{m}$, we conclude that 
\[
2\binom{N}{m}\leq \psi(n,m) = 2\binom{N}{m}+f_m=(2+o(1))\binom{N}{m}\, ,
\]
which is the content of Theorem~\ref{thm5} in this regime. We have shown that almost all antichains of size $m$ are contained in either $\cL$ or $\cU$ and so Theorem~\ref{sparsesperner3} trivially holds in this regime.

\subsection{Theorem \ref{dualthm1}}
\subsubsection{1-statement and behaviour inside the scaling window}
We proceed in much the same way as in the case of $n$ even (\Cref{1statement}).
Let $\beta=m/N=\frac{3}{4}-\frac{\ln(n)}{4n} + \frac{h(n)}{n}$. Theorem \ref{thm5} shows that if $h(n)=\Omega(n)>0$ then $\psi(n,m)=i_m(B_n)=(2+o(1))\binom{N}{m}$. For the 1-statement and behaviour inside the scaling window, we may therefore assume that $h(n)=o(n)$. Let $\lam$ be as in Lemma~\ref{lem17} applied with $t=2$ i.e. $\lam=\beta/(1-\beta)$. 
Combining~\eqref{eqs0s1},~\eqref{eq:psit1},~\eqref{eq:psit2} (again with $t=2$), and \Cref{directtransf} gives
\begin{align}\label{eq:psinm3odd}
\psi(n, m)= (2+o(1)) \binom{N}{m} \exp\left( \sum_{\Gamma\in \mathcal{C}, \|\Gamma\|=1} w_C(\Gamma) \right)\, ,
\end{align}
where $\cC=\cC(\cP_C,\sim)$.
We now turn our attention to computing the sum on the RHS.

Recall that $k=(n+1)/2$. Let $\Gamma\in \cC$ be a cluster of size 1. Then $\Gamma$ consists of precisely one polymer $A$ which has size 1, i.e. it is an element of $L_{k-1}\cup L_{k+1}$. In particular, the incompatibility graph $G_\Gamma$ consists of a single vertex and so $\phi(G_\Gamma)=1$ and  $w_C(\Gamma)=w_C(A)=\lambda (1+\lambda)^{-(k+1)}$ if $A\in L_{k+1}$, $w_C(\Gamma)=w_C(A)=\lambda (1+\lambda)^{-k}$ if $A\in L_{k-1}$. Therefore, recalling that $N=\binom{n}{k}=\binom{n}{k-1}$,
\begin{align}\label{1clusters3}
    \sum\limits_{\|\Gamma \|=1} w_C(\Gamma) 
    &= \binom{n}{k-1} \lambda (1+\lambda)^{-k} + \binom{n}{k+1} \lambda (1+\lambda)^{-(k+1)}\nonumber\\
    &=N\frac{\beta}{\sqrt{1-\beta}} \left[1 +\frac{n-1}{n+3}(1-\beta)\right] (1-\beta)^{n/2}  \\
     &=(1+o(1))\frac{2^n}{\sqrt{\pi n/2}}\cdot \frac{3}{2}\cdot \frac{5}{4}\left(\frac{1+\frac{\ln n}{n}-\frac{4h(n)}{n}}{4}\right)^{n/2}\nonumber\\
 &= (1+o(1))\frac{15}{4\sqrt{2\pi n}} 
 \exp\left\{\frac{\ln n}{2}-2h(n)+O\left(\frac{(\ln n)^2+(h(n))^2}{n}\right)\right\}\nonumber\\
 &=(1+o(1))\frac{15}{4\sqrt{2\pi }} 
 \exp\left\{-2h(n)(1+o(1))\right\}\, ,\nonumber
\end{align}
where for last two equalities we used the assumption that $h(n)=o(n)$.
Considering the cases $h(n)\to\infty$ and $h(n)\to c\in\mathbb R$ recovers the first two statements of Theorem \ref{dualthm1} in the case of $n$ odd.

\begin{remark}\label{rem:R11}
Theorem \ref{thm5} shows that for $\beta=m/N\geq 1/2+\eps$
\[
    \psi(n,m)=(2+o(1))\binom{N}{m}\exp \left( N R_1^{1}\left(n, \beta \right) \cdot \left(1-\beta\right)^{\frac{n+1}{2}} \right)\, .
    \]
    Comparing this to~\eqref{eq:psinm3odd} and~\eqref{1clusters3} reveals that $R_1^{1}\left(n, \beta \right)=\frac{\beta}{1-\beta}+\frac{n-1}{n+3}\beta$.
\end{remark}

\subsubsection{$0$-statement}
The proof of the $0$-statement goes through exactly as in \Cref{0statement}. This completes the proof of \Cref{dualthm1}.

\section{Proof of Theorem~\ref{MT}}
\label{secproofcont} In this section, the \textit{cost} (of a choice) means the logarithm (in base two) of the number of its possibilities. Recall that $\log=\log_2$. 

Recall that \Cref{MT} is key to proving \Cref{trunclemma2} which asserts that the ($\gl$-weighted) number of antichains in three consecutive layers of $B_n$ is closely approximated by the number contained solely in the largest two of the three layers. In his work on Dedekind's problem, Sapozhenko~\cite{sapo} proves a similar result in the special of $\lam=1$. His approach is an early application of the graph container method. Trying to follow this approach naively, however, fails immediately for $\lam< 1$.

It turns out that the natural way to overcome this difficulty is to consider a \emph{pair} of containers (that approximates antichains in three layers) collectively, and analyze them according to their interplay. This enables us to reduce the cost for both constructing containers and recovering antichains from them. The construction of the pairs of containers partially adapts a recent improvement on the graph container method by the authors \cite{lambda} to increase the valid range of $\lambda$ to $\lam=\tilde \Omega(n^{-1/2})$ (it would have been $\tilde \Omega(n^{-1/3})$ without the new idea in \cite{lambda}). Our analysis for the reconstruction steps is inspired by Galvin \cite{Galvin_Zhom}, who also considered a pair of containers in bipartite graphs (as opposed to a poset).

For future reference, we record the basic binomial estimate 
\beq{binom} {n \choose \le k} \le \exp\left(k \log \frac{en}{k}\right) \enq
where $\binom{n}{\leq k}:=\sum_{i \le k}\binom{n}{i}.$ 

\subsection{Preliminaries}\label{sec.container.prelim} This section quickly collects preliminaries on graph containers from the literature. We refer readers to \cite{Galvin_independent} for a less hurried introduction to the topic. Say a bipartite graph $\Sigma$ with bipartition $X \cup Y$ is \textit{$(d_X, d_Y)$-regular} if $d(v)=d_X~ \forall v \in X$ and $d(v)=d_Y~ \forall v \in Y$.  For $A\subseteq X$, we define the \textit{closure} of $A$ to be 
$[A]\coloneqq \{ v\in X: N(v)\subseteq N(A)\}$. For integers $a,g\geq 1$, let
    \begin{align}\label{eq:cGdef}
    \cG=\cG(a,g):=\{A \sub X: A \mbox{ is $2$-linked}, |[A]|=a, |N(A)|=g\}.
    \end{align}
Classical graph container methods can be used to bound the size of $\cG$. As the first step, one constructs a collection of `containers' (\Cref{def.psi.approx}) so that each member of $\cG$ is contained in at least one of the containers. In order to provide a good upper bound on $|\cG|$, the construction of containers needs to meet the following two conditions: i) each container is not too large (\Cref{SF.size}); and ii) there are not too many containers (Lemmas~\ref{lem1} and \ref{lem2}).

Throughout this section we fix a bipartite graph $\Sigma$ with bipartition $X \cup Y$ that is $(d_X, d_Y)$-regular with $d_Y \le d_X$. We also fix $a,g\in\mathbb N$ and set $t=g-a$.

\begin{defn}\label{def.v.approx} For $A\subseteq X$ and $\varphi>0$, let
\[
N(A)^{\varphi}\coloneqq \{y\in N(A) : |N_{[A]}(y)|>\varphi\} .
\]
A $\varphi$-approximation for $A\subseteq X$ is a set $F\subseteq Y$ satisfying
\[
N(A)^{\varphi}\subseteq F\subseteq N(A) \text{ and }N(F)\supseteq [A]\, .
\]
\end{defn}

\begin{defn}\label{def.psi.approx}
For $1 \le \psi_X\leq d_X-1$ and $1\leq \psi_Y \le d_Y-1$, a \textit{$(\psi_X, \psi_Y)$-approximating pair} for $A \sub X$ is a pair $(F,S) \in 2^Y \times 2^X$ satisfying $F \sub N(A)$, $S \supseteq [A]$,
    \beq{S.degree} d_F(v) \ge d_X-\psi_X \quad \forall v \in S, \text{ and}\enq
    \beq{F.degree} d_{X \setminus S}(v) \ge d_Y-\psi_Y \quad \forall v \in Y \setminus F.\enq
\end{defn}

\begin{proposition}[\!\!{\cite[Proposition 2.3 and equation (22)]{lambda}}]\label{SF.size} If $(F,S)$ is a $(\psi_X, \psi_Y)$-approximating pair for $A \in \cG(a,g)$, then, with $s:=|S|$ and $f:=|F|,$ we have
    \beq{eq.SF} s \le f+\frac{1}{d_X}[(g-f)\psi_Y+(s-a)\psi_X].\enq
    In particular, $s\leq 2g.$

\end{proposition}
Let
\beq{m.varphi} m_\varphi=\min\{|N(K)|:v \in Y, K \sub N(v), |K|>\varphi\}.\enq
The following lemma provides a collection of $\varphi$-approximations for the set $\cG(a,g)$.

\begin{lemma}[\!\!{\cite[Lemma A.3]{BalGLW21}}]\label{lem1}
Let $1 \le \varphi \le d_Y-1$ and set $w=gd_Y-ad_X$. 
For any $C>0$ that satisfies $p=C\ln d_X/(\varphi d_X)<1$, the following holds: there exists a family $\cV =\cV(a,g,\varphi) \sub 2^Y$ of size at most
\beq{eq.first.approx} |Y|\exp\left(54g \ln(d_Xd_Y)p+\frac{54 g\ln(d_Xd_Y)}{\exp(p m_\varphi)}+\frac{54w\ln d_Y \ln (d_Xd_Y)}{d_X(d_Y-\varphi)}\right)\binom{3gd_Yp}{\le 3wp}\enq
such that each $A \in \cG(a, g)$ has a $\varphi$-approximation in $\cV$.    
\end{lemma}

The following lemma allows us to construct $(\psi_X,\psi_Y)$-approximations from $\varphi$-approximations. 

\begin{lemma}[\!\!{\cite[Lemma 2.4]{lambda}}] \label{lem2}
There exists an absolute constant $\gamma\geq 2$ such that the following holds. Set $\varphi=d_Y/2$. Given a set $F'\subseteq Y$, write $\cG(a,g,F')$ for the collection of $A \in \cG(a,g)$ for which $F'$ is a $\varphi$-approximation. Then there exists a family $\mathcal{F}=\mathcal{F}(a, g, F')\subseteq 2^Y \times 2^X$ that satisfies the following properties: with
\beq{kap} \kappa:=\max\{i:2^i \le \sqrt{d_X}\},\enq
we have that
\begin{enumerate}
    \item\label{bdcfi} $\cF=\bigcup_{i \in [0, \kappa]} \cF_i \left(=\bigcup_{i \in [0,\kappa]}\cF_i(a,g,F')\right)$ where for each $i \in [0,\kappa]$,
    \[|\cF_i|\le {d_Yg \choose \le 2\gamma w/(d_Xd_Y)}{ d_X^2 d_Y^2 g \choose \le 2\gamma w/(d_Xd_Y)} \exp\left[ O\left(\frac{it\gamma \log(gd_X^2/t)}{d_X}\right)\right];\]
    \item if $(F,S)\in\cF_i$, then $(F,S)$ is a $(d_X/(2^i\gamma), d_Y/\gamma)$-approximating pair for some $A \in \cG(a,g,F')$; in addition, for $i \in [0, \kappa-1]$, every $(F, S) \in \cF_i$ satisfies $|F|<g-t/2^i$; and
    \item every $A \in \cG(a,g,F')$ has a  $(d_X/(2^i\gamma), d_Y/\gamma)$-approximating pair in $\cF_i$ for some $i \in [0, \kappa]$.
\end{enumerate}
    
\end{lemma}

In fact, for our applications of \Cref{lem2} in what follows, instead of the bounds on the individual $|\cF_i|$ given in \cref{bdcfi} above, we will only make use of the following immediate corollary that bounds $|\cF|$: if $t=\Omega(d_X)$ then since $\kappa \le \log d_X$, we have
\beq{bdcf} |\cF|\leq \sum_{i \in [0, \kappa]} |\cF_i| \le {d_Yg \choose \le 2\gamma w/(d_Xd_Y)}{ d_X^2 d_Y^2 g \choose \le 2\gamma w/(d_Xd_Y)} \exp\left[ O\left(\frac{t\gamma \log d_X \log(gd_X^2/t)}{d_X}\right)\right].\enq

We will also need the following variant of \Cref{lem1}. For $b,h\in \mathbb N$ and $S \sub Y$ with $|S|=s$, let
\[\cG_{S}(b,h):=\{B \sub X: |[B]|=b, |N(B)|=h, N(B) \sub S\}.\]

\begin{lemma}\label{lem7}
Let $S\subseteq Y$, $1 \le \varphi \le d_Y-1$ and $b, h\in\mathbb N$. Set $w=hd_Y-bd_X$. 
For any $C>0$ that satisfies $p:=C\ln d_X/(\varphi d_X)<1$, the following holds: there exists a family $\cV = \cV(b,h,S) \sub 2^Y$ of size at most
\begin{equation*}
    \binom{s}{\le 3hp}\binom{s}{\le 3h\exp(-pm_\varphi)}\binom{s}{\le \frac{3w\ln d_Y}{d_X(d_Y-\varphi)}}\binom{3hd_Yp}{\le 3wp}
\end{equation*}
such that each $B \in \cG_{S}(b,h)$ has a $\varphi$-approximation in $\cV$.   
\end{lemma}

\subsubsection{Proof of Lemma \ref{lem7}}\label{pflem7}

For the proof of~\Cref{pflem7} we require the following two results.

For a bipartite graph with bipartition $U \cup W$ we say that $W'\sub W$ \emph{covers} $U$ if each $u \in U$ has a neighbour in $W'$.

\begin{lemma}[Lov\'asz \cite{Lov75}, Stein \cite{Ste74}]\label{lem.LS} Let $ \Sigma$ be a bipartite graph with bipartition $U \cup W$, where $|N(u)|\ge x$ for each $u \in U$ and $|N(w)|\le y$ for each $w \in W$. Then there exists some $W' \sub W$ that covers $U$ and satisfies
\[|W'|\le\frac{|W|}{ x}\cdot(1+\ln y).\]
\end{lemma}

\begin{lemma}[\!\!{\cite[Claim 1]{BalGLW21}}] For any $B \in \cG_S(b, h)$, there is a set $T_0 \sub N(B)$ such that
\[|T_0|\le 3hp,\]
\[e(T_0, X \setminus [B]) \le 3wp,\]
\[|N(B)^\varphi \setminus N(N_{[B]}(T_0))| \le 3h\exp(-pm_\varphi).\]

\end{lemma}

\begin{proof}[Proof of Lemma \ref{lem7}]
Given $B \in \cG_S(b, h)$, take $T_0$ as in the above lemma.

Define
\[T_0'=N(B)^\varphi \setminus N(N_{[B]}(T_0)), \quad L=T_0' \cup N(N_{[B]}(T_0)), \quad \Omega=\nabla(T_0, X \setminus [B]).\]
Let $T_1 \sub N(B) \setminus L$ be a minimal set that covers $[B] \setminus N(L)$ in the graph $\overline \Sigma$ induced on $[B] \setminus N(L) \cup N(B) \setminus L$. Let $F':=L \cup T_1$. Then $F'$ is a $\varphi$-approximation of $B$. Also, since $N(N_{[B]}(T_0))$ is determined by $T_0$ and $\Omega$, $F'$ is determined by $T_0$, $\Omega$, $T_0'$, and $T_1$. Note that
\[|N(B) \setminus L|(d_Y-\varphi) \le e(N(B), X \setminus [B])=w,\]
so by Lemma \ref{lem.LS} (with $U=[B] \setminus N(L)$, $x= d_X$, $W=N(B) \setminus L$ and $ y = d_Y$), we have
\[|T_1| \le \frac{|N(B) \setminus L|}{d_X}(1+\ln d_Y) \le \frac{3w\ln d_Y}{d_X(d_Y-\varphi)}.\]
Finally, using the fact that $T_0, T_0', T_1$ are subsets of $N(A) \sub S$ and $\Omega$ is a subset of $\nabla(T_0)$, the number of choices for $T_0, T_0', T_1$ and $\Omega$ is at most
\begin{equation*}
    \binom{s}{\le 3hp}\binom{s}{\le 3h\exp(-pm_\varphi)}\binom{s}{\le \frac{3w\ln d_Y}{d_X(d_Y-\varphi)}}\binom{3hpd_Y}{\le 3wp}. \qedhere
\end{equation*}
\end{proof}

\subsection{Proof of Theorem \ref{MT}}\label{sec.Thm3.5pf}
 Throughout this section we fix $r\in\{\lfloor n/2 \rfloor, \lceil n/2\rceil\}$ and $a,g,b,$ and $h$ as in the statement of \Cref{MT}. Recall our standing assumption that $g\geq(1+1/n)a$.

We write $l=n-r +1, \ l'=n-r +2, \ u=r$ and $u'=r-1$. Then $ L_{r-2} \cup L_{r-1}$  is $(l',u')$ -regular, and $ L_{r-1} \cup L_{r}$ is $(l,u)$-regular. Recall that we write
\begin{center}
$\cH_r(a,b,g,h)=$ \\ $\{(A,B)\in 2^{L_{r-1}} \times 2^{L_{r-2}}: A \text{ $2$-linked}, |[A]|=a, |[B]|=b, |N^+(A)|=g, |N^+(B)|=h, N^+(B) \subseteq A\}$
\end{center}
and set $t=g-a$ and $t'=h-b$. Our goal is to bound $\sum_{(A,B) \in \cH_r(a,b,g,h)} \lambda^{|A|+|B|-h}$ from above using the container tools from~\Cref{sec.container.prelim}. In \Cref{subsubsec.containers}, we build containers for the elements of $\cH_r(a,b,g,h)$ and we give upper bounds for the cost of these containers using \Cref{lem1} ($\varphi$-approximation), \Cref{lem2} ($(\psi_X,\psi_Y)$-approximation) and \Cref{lem7}. In \Cref{subsubsec.reconstruction}, we present our `reconstruction strategy' which bounds the contribution to the sum $\sum_{(A,B) \in \cH_r(a,b,g,h)} \lambda^{|A|+|B|-h}$ from elements $(A,B)$ with a given container. Combining these estimates will give the desired upper bound.

\subsubsection{Containers}\label{subsubsec.containers} We will always use $(A,B)$ for a member of $\cH_r(a,b,g,h)$. Let $\cA:=\{A: (A,B) \in \cH_r(a,b,g,h)\}$ and $\cB:=\{B: (A,B) \in \cH_r(a,b,g,h)\}$. We construct containers for the members of $\cA$ and $\cB$ separately. Note that if we let $\Sigma$ denote the bipartite subgraph of $B_n$ induced by the consecutive layers $X=L_{r-1}$ and $Y=L_{r}$, then 
$\cA \sub \cG(a,g)$ where $\cG(a,g)$ is as in~\eqref{eq:cGdef}. In particular, the results in \Cref{sec.container.prelim} are applicable to $\cA$. On the other hand, the elements of $\cB$ are not necessarily $2$-linked and so the results in \Cref{sec.container.prelim} are not immediately applicable. To get around this issue, we crucially take advantage of the fact that $N^+(B) \sub A$ which allows us to apply \Cref{lem7}. There are four steps for constructing containers: $\varphi$-approximations and $(\psi_X,\psi_Y)$-approximations for $\cA$ and $\cB$ (for appropriately chosen $\varphi, \psi_X$ and $\psi_Y$).

\nin \textbf{Step 1.} Bounding the cost for $\varphi$-approximations of elements of $\cA$.

We first apply Lemma \ref{lem1} with $\Sigma$ the bipartite subgraph of $B_n$ induced by the  layers $L_{r-1}$ and $L_{r}$ (with $L_{r-1}$ playing the role of $X$ and $L_r$ playing the role of $Y$, respectively) in order to obtain a collection $\cV_\cA \subseteq 2^{L_r}$ such that each $A \in \cA$ has a $\varphi$-approximation in $\cV_\cA$. Note that with this choice of $\Sigma$ we have $d_X=l, d_Y=u$.

We take $\varphi=u/2$ and let $m_\varphi$ be as in \eqref{m.varphi}. By \Cref{isoper}, we have 
\begin{equation}\label{m_varphi}
    m_\varphi\ge n\varphi/50 \ge \varphi l/50\, .
\end{equation}

 The cost of $\varphi$-approximation ($\log |\cV_\cA|$) is at most (using the fact that $l > u$, \eqref{binom} and \eqref{m_varphi})
\beq{A.varphi1} \log | L_{\lfloor n/2 \rfloor}| +O(g \log^2 l/(u l))+O(g\log l/u^{C'})+O(w\log u \log l/(u l))+O(w\log l \log(gu/w)/(u l)),\enq
where we can set (e.g.) $C'=C/50$, where $C$ is in \Cref{lem1}, using \eqref{m_varphi}.

Note that by our assumption that $a\geq n^2$ and Proposition \ref{isoper}, it suffices to consider the case $g=\Omega( n^3)$ (as otherwise $\cH_r(a,b,g,h)$ is empty), and hence $\log | L_{\lfloor n/2 \rfloor}|~(=O(n)) = o(g\log^2 l/(ul))$. Also note that by setting $C=150$ in \Cref{lem1} so that $C'=3$, the third term of~\eqref{A.varphi1} satisfies $g\log l/u^{C'} = o(g\log^2 l/(ul))$ since $u=\Theta(n)$. Therefore, \eqref{A.varphi1} simplifies to 
\beq{A.varphi2}O(g \log^2 n/n^2)+O(w\log^2 n/n^2)+O(w\log n \log(gn/w)/n^2).\enq

We claim that the above expression is $O(t\log^2 n/n)$: recall that $t=\gO(g/n)$ (by assumption), and also note that from $w=gu-al=(a+t)u-al \le tu,$ we have $t\ge w/u$. Hence we have $g\log^2n/n^2 = O(t\log^2n/n)$ and $w\log^2n/n^2 = O(t\log^2n/n).$ Now, we show that $w \log (gn/w)/n=O(t\log n)$, from which it follows that the last term of~\eqref{A.varphi2} is also $O(t\log^2 n/n)$. We have
\[\frac{w}{n}\log(gn/w) \le \frac{w}{n}\log(gn^2/w) = gn \cdot \frac{w}{gn^2}\log(gn^2/w) \stackrel{(\dagger)}{\le} gn \cdot \frac{t}{gn}\log(gn/t)=O(t\log n),\]
where $(\dagger)$ uses the facts that $x\log\left(\frac{1}{x}\right)$ is strictly increasing in $(0, 1/e)$, $t \ge w/n$, and $t/(gn) \le 1/e$ for large enough $n$. In conclusion 
\beq{cost1} \log |\cV_\cA| = O(t\log^2n/n)\, . \enq

\nin \textbf{Step 2.} Bounding the cost for $(\psi_X,\psi_Y)$-approximations of $\cA$.

Given $F' \in \cV_\cA$, by \Cref{lem2} (again with $X=L_{r-1}$, $d_X=l$, $Y=L_r$ and $d_Y=u$), there is a family $\cF_\cA=\cF_\cA(F')\subseteq 2^{L_r} \times 2^{L_{r-1}}$ that satisfies the conclusions of \Cref{lem2}, namely: 
\begin{enumerate}
    \item $\cF_\cA =\cup_{i\in[0, \kappa]} (\cF_\cA)_i,$ where by \eqref{bdcf} we have
    \[ |\cF_\cA|\leq {ug \choose \le 2\gamma w/(ul)}{ u^2 l^2 g \choose \le 2\gamma w/(ul)} \exp\left[ O\left(\frac{t\gamma \log l \log(gl^2/t)}{l}\right)\right];\]
   \item if $(F,S)\in(\cF_\cA)_i$, then $(F,S)$ is a $(l/(2^i\gamma), u/\gamma)$-approximating pair for some $A \in \cG(a,g,F')$; in addition, for $i \in [0, \kappa-1]$, every $(F, S) \in (\cF_\cA)_i$ satisfies $|F|<g-t/2^i$; and
    \item every $A \in \cG(a,g,F')$ has a  $(l/(2^i\gamma), u/\gamma)$-approximating pair in $\cF_i$ for some $i \in [0, \kappa]$.
\end{enumerate}

Therefore,
\[\log |\cF_\cA|= O(w\log (glu^2/w)/(ul)) + O(w\log(gul/w)/(ul))+ O(t\log l \log (gl^2 /t)/l). \]
Similarly to the final paragraph of Step 1, the above sum is
\beq{cost2} O(t\log^2n/n).\enq

\nin \textbf{Step 3.} Bounding the cost for $\varphi$-approximations of $\cB$

In what follows, given $F' \in \cV_\cA$ and $(F,S)\in\mathcal{F}_\cA(F')$, we construct a collection $\cV_B(S)$ such that any $B \in \cB$ with $N^+(B) \sub S$ has a $\varphi$-approximation in $\cV_B(S)$. So performing the process below for all $F' \in \cV_\cA$ and $(F,S) \in \cF_\cA(F')$ will produce $\varphi$-approximations for all $B \in \cB$.

Recall the definitions of $u'$ and $l'$ from the beginning of \Cref{sec.Thm3.5pf}. Let $w'=hu'-bl'$. To bound the cost for $\varphi$-approximations of $\cB$, we apply \Cref{lem7} with $\Sigma$ the bipartite subgraph of $B_n$ induced by the  layers $L_{r-2}$ and $L_{r-1}$ (with $L_{r-1}$ now playing the role of $Y$) and $S\subseteq L_{r-1}$ given by Step 2 above. 

Let $(F,S) \in \cF_\cA$ be given and set $\varphi=u'/2$. By Lemma \ref{lem7}, there exists a collection $\cV_B=\cV_B(S)\subseteq 2^{L_{r-1}}$ such that each $B \in \cB$ with $B \sub S$ has a $\varphi$-approximation in $\cV_B$; moreover, using the fact that $|S| \le 2g$,
\[\begin{split} \log |\cV_B| &\le  O\left(\frac{h\log l' \log (gu' l'/h)}{u' l'}\right)+O\left(\frac{w'\log u'}{l'u'}\log\left(\frac{hl'u'}{w'\log u'}\right)\right)+O\left(\frac{w'\log l'\log(hu'/w')}{u' l'}\right) \\
&= O\left(\frac{g\log^2 n}{n^2}\right)+O\left(\frac{w'\log n}{n^2}\log\left(\frac{gn^2}{w'}\right)\right)+O\left(\frac{w'\log n}{n^2}\log\left( \frac{hn}{w'}\right)\right).\end{split}\]

Recall that $t=g-a$ and $t'=h-b$. Write $t^{max}=\max\{t, t'\}$. We claim that the expression above is $O(t^{max}\log^2 n/n)$. By a similar analysis to that in Step 1, using the facts that $t'\geq w'/u'$ and $t'=\Omega(h/n)$ (by \Cref{isoper}), we see that the first and third terms are both $O(t^{max}\log^2n/n)$. For the second term, we have
\[ \frac{w' \log n}{n^2} \log\left( \frac{gn^2}{w'} \right) \leq \frac{t^{max}\log n}{n} \log \left( \frac{gn}{t^{max}} \right) \leq \frac{t^{max}\log n}{n} \log \left( \frac{gn}{t} \right) = O(t^{max}\log^2 n/n).\]
In conclusion,
\beq{cost3} \log|\cV_B|=O(t^{max}\log^2n/n). \enq

\nin \textbf{Step 4.} Bounding the cost for $(\psi_X,\psi_Y)$-approximations of $\cB$.

Let $P' \in \cV_\cB$ be given. As in Step 2, there is a family $\cF_\cB=\cF_\cB(b, h, P')\subseteq 2^{L_{r-1}}\times 2^{L_{r-2}}$ that satisfies the conclusions of \Cref{lem2} with the appropriate choice of parameters. Namely, writing 
\[\kappa':=\max\{j:2^j \le \sqrt{l'}\}:\]
\begin{enumerate}
    \item the size of $\cF_\cB$ satisfies (cf. \eqref{cost2})
    \beq{cost4}\log |\cF_\cB|=O\left({t' \log^2 n}/{n}\right); \enq

    \item if $(P,Q)\in(\cF_\cB)_j$, then $(P,Q)$ is a $(l'/(2^j\gamma), u'/\gamma)$-approximating pair for some $B \in \cG(b,h,P')$; in addition, for $j \in [0, \kappa'-1]$, every $(P, Q) \in (\cF_\cB)_j$ satisfies $|P|<h-t'/2^j$; and
    \item every $B \in \cG(b, h, P')$ has a  $(l'/(2^j\gamma), u'/\gamma)$-approximating pair in $(\cF_\cB)_j$ for some $j \in [0, \kappa']$.
\end{enumerate}

In sum, putting together the costs \eqref{cost1}, \eqref{cost2}, \eqref{cost3} and \eqref{cost4} bounds the total cost of our approximations for elements $(A,B) \in \cH_r(a,b,g,h)$ by
\beq{total.cost} O(t^{max}\log^2n/n).\enq

\subsubsection{Bounding $\sum_{(A,B) \in \cH_r(a,b,g,h)} \gl^{|A|+|B|-h}$ using containers}\label{subsubsec.reconstruction}

We always write $(F,S)$ for an approximating pair for $A\in\cA$  and $(P, Q)$ for an approximating pair for $B\in\cB$. Let $\gamma$ be the absolute constant from \Cref{lem2}. Recall that we write (\textit{cf.} \eqref{kap}) $\kappa=\max\{i:2^i \le \sqrt l\}$ and $\kappa'=\max\{j:2^j \le \sqrt{l'}\}$. 

As we sketched at the beginning of \Cref{secproofcont}, we apply different reconstruction strategies depending on the interplay of the quadruple $(F, S, P, Q)$. To this end, we define $F$ to be \textit{tight} if $|F| \ge g-t/2^\kappa~ (\ge g-2t/\sqrt l)$, and \textit{slack} otherwise.
Similarly, say $P$ is \textit{tight} if $|P|\ge h-t'/2^{\kappa'}~(\ge h-2t'/\sqrt{l'})$, and \textit{slack} otherwise.

We note that if $(F,S)\in \cF_\cA(F')$ for some $F'\in \cV_\cA$ then 
\beq{s.t.bd} |S|-a\leq 2t. \enq
Indeed, by definition $(F,S)$ is a $(l/(2^i\gamma), u/\gamma)$-approximating pair for some $A\in \cG(a,g)$ and some $i\in[0,\kappa] $. By \Cref{SF.size} and the inequality $u\leq l$, we have (writing $f=|F|$ and $s=|S|$) that $s \le f+(t-f+s)/\gamma,$ which yields $s \le f+t/(\gamma-1)$. So (recalling that $\gamma\geq 2$), we have $s-a \le f-a+t \le g-a+t=2t.$ 

Similarly, if $(P,Q)\in \cF_\cB(P')$ for some $P'\in \cV_\cB$ then
\beq{P.t.bd} |Q|-b\leq 2t'. \enq

We write $(A,B) \sim_{i,j}(F, S, P, Q)$ if $(F,S)\in 2^{L_r} \times 2^{L_{r-1}}$ is a $(l/(2^i\gamma), u/\gamma)$-approximating pair for $A$ and $(P,Q)\in 2^{L_r-1} \times 2^{L_{r-2}}$ is a $(l'/(2^j\gamma), u'/\gamma)$-approximating pair for $B$.

Our goal is to establish the inequality below for any given $(i,j) \in [0, \kappa] \times [0, \kappa']$ and $(F,S,P,Q)$:
\begin{align}\label{eq:FSPQgoal}
\sum_{(A,B) \sim_{i,j}(F, S, P, Q)} \gl^{|A|+|B|-h} \le (1+\gl)^{g-\gO(t^{max}/\sqrt n)}.
\end{align}
 This, combined with \eqref{total.cost} (the total cost for $(F,S,P,Q)$ and $i,j \in [0,\kappa] \times [0, \kappa']$) and the assumption that $\gl\ge C\log^2 n/\sqrt n$ (and that we can choose a large enough $C$), will yield the conclusion of \Cref{MT}. 

We consider four cases to bound the sum in~\eqref{eq:FSPQgoal}. In what follows, we use $G$ for $N^+(A)$ and $H$ for $N^+(B)$ for simplicity. We also use lower case letters to denote the cardinality of sets denoted by upper case letters, e.g. $p:=|P|, q:=|Q|$, etc.

\nin \textbf{Case 1.} $F, P$ slack

In this case, we have 
\beq{F.slack} g-f>t/2^\kappa \ge t/\sqrt l,\enq
and
\beq{Q.slack} h-p>t'/2^{\kappa'} \ge t'/\sqrt{l'}.\enq

\nin By \Cref{SF.size} and the fact that $u\leq l$, we have that
\[s \leq f+\frac{1}{l}[(g-f)l/\gamma+(s-a)2^{-i}l/\gamma],\]
so
\beq{g.s.diff}\begin{split} g-s &\geq g-f-(g-f)/\gamma-(s-a)2^{-i}/\gamma \stackrel{\eqref{F.slack}}{\geq} (g-f)(1-1/\gamma)-(s-a)(g-f)/(t\gamma) \\ &\stackrel{\eqref{s.t.bd}}{\geq}(g-f)(1-1/\gamma-2/\gamma)=\gO(g-f)\stackrel{\eqref{F.slack}}{=} \Omega(t/\sqrt{n}).\end{split}
\enq
Similarly, we also have 
\beq{h.s.diff} h-q= \Omega( t'/\sqrt{n}).\enq

Now, we specify $B$ as a subset of $Q$. This specifies $H=N^{+}(B)$, and then we specify $A \setminus H$ as a subset of $S \setminus H$. Therefore,
\begin{align*}
\sum_{(A,B) \sim_{i,j}(F, S, P, Q)} \lambda^{|A|+|B|-h}
\le \sum_{B\subseteq Q}\lam^{|B|}\sum_{A\backslash H \subseteq S\backslash H} \lam^{|A\backslash H|}
\le (1+\lambda)^{q}(1+\lambda)^{s-h}~ &=(1+\lambda)^g (1+\lam)^{-(g-s)-(h-q)}\\&\leq (1+\lam)^{g-\gO(t^{max}/\sqrt n)}.
\end{align*}

\nin \textbf{Case 2.} $F$ tight, $P$ slack

We first specify $G$ by specifying $G \setminus F$ as subset of $N(S)$. Note that, by \Cref{SF.size}, $|N(S)|\le |S|l \leq 2gl$. Also, by the assumption that $F$ is tight, $|G\setminus F|\leq t/2^\kappa \le  2t/\sqrt{l}$, so the cost for $G$ is
\beq{cost.G}\log\binom{2gl}{\leq 2t/\sqrt{l}} = O(t\log n/\sqrt{n}). \enq
The specification of $G$ gives $[A]$ for free (but not $A$).

Next, we specify $B$ as a subset of $Q$, then specify $A \setminus H$ as a subset of $[A] \setminus H$. Noting that $|Q|=q$ and $[A] \setminus H=a-h$ we therefore have
\begin{align*}
\sum_{(A,B) \sim_{i,j}(F, S, P, Q)} \lambda^{|A|+|B|-h} &\le \exp\left[O(t\log n/\sqrt{n})\right](1+\lambda)^{q}(1+\lambda)^{a-h}~ \\
& = \exp\left[O(t\log n/\sqrt{n})\right] (1+\lambda)^g (1+\lam)^{-t-(h-q)} \\ &\leq (1+\lam)^{g-\gO(t^{max}/\sqrt n)},
\end{align*}
where the last inequality uses the facts that $t\log n/\sqrt n \ll t\gl$ and \eqref{h.s.diff}.

\nin \textbf{Case 3.} $F$ slack, $P$ tight

We first specify $H$ by specifying $H \setminus P$ as a subset of $N(Q)$. Note that again by \Cref{SF.size}, we have $|N(Q)|\le |Q|l'\leq 2hl'$ and $|H\setminus P| \leq 2t'/\sqrt{l'}$ (since $P$ is tight), and so the cost for $H$ is
\beq{cost.H}\log\binom{2hl'}{2t'/\sqrt{l'}} = O(t'\log n/\sqrt{n}). \enq

Then (with $[B]$ given by $H$) specify $B$ as a subset of $[B]$. Finally, specify $A \setminus H$ as a subset of $S \setminus H$. Therefore,
\begin{align*}
\sum_{(A,B) \sim_{i,j}(F, S, P, Q)} \lambda^{|A|+|B|-h} &\le \exp\left[O(t'\log n/\sqrt{n})\right] (1+\lambda)^{b}(1+\lambda)^{s-h}\\
& \leq \exp\left[O(t'\log n/\sqrt{n})\right] (1+\lambda)^g (1+\lam)^{-(g-s)-t'}\\ &\leq (1+\lam)^{g-\gO(t^{max}/\sqrt n)},
\end{align*}
where the last inequality uses the facts that $t'\log n/\sqrt n \ll t'\gl$ and \eqref{g.s.diff}.

\nin \textbf{Case 4.} $F, P$ tight

We first specify $H$ by specifying $H \setminus P$ as a subset of $N(Q)$ whose cost is as in \eqref{cost.H}, then specify $B$ as a subset of $[B]$. Next, we specify $G$ by specifying $G \setminus F$ as a subset of $N(S)$ which costs \eqref{cost.G}, then specify $A \setminus H$ as a subset of $[A]\setminus H$. Thus,
\begin{align*}
\sum_{(A,B) \sim_{i,j}(F, S, P, Q)} \lambda^{|A|+|B|-h} & \le  \exp\left[O((t+t')\log n/\sqrt{n})\right](1+\lambda)^{b+a-h} \\
& \le (1+\gl)^{g-\gO(t^{max})},
\end{align*}
where the last inequality uses that $b+a-h=(g-t)-(h-b)=g-(t+t')$ and $\gl t^{max} \gg t\log n/\sqrt n$. This establishes~\eqref{eq:FSPQgoal} as desired. 

\section{Concluding remarks}
\label{conclusion}
In this paper we proved several results on the number and typical structure of antichains in $B_n$. 
Some of our results (e.g. \Cref{sparsesperner3}) are limited by our application of the container tools in the previous section. We believe that the following strengthening of \Cref{sparsesperner3} holds: 

\begin{conjecture}\label{conj1}
There exists a function $\beta^*(n)=\tilde \Theta(1/n)$ such that if 
$\beta\geq \beta^*(n)$, then almost all antichains of size $\beta N$ in $B_n$ are contained in three central layers.    
\end{conjecture}

 Conjecture~\ref{conj1} would follow if one could prove a version of \Cref{MT} that holds for $\lam=\tilde \Omega(1/n)$.  Moreover for $\beta$ in the relevant range, a heuristic calculation suggests that the probability a vertex $v\in L_{n/2-2}\cup L_{n/2+2}$ (assuming $n$ is even for simplicity) is contained in a uniformly random antichain of size $\beta N$ is approximately $\beta(1-\beta)^{n^2/4}$. The factor of $(1-\beta)^{n^2/4}$ comes from the fact that $v$ blocks approximately $n^2/4$ vertices in the middle layer $L_{n/2}$. $\beta=\Theta(1/n)$ is therefore the threshold at which one might expect to see vertices in layers $L_{n/2-2}\cup L_{n/2+2}$ in a uniformly random antichain of size $\beta N$.

In fact, we conjecture the following sequence of thresholds generalising \Cref{conj1}.

\begin{conjecture}
    \label{conj2}
    Given $k\in \mathbb{N}$, there exists a function $\beta^*(n)=\tilde \Theta(1/n^k)$ such that if 
$\beta\geq \beta^*(n)$, then almost all antichains of size $\beta N$ in $B_n$ are contained in $2k+1$ central layers.
\end{conjecture}

A somewhat different direction of research would be to study the number and typical structure of antichains of size $\beta N$ in the (vertex) percolated lattice $B_{n,p}$ for some $p\in(0, 1)$. 
 Such questions might help to elucidate the connection between the sparsification of Dedekind's problem/ Sperner's theorem considered here and those considered in  \cite{BalK22, collares2016maximum, embed, kohayakawa2002width, BalMT14, osthus2000maximum, renyi1961random}. In a similar spirit, Kronenberg and Spinka \cite{kronenberg2022independent} study the hard-core model on the (edge) percolated hypercube $Q_d$ and it would be interesting to see if parts of their approach could be adapted to this context.

\section*{Acknowledgements}
The authors would like to thank Mauricio Collares for writing the computer code discussed in the Appendix and for spotting mistakes in a previous draft of this paper. MJ would like to thank Peter Keevash for many valuable discussions. MJ is supported by a UK Research and Innovation Future Leaders Fellowship MR/W007320/2. JP is supported by NSF grant DMS-2324978 and a Sloan Fellowship.

\newpage
\bibliographystyle{plain}
\bibliography{bibliography}

\begin{thebibliography}{10}

\bibitem{alon2014counting}
N.~Alon, J.~Balogh, R.~Morris, and W.~Samotij.
\newblock Counting sum-free sets in abelian groups.
\newblock {\em Israel Journal of Mathematics}, 199(1):309--344, 2014.

\bibitem{alon2014refinement}
N.~Alon, J.~Balogh, R.~Morris, and W.~Samotij.
\newblock A refinement of the {C}ameron--{E}rd{\H{o}}s conjecture.
\newblock {\em Proceedings of the London Mathematical Society}, 108(1):44--72,
  2014.

\bibitem{balogh2021independent}
J.~Balogh, R.~I. Garcia, and L.~Li.
\newblock Independent sets in the middle two layers of {B}oolean lattice.
\newblock {\em Journal of Combinatorial Theory, Series A}, 178:105341, 2021.

\bibitem{BalGLW21}
J.~Balogh, R.~I. Garcia, L.~Li, and A.~\relax{Zs}. Wagner.
\newblock Intersecting families of sets are typically trivial.
\newblock {\em Journal of Combinatorial Theory, Series B}, 164:44--67, 2021.

\bibitem{BalK22}
J.~Balogh and R.~A. Krueger.
\newblock A sharp threshold for a random version of {S}perner's {T}heorem.
\newblock {\em Random Structures \& Algorithms}, to appear.

\bibitem{balogh2015independent}
J.~Balogh, R.~Morris, and W.~Samotij.
\newblock Independent sets in hypergraphs.
\newblock {\em Journal of the American Mathematical Society}, 28(3):669--709,
  2015.

\bibitem{balogh2016typical}
J.~Balogh, R.~Morris, W.~Samotij, and L.~Warnke.
\newblock The typical structure of sparse ${K}_{r+1}$-free graphs.
\newblock {\em Transactions of the American Mathematical Society},
  368(9):6439--6485, 2016.

\bibitem{BalMT14}
J.~Balogh, R.~Myroft, and A.~Treglown.
\newblock A sparse version of {S}perner's theorem.
\newblock {\em Journal of Combiatorial Theory, Series A}, 128:104--110, 2014.

\bibitem{balogh2016applications}
J.~Balogh, A.~Treglown, and A.~\relax{Zs}. Wagner.
\newblock Applications of graph containers in the boolean lattice.
\newblock {\em Random Structures \& Algorithms}, 49(4):845--872, 2016.

\bibitem{barvinok2015computing}
A.~Barvinok.
\newblock Computing the partition function for cliques in a graph.
\newblock {\em Theory of Computing}, 11(13):339--355, 2015.

\bibitem{barvinok2017combinatorics}
A.~Barvinok.
\newblock {\em Combinatorics and Complexity of Partition Functions}, volume~30
  of {\em Algorithms and Combinatorics}.
\newblock Springer, 2017.

\bibitem{ursellcomp}
A.~Bj\"orklund, T.~Husfeldt, P.~Kaski, and M.~Koivisto.
\newblock Computing the {T}utte polynomial in vertex-exponential time.
\newblock In {\em Proceedings of the Forty-ninth Annual Symposium on
  Foundations of Computer Science, FOCS}, pages 677--686, 2008.

\bibitem{brydges1984short}
D.~C. Brydges.
\newblock A short course on cluster expansions.
\newblock {\em Les Houches}, PART I, 1984.

\bibitem{mauricio}
M.~Collares, J.~Erde, A.~Geisler, and M.~Kang.
\newblock Counting independent sets in bipartite regular product graphs.
\newblock {\em In preparation}.

\bibitem{collares2016maximum}
M.~Collares and R.~Morris.
\newblock Maximum-size antichains in random set-systems.
\newblock {\em Random Structures \& Algorithms}, 49(2):308--321, 2016.

\bibitem{conlon2014combinatorial}
D.~Conlon.
\newblock Combinatorial theorems relative to a random set.
\newblock {\em Preprint, arXiv:1404.3324}, 2014.

\bibitem{ConG16}
D.~Conlon and W.~T. Gowers.
\newblock Combinatorial theorems in sparse random sets.
\newblock {\em Annals of Mathematics}, 184(2):367--454, 2016.

\bibitem{dedekind}
R.~Dedekind.
\newblock Ueber zerlegungen von zahlen durch ihre gr\"{o}ssten gemeinsamen
  teiler.
\newblock In {\em Festschrift Hoch. Braunschweig u. ges. Werke, II}, pages
  103--148. 1897.

\bibitem{dobrushin}
R.~L. Dobrushin.
\newblock Estimates of semi-invariants for the {I}sing model at low
  temperatures.
\newblock {\em American Mathematical Society Translations: Series 2},
  177:59--82, 1996.

\bibitem{engelberg2021typical}
O.~Engelberg, W.~Samotij, and L.~Warnke.
\newblock On the typical structure of graphs not containing a fixed
  vertex-critical subgraph.
\newblock {\em Random Structures \& Algorithms}, 65:644–702, 2024.

\bibitem{FRT23}
V.~Falgas-Ravry, E.~R\"{a}ty, and I.~Tomon.
\newblock Dedekind's problem in the hypergrid.
\newblock {\em Preprint, arXiv:2310.12946}, 2023.

\bibitem{Galvin_Zhom}
D.~Galvin.
\newblock On homomorphisms from the {H}amming cube to $\mathbb{Z}$.
\newblock {\em Israel Journal of Mathematics}, 138:189--213, 2003.

\bibitem{galvin}
D.~Galvin.
\newblock A threshold phenomenon for random independent sets in the discrete
  hypercube.
\newblock {\em Combinatorics, Probabability and Computing}, 20(1):27--51, 2011.

\bibitem{Galvin_independent}
D.~Galvin.
\newblock Independent sets in the discrete hypercube.
\newblock {\em Preprint, arXiv:1901.01991}, 2019.

\bibitem{GalK04}
D.~Galvin and J.~Kahn.
\newblock On phase transition in the hard-core model in $\mathbb{Z}^d$.
\newblock {\em Combininatorics, Probability and Computing}, 13(2):137--164,
  2004.

\bibitem{gilbert1954lattice}
E.~N. Gilbert.
\newblock Lattice theoretic properties of frontal switching functions.
\newblock {\em Journal of Mathematics and Physics}, 33(1-4):57--67, 1954.

\bibitem{embed}
A.~Hamm and J.~Kahn.
\newblock On {E}rd{\H{o}}s-{K}o-{R}ado for random hypergraphs {II}.
\newblock {\em Combinatorics, Probabability and Computing}, 28(1):61--80, 2019.

\bibitem{hansel1966nombre}
G.~Hansel.
\newblock Sur le nombre des fonctions bool{\'e}ennes monotones de n variables.
\newblock {\em Comptes rendus de l'Acad\'emie des Sciences de Paris},
  262(20):1088--1090, 1966.

\bibitem{HelmuthAlgorithmic2}
T.~Helmuth, W.~Perkins, and G.~Regts.
\newblock Algorithmic {P}irogov-{S}inai theory.
\newblock {\em Probability Theory and Related Fields}, 176:851--895, 2020.

\bibitem{1dedekind9}
C.~J\"{a}kel.
\newblock A computation of the ninth {D}edekind {N}umber.
\newblock {\em Preprint, arxiv:2304.00895}, 2023.

\bibitem{janson1988normal}
S.~Janson.
\newblock Normal convergence by higher semiinvariants with applications to sums
  of dependent random variables and random graphs.
\newblock {\em The Annals of Probability}, pages 305--312, 1988.

\bibitem{jenssen2023homomorphisms}
M.~Jenssen and P.~Keevash.
\newblock Homomorphisms from the torus.
\newblock {\em Advances in Mathematics}, 430:109212, 2023.

\bibitem{JKP2}
M.~Jenssen, P.~Keevash, and W.~Perkins.
\newblock Algorithms for \#{BIS}-hard problems on expander graphs.
\newblock {\em SIAM Journal on Computing}, 49(4):681--710, 2020.

\bibitem{lambda}
M.~Jenssen, A.~Malekshahian, and J.~Park.
\newblock A refined graph container lemma and applications to the hard-core
  model on bipartite expanders.
\newblock {\em Preprint, arXiv:2411.03393}, 2024.

\bibitem{hypercube}
M.~Jenssen and W.~Perkins.
\newblock Independent sets in the hypercube revisited.
\newblock {\em Journal of the London Mathematical Society}, 102:645--669, 2020.

\bibitem{struct}
M.~Jenssen, W.~Perkins, and A.~Potukuchi.
\newblock Independent sets of a given size and structure in the hypercube.
\newblock {\em Combinatorics, Probabability and Computing}, 31(4):702--720,
  2022.

\bibitem{jenssen2023evolution}
M.~Jenssen, W.~Perkins, and A.~Potukuchi.
\newblock On the evolution of structure in triangle-free graphs.
\newblock {\em Preprint, arXiv:2312.09202}, 2023.

\bibitem{jenssen2024sampling}
M.~Jenssen, W.~Perkins, A.~Potukuchi, and M.~Simkin.
\newblock Sampling and counting triangle-free graphs near the critical density.
\newblock {\em arXiv preprint arXiv:2410.22951}, 2024.

\bibitem{kahn2002entropy}
J.~Kahn.
\newblock Entropy, independent sets and antichains: a new approach to
  {D}edekind’s problem.
\newblock {\em Proceedings of the American Mathematical Society},
  130(2):371--378, 2002.

\bibitem{katona}
Gy. O.~H. Katona.
\newblock A theorem of finite sets.
\newblock In P.~Erd\H{o}s and Gy. O.~H. Katona, editors, {\em Theory of Graphs:
  proceedings of the colloquim held at Tihany, Hungary, September 1966}, pages
  187--207. Akademiai Kiado, 1966.

\bibitem{kleitman}
D.~Kleitman.
\newblock On {D}edekind’s problem: The number of monotone {B}oolean
  functions.
\newblock {\em Proceedings of the American Mathematical Society}, 21:677--682,
  1969.

\bibitem{km}
D.~Kleitman and G.~Markowsky.
\newblock On {D}edekind’s problem: the number of isotone {B}oolean functions.
  {II}.
\newblock {\em Transactions of the American Mathematical Society},
  213:373--390, 1975.

\bibitem{kleitman1982number}
D.~J. Kleitman and K.~J. Winston.
\newblock On the number of graphs without 4-cycles.
\newblock {\em Discrete Mathematics}, 41(2):167--172, 1982.

\bibitem{kohayakawa2002width}
Y.~Kohayakawa and B.~Kreuter.
\newblock The width of random subsets of {B}oolean lattices.
\newblock {\em Journal of Combinatorial Theory, Series A}, 100(2):376--386,
  2002.

\bibitem{korobkov1965monotone}
V.~K. Korobkov.
\newblock On monotone functions of the algebra of logic.
\newblock {\em Problemy Kibernetiki}, 13:5--28, 1965.

\bibitem{korshunov}
A.~D. Korshunov.
\newblock The number of monotone {B}oolean functions (in {R}ussian).
\newblock {\em Problemy Kibernetiki}, 31:5--108, 1981.

\bibitem{korshunov1983number}
A.~D. Korshunov and A.~A. Sapozhenko.
\newblock The number of binary codes with distance 2.
\newblock {\em Problemy Kibernetiki}, 40:111--130, 1983.

\bibitem{korshunov2002distribution}
A.~D. Korshunov and I.~Shmulevich.
\newblock On the distribution of the number of monotone {B}oolean functions
  relative to the number of lower units.
\newblock {\em Discrete mathematics}, 257(2-3):463--479, 2002.

\bibitem{kp}
R.~Koteck\'y and D.~Preiss.
\newblock Cluster expansion for abstract polymer models.
\newblock {\em Communications in Mathematical Physics}, 103:491--498, 1986.

\bibitem{kronenberg2022independent}
G.~Kronenberg and Y.~Spinka.
\newblock Independent sets in random subgraphs of the hypercube.
\newblock {\em Preprint, arXiv:2201.06127}, 2022.

\bibitem{kruskal}
J.~B. Kruskal.
\newblock The number of simplices in a complex.
\newblock In R.~Bellman, editor, {\em Mathematical optimization techniques},
  pages 251--278. Univ. of California Press, Berkeley, California, 1963.

\bibitem{li2023number}
L.~Li, G.~McKinley, and J.~Park.
\newblock The number of colorings of the middle layers of the {H}amming cube.
\newblock {\em Preprint, arXiv:2304.03203}, 2023.

\bibitem{Lov75}
L.~Lov\'asz.
\newblock On the ratio of optimal integral and fractional covers.
\newblock {\em Discrete Mathematics}, 13(4):383--390, 1975.

\bibitem{luczak2000triangle}
T.~{\L}uczak.
\newblock On triangle-free random graphs.
\newblock {\em Random Structures \& Algorithms}, 16(3):260--276, 2000.

\bibitem{morris2024asymmetric}
R.~Morris, W.~Samotij, and D.~Saxton.
\newblock An asymmetric container lemma and the structure of graphs with no
  induced $4 $-cycle.
\newblock {\em Journal of the European Mathematical Society}, 2024.

\bibitem{osthus2000maximum}
D.~Osthus.
\newblock Maximum antichains in random subsets of a finite set.
\newblock {\em Journal of Combinatorial Theory, Series A}, 90(2):336--346,
  2000.

\bibitem{OstPT03}
D.~Osthus, H.~J. Pr\"{o}mel, and A.~Taraz.
\newblock For {W}hich {D}ensities are {R}andom {T}riangle-{F}ree {G}raphs
  {A}lmost {S}urely {B}ipartite?
\newblock {\em Combinatorica}, 23:105--150, 2003.

\bibitem{PST23}
J.~Park, M.~Sarantis, and P.~Tetali.
\newblock Note on the number of antichains in generalizations of the {B}oolean
  lattice.
\newblock {\em Combinatorial Theory}, to appear.

\bibitem{patel2016deterministic}
V.~Patel and G.~Regts.
\newblock Deterministic polynomial-time approximation algorithms for partition
  functions and graph polynomials.
\newblock {\em SIAM Journal on Computing}, 46(6):1893--1919, 2017.

\bibitem{penrose1963convergence}
O.~Penrose.
\newblock Convergence of fugacity expansions for fluids and lattice gases.
\newblock {\em Journal of Mathematical Physics}, 4(10):1312--1320, 1963.

\bibitem{pippenger1999entropy}
H.~Pippenger.
\newblock Entropy and enumeration of {B}oolean functions.
\newblock {\em IEEE Transactions on Information Theory}, 45(6):2096--2100,
  1999.

\bibitem{regts2015zero}
G.~Regts.
\newblock Zero-free regions of partition functions with applications to
  algorithms and graph limits.
\newblock {\em Combinatorica}, pages 1--29, 2015.

\bibitem{renyi1961random}
A.~R{\'e}nyi.
\newblock On random subsets of a finite set.
\newblock {\em Mathematica (Cluj)}, 3(26):355--362, 1961.

\bibitem{ruelle1963correlation}
D.~Ruelle.
\newblock Correlation functions of classical gases.
\newblock {\em Annals of Physics}, 25(1):109--120, 1963.

\bibitem{sapozhenko1987number}
A.~A. Sapozhenko.
\newblock On the number of connected subsets with given cardinality of the
  boundary in bipartite graphs.
\newblock {\em Metody Diskretnovo Analiza}, 45(45):42--70, 1987.

\bibitem{sapo}
A.~A. Sapozhenko.
\newblock On the number of antichains in multilevelled ranked posets.
\newblock {\em Discrete Mathematics and Applications}, 1(2):149--169, 1991.

\bibitem{saxton2015hypergraph}
D.~Saxton and A.~Thomason.
\newblock Hypergraph containers.
\newblock {\em Inventiones mathematicae}, 201(3):925--992, 2015.

\bibitem{Sch16}
M.~Schacht.
\newblock Extremal results for random discrete structures.
\newblock {\em Annals of Mathematics}, 184(2):333--365, 2016.

\bibitem{ScoS05}
A.~D. Scott and A.~D. Sokal.
\newblock The repulsive lattice gas, the independent-set polynomial, and the
  {L}ov\'{a}sz local lemma.
\newblock {\em Journal of Statistical Physics}, 118(5-6):1151--1261, 2005.

\bibitem{sperner}
E.~Sperner.
\newblock Ein {S}atz \"{u}ber {U}ntermengen einer endlichen {M}enge (in
  {G}erman).
\newblock {\em Mathematische Zeitschrift}, 27(1):544--548, 1928.

\bibitem{Ste74}
S.~K. Stein.
\newblock Two combinatorial covering theorems.
\newblock {\em Journal of Combinatorial Theory, Series A}, 16:391--397, 1974.

\bibitem{2dedekind9}
L.~{Van Hirtum}, P.~{De Causmaecker}, J.~{Goemaere}, T.~{Kenter}, H.~{Riebler},
  M.~{Lass}, and C.~{Plessl}.
\newblock A computation of {D}(9) using {FPGA} {S}upercomputing.
\newblock {\em Preprint, arxiv:2304.03039}, 2023.

\bibitem{wiedemann1991computation}
D.~Wiedemann.
\newblock A computation of the eighth {D}edekind number.
\newblock {\em Order}, 8:5--6, 1991.

\end{thebibliography}

\appendix
\setcounter{secnumdepth}{0}
\section{Appendix. Computing the functions $P_j$ and $R_j$}\label{code}

As mentioned in the statements of \Cref{thm5,saporefined2}, the coefficients of each function $P_j^r$ and $R_j^r$ ($r\in\{0,1,2\}$) can be computed in time $e^{O(j\log j)}$, independently of $n$. Computer code written by Mauricio Collares which implements these computations (and the computations of the rational functions $S_j^r(n, \lam)$ from \Cref{clusterexpansion}, $F_j^r(n, \lam)$ from \Cref{cor14} and $B_j^r(n, \beta)$) from \Cref{lem17}) can be found in the public GitHub repository \url{https://github.com/collares/cluster-coefficients}. Further details on the above implementation will be provided in \cite{mauricio}.

Here we give just a few more values obtained by running the above code. We have

\[ P_3^0 = \frac{1}{512} \, n^{4} - \frac{1}{384} \, n^{3} - \frac{5}{128} \, n^{2} + \frac{7}{96} \, n + \frac{1}{3}, \]
\[     P^1_3 = \frac{11}{512} \, n^{4} - \frac{1}{192} \, n^{3} - \frac{181}{768} \, n^{2} + \frac{1}{192} \, n + \frac{841}{1536},\]

\[ P^2_3 = -\frac{1}{256} \, n^{4} - \frac{1}{96} \, n^{3} + \frac{17}{384} \, n^{2} + \frac{9}{32} \, n + \frac{401}{768},\]

\[  P_4^0 = \frac{1}{6144} \, n^{6} + \frac{1}{3072} \, n^{4} - \frac{25}{1536} \, n^{3} + \frac{3}{128} \, n^{2} + \frac{1}{96} \, n - \frac{1}{4},\]

\[P^1_4 = \frac{225}{32768} \, n^{6} - \frac{39}{4096} \, n^{5} - \frac{9113}{98304} \, n^{4} - \frac{803}{6144} \, n^{3} + \frac{62537}{98304} \, n^{2} + \frac{1723}{12288} \, n - \frac{26225}{32768}\]
and

\[P^2_4 = -\frac{45}{32768} \, n^{6} + \frac{9}{8192} \, n^{5} + \frac{4153}{98304} \, n^{4} + \frac{233}{12288} \, n^{3} - \frac{27709}{98304} \, n^{2} - \frac{10477}{24576} \, n - \frac{13095}{32768}.\]

For the $R_j$'s, we give

\[R_2^0 = -\frac{n^2(n+2)(n+6) \beta^{3} + n(n+2)(n^2-14n+8) \beta^{2} + 16 \, n^{2} \beta}{8 (n+2)^2 (1-\beta)},\]

\begin{align*}
R_2^1 = \frac{1}{8(n+3)^2(\beta-1)^3} \times \bigg[ 2 ((n-1)(n+1)(n+3) \beta^{5} - (n-1)(n+3)(n^2+12n-1) \beta^{4} + \\ 4(n-1)(5n^2+18n+5) \beta^{3} +(n+1)(n^3-11n^2-53n+31)\beta^{2} + 16(n+1)^2 \beta\bigg],
\end{align*}

\begin{align*}
R_3^0 = \frac{1}{192 (n+2)^3(\beta-1)^3} \times \bigg[ n^2(n+2)^2(3n^3+28n^2+132n+112) \beta^{6} + 3 n^2(n+2)^2(n^3-28n^2-212n+16) \beta^{5} - \\ n(n+2)^2(3n^4+36n^3-1308n^2+656n-128) \beta^{4} -  n{\left(3 n^{6} - 72 \, n^{5} + 384 \, n^{4} + 3504 \, n^{3} + 2640 \, n^{2} - 3136 \, n^{1} + 512\right)} \beta^{3} -\\ 96 \, n^2(n+2)(n^2-14n+8) \beta^{2}-512 \, n^{3} \beta  \bigg],
\end{align*}
and

\begin{align*}
    R_3^1 = -\frac{1}{384 (n+3)^3(\beta-1)^6} \times \bigg[8 (n-1)(n+1)(n+3)^2(n^2-6n-19) \beta^{9} + \\ 3 (n-1)(n+1)(n+3)^2(n^3+3n^2+167n+117) \beta^{8} - \\ 2 (n-1)(n+3)^2(3n^4+90n^3+1068n^2+566n+225) \beta^{7} - \\ 2 (n-1)(3n^6-118n^5-3227n^4-16452n^3-24739n^2-4166n-2501) \beta^{6} + \\ 12 (n-1)(n^6+10n^5-497n^4-3564n^3-6409n^2-1310n-7) \beta^{5} + \\ {\left(3 \, n^{7} - 309 \, n^{6} + 2499 \, n^{5} + 30275 \, n^{4} + 51089 \, n^{3} - 37383 \, n^{2} - 39255 \, n + 1273\right)} \beta^{4} -\\ 2 (n+1)(3n^6-66n^5+81n^4+6180n^3+13269n^2-10018n-3305) \beta^{3} -\\  192 (n+1)^2(n^3-11n^2-45n+23)\beta^{2} - 1024 (n+1)^3 \beta \bigg].
\end{align*}
\printindex

\end{document}